\documentclass[aap]{imsart}

\RequirePackage{amsthm,amsmath,amsfonts,amssymb}
\RequirePackage[numbers]{natbib}
\RequirePackage[colorlinks,citecolor=blue,urlcolor=blue]{hyperref}
\RequirePackage{enumitem}

\startlocaldefs
\theoremstyle{remark}
\newtheorem{hypothesis}{Hypothesis}

\newtheorem{theorem}{Theorem}[section]
\newtheorem{proposition}[theorem]{Proposition}
\newtheorem{lemma}[theorem]{Lemma}
\newtheorem{corollary}[theorem]{Corollary}
\newtheorem{definition}[theorem]{Definition}
\newtheorem{remark}[theorem]{Remark}

\newcommand\rrd{\mathbb{R}^d}
\newcommand\rr{\mathbb{R}}
\newcommand\ssd{\mathbb{S}^{d-1}}
\newcommand\h{\hspace{0.1cm}}
\newcommand\hs{\hspace{1cm}}
\newcommand\cp{\mathcal{P}}
\newcommand\PP{\mathbb{P}}
\newcommand\PPm{\mathbb{Q}}

\newcommand\EE{\mathbb{E}}
\newcommand\EEm{\mathbb{E}_\mathbb{Q}}
\newcommand\D{\mathcal{D}}
\newcommand\K{\mathcal{K}}
\newcommand\U{\mathcal{U}}
\newcommand\V{\mathcal{V}}
\newcommand\A{\mathcal{A}}
\newcommand{\Mm}{\mathcal{M}(E)}
\newcommand{\tri}{t^{(r)}_i}
\newcommand{\trim}{t^{(r)}_{i-1}}

\endlocaldefs

\begin{document}

\begin{frontmatter}
\title{Large Deviations of Kac's Conservative Particle System and Energy Non-Conserving Solutions to the Boltzmann Equation: A Counterexample to the Predicted Rate Function\thanksref{T1}}
\runtitle{Large Deviations of Kac's Conservative Particle System}
\thankstext{T1}{This work was supported by the UK Engineering and Physical Sciences Research Council (EPSRC) grant EP/L016516/1 for the University of Cambridge Centre for Doctoral Training, the Cambridge Centre for Analysis}

\begin{aug}
\author[A]{\fnms{Daniel} \snm{Heydecker}\ead[label=e1]{daniel.heydecker@mis.mpg.de}}
\address[A]{Max Planck Institut f\"ur Mathematik in den Naturwissenschaften, Inselstra{\ss}e 22, 04103 Leipzig, Germany
\printead{e1}}

\end{aug}

\begin{abstract}
We consider the dynamic large deviation behaviour of Kac's collisional process for a range of initial conditions including equilibrium. We prove an upper bound with a rate function of the type which has previously been found for kinetic large deviation problems, and a matching lower bound restricted to a class of sufficiently good paths. However, we are able to show by an explicit counterexample that the predicted rate function does not extend to a global lower bound: even though the particle system almost surely conserves energy, large deviation behaviour includes solutions to the Boltzmann equation which do not conserve energy, as found by Lu and Wennberg, and these occur strictly more rarely than predicted by the proposed rate function. At the level of the particle system, this occurs because a macroscopic proportion of energy can concentrate in $\mathfrak{o}(N)$ particles with probability $e^{-\mathcal{O}(N)}$.
\end{abstract}



\end{frontmatter}
\tableofcontents


\section{Introduction \& Results}
We consider a family of models, of a kind introduced by Kac \cite{kac1956foundations}, modelling a homogenous gas with elastic collisions associated to the homogeneous Boltzmann equation. We consider an $N$-particle system, where indistinguishable particles of mass $N^{-1}$ have velocities $V^1_t,...V^N_t \in \rrd$; since the particles are indistinguishable, all physical quantities are encoded in the normalised empirical measure $\mu^N_t=N^{-1}\sum_i\delta_{V^i_t}$. The velocities evolve over time by random, binary collisions, governed by a collision kernel $B:\rrd\times \ssd\to [0,\infty)$. For every (ordered) pair of particles with velocities $v, v_\star \in \text{supp}(\mu^N_t)$, the velocities change to $v'=v'(v, v_\star, \sigma)$ and $v'_\star=v'_\star(v,v_\star, \sigma)$, given by \begin{equation}\label{eq: PCV} v'(v, v_\star, \sigma)=v-((v-v_\star)\cdot \sigma)\sigma; \hspace{0.5cm} v_\star'(v, v_\star, \sigma)=v_\star+((v-v_\star)\cdot \sigma)\sigma \end{equation} at rate $2B(v-v_\star,\sigma)/N$. At the level of the empirical measure, we see the change \begin{equation} \label{eq: change of measure at collision} \mu \mapsto \mu^{N, v, v_\star, \Sigma} = \mu+\frac{1}{N}(\delta_{v'}+\delta_{v'_\star}-\delta_{v}-\delta_{v_\star})=\mu+\frac{1}{N}\Delta(v,v_\star,\sigma)  \end{equation} We will write $\Delta f(t,v,v_\star,\sigma)$ for the resulting change in a function $\Delta f(t,v,v_\star,\sigma)=f(t,v')+f(t,v_\star')-f(t,v)-f(t,v_\star)$ whenever $f:[0,T]\times\rrd\to \rr$ or $f: \rrd\to\rr$. It is straightforward to see that $\Delta f=0$ for $f=1, v_i, |v|^2$, which means that the quantities $\langle 1, \mu^N_t\rangle, \langle v, \mu^N_t\rangle$ and $\langle |v|^2, \mu^N_t\rangle $ are conserved as time runs. We will work throughout with one of the two kernels  \begin{equation} \label{eq: form of B} B(v,\sigma)=\begin{cases} 1+|v| & \text{(Regularised Hard Spheres)} \\ 1 & \text{(Maxwell Molecules)}. \end{cases}\end{equation} The first case is a `regularised' version of the (true) hard spheres case $B=|v|$, perturbed away from $0$ so that $v, \sigma \mapsto \log B(v,\sigma)$ is globally Lipschitz continuous. In both cases, since $B$ is independent of $\sigma$, we will use the abuse of notation $B(v)$, and the dynamics described informally above lead to a well-defined Markov process.  We will suppose, throughout, that the initial data  $V^1_0,\dots,V^N_0$ are given by independent samples from a reference measure $\mu_0^\star$, which we normalise so that $\langle v, \mu^\star_0\rangle=0, \langle |v|^2, \mu^\star_0\rangle=1$. We will write $\mathcal{E}_z(\mu^\star_0)$ for the Gaussian moments $$ \mathcal{E}_z(\mu^\star_0):=\int_{\rrd}e^{z|v|^2}\mu_0^\star(dv)\in [1,\infty]. $$  Kac's process is closely connected to the spatially homogeneous Boltzmann equation. Indeed, Kac introduced the collisional process as a toy model for which Boltzmann's celebrated \emph{Sto{\ss}zahlansatz} could be more easily derived: if the velocities are initially approximately uncorrelated, then this property is propagated in time. At the level of empirical measures this is equivalent to a hydrodynamic limit or a law of large numbers for the empirical measures $\mu^N_t$ towards a solution to the Boltzmann equation:  for all $t\ge 0$ and $f$ Lipschitz, \begin{equation}\tag{BE}\label{eq: BE} \langle f, \mu_t\rangle = \langle f, \mu_0\rangle + \int_0^t \langle f, Q(\mu_s)\rangle ds \end{equation} where $Q(\mu)$ is the Boltzmann collision operator, given for measures with finite first moment by \begin{align} Q(\mu)=\int_{\rrd\times\rrd\times\ssd}\Delta(v,v_\star,\sigma) B(v-v_\star)\mu(dv)\mu(dv_\star)d\sigma.  \end{align} Pathwise, one can understand the Kac process as a stochastic pertubation of the Boltzmann equation (\ref{eq: BE}): for any bounded $f$,  $$ \langle f, \mu^N_t\rangle =\langle f, \mu^N_0\rangle + \int_0^t \langle f, Q(\mu^N_s)\rangle ds + M^{N,f}_t$$ where $M^{N,f}_t$ is a martingale of quadratic variation $[M^{N,f}]_t=\mathcal{O}(N^{-1})$; see, for example, Norris \cite{norris2016consistency}. In this framework, the law of large numbers, for the kernels we consider, is well-understood; qualitative results go back to Sznitman \cite{sznitman1984equations}, and quantitative results have been obtained by Mischler and Mouhot \cite{mischler2013kac} and Norris \cite{norris2016consistency}. Following the ideas of Freidlin and Wentzel \cite{freidlin1998random}, we consider the dynamic large deviation behaviour of the Kac process on a fixed time interval $[0, T]$; throughout, we will use the subscript $_\bullet$ to denote processes indexed\footnote{This notation is chosen to minimise ambiguity between functionals which may either depend on the whole process $\mu_\bullet$, or only on a single measure $\mu$.} by this time interval $t\in [0,T]$, so that $\mu^N_\bullet=(\mu^N_t)_{0\le t\le T}$. \bigskip \\ It will be useful to consider the Kac process together with an auxiliary \emph{empirical flux}, which records the collision history of the process \cite{patterson2018large,renger2018flux}. We write $E=[0,T]\times\rrd\times\rrd\times \ssd$ for the parameter space of collisions and form measures $w^N_t$ on $E$ by setting $w^N_0=0$ and changing, at collisions, \begin{equation} w^N_t=w^N_{t-}+\frac{1}{N}\delta_{(t,v,v_\star,\sigma)}\end{equation} at times $t$ where there is a collision, choosing one possible assignment $(v,v_\star,\sigma)$ uniformly at random between the four\footnote{The collision parameters $(t,v, v_\star, \pm \sigma), (t, v_\star,v, \pm \sigma)$ all correspond to the same physical collision.} possible choices of collision parameters. In this way, the pair $(\mu^N_t, w^N_t)$ together form a Markov process, which are linked by a consistency relation (\ref{eq: CE}) below, and we write $w^N:=w^N_T$ for the final measure, containing the entire collisional history of the process. With this notation, we investigate estimates informally given by \begin{equation} \mathbb{P}\left((\mu^N_\bullet,w^N) \approx (\mu_\bullet,w) \right)\asymp  e^{-N\mathcal{I}(\mu_\bullet,w)}.\label{eq: informal LD}\end{equation} Formally, we consider the space $\cp_2$ of probability measures on $\rrd$ with finite second moment, equipped with the Monge-Kantorovich-Wasserstein distance \begin{equation} W(\mu, \nu)=\sup\left\{\langle f, \mu-\nu\rangle: f\in \mathcal{F}\right\};\end{equation}\begin{equation} \label{eq: mathcal F} \mathcal{F}=\left\{f: \mathbb{R}^d\rightarrow \mathbb{R}, \hspace{0.1cm}\|f\|_\infty \le 1, \hspace{0.1cm}\sup_{v\neq w} \frac{|f(v)-f(w)|}{|v-w|} \le 1\right\}  \end{equation} and write $\cp_2^N$ for the subspace consisting of empirical measures on $N$ points, and $\D$ for the Skorokhod space  \begin{equation}\label{eq: definition of D} \D:=\left\{\mu_\bullet \in D([0,T],(\cp_2,W)): \sup_{t\le T} \langle |v|^2, \mu_t\rangle <\infty\right\}\end{equation} which we equip with a metric inducing the Skorokhod $J_1$-topology. For the empirical fluxes, we write $\Mm$ for the space of finite Borel measures on $E$, which we equip with the Wasserstein$_1$ metric $d$ given analogously by \begin{equation} d(w, w')=\sup\left\{\langle g,w-w'\rangle: \sup_E |g|\le 1, \sup_{p, q\in E, p\neq q} \frac{|g(p)-g(q)|}{|p-q|}\le 1 \right\} \end{equation} where $| \cdot |$ is the Euclidean norm on $E\subset \mathbb{R}^{3d+1}$.  With this notation, the pair $(\mu^N_t, w^N_t)$ is a Markov process in $\cp_2^N\times\Mm$ with generator given on bounded functions by \begin{equation} \begin{split} \label{eq: generator} \mathcal{G}^NF(\mu^N, w^N)& =N\int_{\rrd\times\rrd\times \ssd} (F(\mu^{N,v,v_\star,\sigma}, w^{N,t,v,v_\star,\sigma})-F(\mu^N, w^N))\\& \hs\hs\hs \hs\hs \dots B(v-v_\star)\mu^N(dv)\mu^N(dv_\star)d\sigma\end{split}\end{equation}with \begin{equation} \mu^{N,v,v_\star,\sigma}:=\mu^N+\frac{1}{N}\Delta(v,v_\star,\sigma); \qquad  w^{N,v,v_\star,\sigma}:=w^N+ \frac{1}{N}\delta_{(t,v,v_\star,\sigma)}. \end{equation}   Formally, (\ref{eq: informal LD}) means that, for Kac processes $\mu^N_\bullet$ and associated empirical fluxes $w^N$, \begin{equation} \label{eq: formal LD UB} \limsup_N\frac{1}{N}\log \PP\left((\mu^N_\bullet, w^N) \in \A\right)\le -\inf\left\{ \mathcal{I}(\mu_\bullet,w) : (\mu_\bullet, w)\in \A\right\} \end{equation} for any $\A\subset \D\times\Mm$ closed, and \begin{equation} \label{eq: formal LD LB} \liminf_N\frac{1}{N}\log \PP\left((\mu^N_\bullet,w^N) \in \U\right)\ge -\inf\left\{\mathcal{I}(\mu_\bullet,w): (\mu_\bullet, w)\in \U\right\} \end{equation} $\U\subset \D\times \Mm$ open. It is well-known that such bounds give a precise mathematical meaning to Boltzmann's notion of entropy in terms of \emph{``volume of accessible microstates"}; see the discussion in \cite{villani2008h}. We will make the following further hypotheses on the reference measure $\mu_0^\star$: \begin{hypothesis} \label{hyp: condition on ref ms} \begin{enumerate}[label=\roman*).] \item \emph{Gaussian upper bound:} there exists $z_1>0$ such that $\mathcal{E}_{z_1}(\mu^\star_0)<\infty$. \item {Gaussian Lower Bound: }there exists $z_2<\infty$ such that $\mathcal{E}_z(\mu_0^\star)<\infty$ for $z<z_2$, and $\mathcal{E}_z(\mu_0^\star)\to \infty$ as $z\uparrow z_2$. \item \emph{Continuous Density:} $\mu_0^\star$ has a continuous density $f_0^\star$ with respect to the Lebesgue measure, and for some $z_3\in (0,\infty)$ and $c>0$, \begin{equation} f_0^\star \ge  ce^{-z_3|v|^2}.\end{equation} \end{enumerate} \end{hypothesis} Let us remark that this hypothesis allows the natural choice \begin{equation}\label{eq: Gaussian}  \mu_0^\star(dv)=\gamma(dv)=\frac{1}{(2\pi d)^{d/2}}e^{-d|v|^2/2}dv \end{equation} which is the normalised\footnote{to $0$ average velocity $\langle v, \mu\rangle$ and unit energy $\langle |v|^2, \mu\rangle$.} equilibrium for the Boltzmann equation (\ref{eq: BE}), and whose $N$-fold tensor product $\gamma^{\otimes N}$ is a normalised, reversible equilibrium for the many particle system. \bigskip \\ Under the hypothesis above, Sanov's Theorem \cite{zeitouni1998large} applies to show that the initial data satisfy a large deviation function in $(\cp_2, W)$ with rate function \begin{equation} H(\mu_0|\mu_0^\star):=\begin{cases} \int_{\rrd} \frac{d\mu_0}{d\mu_0^\star}\log \left(\frac{d\mu_0}{d\mu_0^\star}\right)\h \mu_0^\star(dv) & \text{if }\mu_0 \ll \mu_0^\star; \\ \infty & \text{else.}  \end{cases} \end{equation}   \subsection{A Proposed Rate Function} Let us review a possible rate function identified by L\'eonard\footnote{We remark that this definition differs from the works \cite{leonard1995large, basile2021large} by a factor of $\frac{1}{2}$; our definition of the Kac process rescales time by a factor of $2$ relative to these works, or equivalently summing over all ordered pairs rather than unordered pairs.} \cite{leonard1995large} for exactly this problem. For $\mu_\bullet \in \D$ we define $\overline{m}_{\mu}\in \Mm$ by \begin{equation} \overline{m}_{\mu}(dt,dv,dv_\star,d\sigma)=B(v-v_\star)dt \mu_t(dv)\mu_t(dv_\star)d\sigma.\end{equation}We say that $(\mu_\bullet, w)\in \D\times\Mm$ is a \emph{measure-flux pair} if $w\ll \overline{m}_{\mu}$ and if they solve the \emph{continuity equation}: for all $0\le t\le T$, \begin{equation} \label{eq: CE}\tag{CE} \mu_t=\mu_0+\int_E \Delta(v,v_\star,\sigma)1_{s\le t} \h w(ds,dv,dv_\star,d\sigma). \end{equation} We will use, throughout, the notation $K$ for the density $\frac{dw}{d\overline{m}_\mu}$, which we call a tilting function. With this notation, if $(\mu_\bullet, w)$ is a measure-flux pair, then $\mu_\bullet$ solves a modified Boltzmann equation \begin{equation} \label{eq: BEK} \tag{BE$_K$} \begin{split} \mu_t&=\mu_0+\int_{E} \Delta (v,v_\star,\sigma)K(s,v,v_\star,\sigma)B(v-v_\star)ds\mu_s(dv)\mu_s(dv_\star)d\sigma. \end{split} \end{equation}  Equivalently, given $\mu_\bullet$ solving (\ref{eq: BEK}) for some $K\in L^1(\overline{m}_\mu)$, one can define $w=K\overline{m}_\mu$ and $(\mu_\bullet, w)$ is a measure-flux pair. We define the dynamic cost of a trajectory $(\mu_\bullet, w)\in \mathcal{D}\times\Mm$ to be \begin{equation} \mathcal{J}(\mu_\bullet, w):=
\begin{cases}\int_E \tau\left(\frac{dw}{d\overline{m}_{\mu}}\right)\overline{m}_\mu(ds,dv,dv_\star,d\sigma) & \text{ if $(\mu_\bullet, w)$ is a measure-flux pair;} \\ \infty &\text{else} \end{cases} \end{equation} where $\tau:[0,\infty]\to[0,\infty]$ is the function $\tau(k)=k\log k-k+1$, and define the full rate function to be \begin{equation}\label{eq: leo rate function} \mathcal{I}(\mu_\bullet,w):=H(\mu_0|\mu_0^\star)+\mathcal{J}(\mu_\bullet,w).\end{equation} An analagous upper bound, which can be obtained from this rate function using the contraction principle on $(\mu_\bullet, w)\to \mu_\bullet$ is obtained by L\'eonard \cite{leonard1995large} in a different topology, and the same rate function has been found in other contexts for kinetic large deviations. Since the first version of this work, the works \cite{basile2021large',basile2022asymptotic} have introduced, for the same problem, closely related but strictly larger rate functions; these connections will be discussed in the literature review below. \subsection{Main Results} Our first result collects some useful facts on the proposed rate function $\mathcal{I}$ and on the exponential tightness.  \begin{proposition}[Exponential Tightness and Semicontinuity] \label{prop: ET + UB} Fix a probability space $(\Omega, \mathfrak{F}, \PP)$. For $N\ge 2$, let $\mu^N_\bullet$ be either regularised hard sphere or Maxwell Molecule Kac processes with initial velocities drawn independently from a measure $\mu_0^\star$ satisfying Hypothesis \ref{hyp: condition on ref ms}i). Then the following hold. \begin{enumerate}[label=\roman*).] \item  The random variables $(\mu^N_\bullet,w) \in \D\times \Mm$ are exponentially tight: for any $M>0$, there exists a compact set $\K\subset \D\times \Mm$ such that \begin{equation}\label{eq: ET} \limsup_N\frac{1}{N}\log \PP \left((\mu^N_\bullet, w^N)\not \in \K\right)\le -M. \end{equation}   \item The function $\mathcal{I}$ is lower semicontinuous on $\D\times \Mm$: the lower sub-level sets \begin{equation} \{(\mu_\bullet,w)\in \D\times\Mm: \mathcal{I}(\mu_\bullet,w)\le a\}\subset \D\times\Mm \end{equation} are closed when $\D$ has the Skorokhod topology and $\Mm$ the weak topology, metrised by $d$.\end{enumerate}\end{proposition} We emphasise that we do not claim that $\mathcal{I}$ is `good', in that the sub-level sets are compact; indeed, Theorem \ref{thm: main} suggests that this is false. \bigskip \\ The positive result we prove on the large deviations is as follows. We rederive, in our context, the upper bound with our rate function, which reproduces the result of L\'eonard \cite{leonard1995large}, and prove a lower bound with the same rate function on a restricted set. In this way, the proposed rate function captures at least some of the correct large deviation behaviour of the Kac process. \begin{theorem}\label{thrm: main positive} Let $B$ be either the regularised hard spheres or Maxwell molecules kernel, and for $N\ge 2$ let $(\mu^N_\bullet, w^N)$ be a Kac process and its flux, with particles drawn initially from $\mu_0^\star$ satisfying Hypothesis \ref{hyp: condition on ref ms}, and let $\mathcal{I}$ be the rate function given above. Then \begin{enumerate}[label=\roman*).] \item For all $\A\subset \D\times\Mm$ closed, we have \begin{equation} \label{eq: UB} \limsup_N \h \frac{1}{N}\log\PP\left((\mu^N_\bullet, w^N)\in \A \right) \le -\inf\left\{\mathcal{I}(\mu_\bullet, w): (\mu_\bullet, w)\in \A\right\}. \end{equation}  \item For all $\U\subset \D\times\Mm$ open, we have \begin{equation} \label{eq: RLB} \liminf_N \h \frac{1}{N}\log\PP\left((\mu^N_\bullet, w^N)\in \U \right) \ge -\inf\left\{\mathcal{I}(\mu_\bullet, w): (\mu_\bullet, w)\in \U\cap\mathcal{R}\right\} \end{equation} where $\mathcal{R}=\{(\mu_\bullet, w)\in \D\times \Mm: \langle 1+|v|^2+|v_\star|^2, w\rangle <\infty\}.$ \end{enumerate}\end{theorem} As discussed in the literature review, the proposed rate function is a very natural candidate for describing the large deviations behaviour, and one might expect to be able to find a `true' lower bound (\ref{eq: formal LD LB}). In this context, the restricted lower bound presented here is somewhat dissatisfying, as it leaves open the question of which open sets $\U$ are such that $\inf_\U \mathcal{I}=\inf_{\U\cap\mathcal{R}} \mathcal{I}$, or the possibility that a better upper bound may be possible. The restriction to a set $\mathcal{R}$ of `good' paths, as here, is necessary for the paths in question to be approximated by paths which can be recovered by a Girsanov transform; see Lemma \ref{lemma: approximation lemma}. Key to this argument is that these paths should be uniquely specified by the initial data and tilting $K$, so that the path is the unique possible hydrodynamic limit of `tilted' dynamics along any subsequence. However, at the level of the Boltzmann equation (\ref{eq: BE}), this uniqueness is known \emph{not} to hold: solutions with increasing energy have been constructed by Lu and Wennberg \cite{lu1999solutions}. Since the energy $\langle |v|^2, \mu^N_t\rangle$ is almost surely conserved by the paths of the stochastic Kac process, one might hope that such solutions are spurious and can be excluded from the large deviation analysis, so that uniqueness does hold.  However, we prove the following theorem, which shows that such solutions can be reached with finite exponential cost, and so cannot be excluded from the large deviation analysis, but the occurrence of such paths is not correctly predicted by the proposed rate function. Equivalently, this example can be understood as producing explicitable open sets $\U$ such that the infima of the rate function over $\U$ and $\U\cap\mathcal{R}$ do not coincide.
\begin{theorem}\label{thm: main} Assume the notation of Proposition \ref{prop: ET + UB}. \begin{enumerate}[label=\roman*).]\item Suppose $B$ is the regularised hard spheres kernel, and the reference measure $\mu_0^\star$ satisfies Hypothesis \ref{hyp: condition on ref ms}i-ii). Let $\Theta:[0,T]\to (0,\infty)$ be nondecreasing, nonconstant and left-continuous, with $\Theta(0)=1$ and such that, for some closed set $P\subset [0,T], 0\in P, T\not \in P$ with null interior, $\Theta$ is locally constant on $[0,T]\setminus P$.  For some constant $\alpha=\alpha(\Theta(T))$, define \begin{equation}\label{eq: At} A(t):=\alpha \left(\inf_{s\in P: s\le t} (t-s)\right)^{-2} \in (0,\infty] \end{equation} and consider the set $\mathcal{A}_\Theta$ given by \begin{equation}\begin{split} \label{eq: bad set AE}\mathcal{A}_\Theta:= &\bigg\{(\mu_\bullet, w) \in \D\times \Mm: \mu_\bullet \text{ solves (\ref{eq: BE}), }w=\overline{m}_\mu, \h \mu_0=\mu^\star_0, \\ & \hspace{3cm} \text{and for all }t\ge 0, \h \langle |v|^2, \mu_t\rangle=\Theta(t) \text{ and }\langle |v|^4, \mu_t\rangle \le A(t) \bigg\}.\end{split}\end{equation} Then $\mathcal{A}_\Theta$ is compact, nonempty, and $\mathcal{I}(\mu_\bullet,w)=0$ on $\mathcal{A}_\Theta$. We have \begin{equation}\label{eq: first item of main} \inf_{\U\supset\A_\Theta}\liminf_N\frac{1}{N}\log\PP\left((\mu^N_\bullet, w^N)\in \U\right)\ge -\Theta(T)z_2 \end{equation} where the infimum runs over all open sets $\U\subset\D\times \Mm$ containing $\A_\Theta$, and there exists an open set $\V\supset \mathcal{A}_\Theta$ such that \begin{equation} \label{eq: second item of main} \liminf_N \frac{1}{N} \log \PP\left((\mu^N_\bullet, w^N) \in \V \right)<0.\end{equation}\item  Suppose instead that $B$ is the cutoff Maxwell Molecules kernel. For all $\delta>0$ and $\Theta$ as above, and with $A$ as above with $\alpha$ depending on $\delta$ as well as $\Theta(T)$, define the set \begin{equation} \label{eq: bad set AED} \begin{split} \A_{\Theta, \delta}:= &\bigg\{(\mu_\bullet,w) \in \D\times \Mm: (\mu_\bullet,w) \text{ is a measure-flux pair with }K=1+\delta|v-v_\star|,  \\ & \hspace{2cm} \mu_0=\mu^\star_0,\text{ and for all }t>0, \h \langle |v|^2, \mu_t\rangle=\Theta(t) \text{ and }\langle |v|^4, \mu_t\rangle \le A(t) \bigg\}.\end{split} \end{equation} The sets $\A_{\Theta, \delta}$ are compact and nonempty, and $\mathcal{I}(\mu_\bullet,w)\le 4\delta^2\Theta(T)T$ for all $(\mu_\bullet,w) \in \A_{\Theta, \delta}$. We have \begin{equation}\label{eq: first point of main MM} \inf_{\U\supset\A_{\Theta, \delta}}\liminf_N \frac{1}{N}\log \PP\left((\mu^N_\bullet,w^N) \in \U\right) \ge -\Theta(T)(z_2+C\delta)\end{equation} where, as above, the infimum runs over all open sets $\U$ containing $\A_{\Theta, \delta}$. However, there exist open sets $\V_\delta \supset \A_{\Theta, \delta}$ such that, for any $\Theta$,  \begin{equation} \label{eq: second point of main MM}\limsup_{\delta\downarrow 0}\h \liminf_N \frac{1}{N}\PP\left((\mu^N_\bullet,w^N) \in \mathcal{V}_{\delta}\right)<0.  \end{equation} \end{enumerate} \end{theorem} In both cases, the first point shows that such behaviour cannot be excluded by superexponential estimates, and so is a form of behaviour which must be taken into account in the large deviation theory; the second point shows that the rate function on such paths is not that predicted above. The argument we will present is a stochastic, large deviation analogue of the construction of Lu and Wennberg \cite{lu1999solutions}, keeping track of the exponential change of measure necessary. In the first case, we will construct changes of measure $\PPm^N\ll\PP$, with an exponential cost associated to changing the initial data and sub-exponential cost associated to modifying the dynamics, so that $\mathfrak{o}(N)$ particles containing $\mathcal{O}(1)$ energy are temporarily `frozen' and, under these new measures, the Kac processes concentrate on the set $\A_\Theta$ given. The argument for the Maxwell molecule case is similar, with an additional exponential cost $\mathcal{O}(e^{N\delta})$ necessary to modify the dynamics.   In these cases, the behaviour of $\mathfrak{o}(N)$ particles has a macroscopic effect on the evolution of the whole process, meaning that the large deviation behaviour is not purely captured by the empirical measure and control $K$. Possible generalisations of this phenomenon will be discussed in Section \ref{sec: lit}.2 below \bigskip \\ Let us now examine some consequences of Theorems \ref{thrm: main positive}, \ref{thm: main}. One might hope that it is possible to prove a true large deviation principle under well-chosen initial conditions where one puts in `by hand' that there is no such concentration initially. The following easy corollary, exploiting the reverseability of the particle system $\mu^N_\bullet$ in the equilibrium $\gamma^{\otimes N}$, suggests that, even under such well-chosen conditions, the same concentration of energy can occur as a result of the binary collisions. \begin{corollary}\label{cor: bad by evolution} Let us take $\mu_0^\star=\gamma$, and fix a decreasing, right-continuous function $\Theta$, $\Theta(T)=1$, which is locally constant aside from at a closed set $P\subset [0,T]$ with empty interior, $T\in P, 0\not \in P$. For either Maxwell molecules or hard spheres, there exists an explicitable function $A$ such \begin{equation} \begin{split} \mathcal{B}&=\bigg\{(\mu_\bullet,w) \in \D\times \Mm: \langle |v|^2, \mu_t\rangle = \Theta(t)\text{ for all }t\in [0,T] \text{ and }\langle |v|^4, \mu_t\rangle \le A(t)  \bigg\} \end{split} \end{equation} satisfies \begin{equation} \inf_{\U\supset \mathcal{B}} \h\liminf_N \frac{1}{N} \log \PP\left((\mu^N_\bullet,w^N) \in \U\right) > -\infty. \end{equation} where, as above, the outer infimum runs over open $\U\subset \D$ containing $\mathcal{B}$. \end{corollary} As a result, it is not possible to find a superexponential estimate to prevent the accumulation of energy in $\mathfrak{o}(N)$ particles at future times. \bigskip \\  Since the stochastic processes $\mu^N_\bullet$ are exponentially tight in $\D$, it follows that one can extract subsequences satisfying a \emph{true} large deviation principle. As a consequence of Theorem \ref{thm: main}, no such subsequence can avoid the bad paths we have constructed. \begin{corollary}\label{cor: no energy conserving LDP} Let $\mu^\star_0$ be a reference measure satisfying Hypothesis \ref{hyp: condition on ref ms}i-ii), and let $\mu^N_\bullet, w^N$ be $N$-particle Kac processes, either for the regularised hard spheres or Maxwell molecules case. Suppose that $S\subset \mathbb{N}$ is an infinite subsequence such that $(\mu^N_\bullet, w^N)_{N\in S}$ satisfy a large deviation principle in $\D\times\Mm$ with some rate function $\widetilde{\mathcal{I}}$. Then there exists $(\mu_\bullet, w)$ in $\D\times\Mm$ such that $\widetilde{\mathcal{I}}(\mu_\bullet,w)<\infty$ but such that $t\mapsto \langle |v|^2, \mu_t\rangle$ is not constant. \end{corollary} Our final corollary is a positive result, following from Theorem \ref{thrm: main positive}, which shows how the entropy plays the role of a \emph{quasipotential} for the Kac dynamics. Let us refer the reader to \cite[Section 3.3]{bouchet2020boltzmann} for a general discussion of such results. \begin{corollary}[Entropy as a Quasipotential]\label{cor: quasipotential} Let $B$ be either the regularised hard spheres or Maxwell molecules kernel, and fix $\mu\in \cp_2$. Then \begin{equation}\label{eq: LB of quasipotential} H(\mu|\gamma) \ge \inf\left\{H(\nu_0|\gamma)+\int_E \tau(K)d\overline{m}_{\nu}: \nu\in \D, \nu\text{ solves (\ref{eq: BEK})}, \nu_T=\mu\right\} \end{equation} and \begin{equation}\label{eq: UB of quasipotential} \begin{split} H(\mu|\gamma) & \le  \inf\bigg\{H(\nu_0|\gamma)+\int_E \tau(K)d\overline{m}_{\nu}: \nu\in \D, \nu\text{ solves (\ref{eq: BEK})}, \nu_T=\mu,\\  &\hspace{7cm}\text{ and }\int_E (|v|^2+|v_\star|^2)Kd\overline{m}_\nu<\infty\bigg\}. \end{split} \end{equation} \end{corollary}  In this sense, we view $\tau(K)$ as the entropic cost of moving to a higher-entropy state by tilted Boltzmann dynamics (\ref{eq: BEK}). The second item generalises Boltzmann's famous $H$-Theorem; however, in light of the nonemptiness of the sets $\A_\Theta$ in Theorem \ref{thm: main}, the upper bound would be false without the second moment condition.\bigskip \\  The paper is structured as follows. In the remainder of this section, we will review some recent works on large deviations and related problems, and make some remarks on the hypothesis and functional framework. In Section \ref{sec: ETUB}, we derive an upper bound Theorem \ref{thrm: main positive}i); in doing so, we will prove Proposition \ref{prop: ET + UB} via a variational formulation of the rate function $\mathcal{I}$ which appears in the upper bound. Section \ref{sec: properties_and_COM} gathers some properties of the Kac process and changes of measure from the literature for later convenience. Section \ref{sec: RLB} proves the restricted lower bound Theorem \ref{thrm: main positive}ii), based on an approximation argument for paths belonging to $\mathcal{R}$ and a standard `tilting' argument. The proof of Theorem \ref{thm: main} is given in Section \ref{sec: pf of main}, based on the properties of the Kac process in Section \ref{sec: properties_and_COM} and a careful analysis of Cram\'er bounds, and we deduce the corollaries in Section \ref{sec: corrs}. Finally, Appendix \ref{sec: sk} is a self-contained appendix on the Skorokhod topology and Appendix \ref{sec: singular girsanov}  contains a justification of the change-of-measure formula. \subsection{Literature Review \& Discussion}\label{sec: lit}

\paragraph{Large Deviations for Jump Particle Systems} The  theory of large deviations for Markov processes in the small-noise limit goes back to Freidlin and Wentzell \cite{freidlin1998random}. The seminal work of Feng \& Kurz \cite{feng2006large} developed tools based on a comparison principle for Hamilton-Jacobi equation in infinite dimensions, which are general but hard to verify. The analysis is somewhat different in the case where the dynamics are driven by diffusive rather than jump noise, see the discussion in L\'eonard \cite{leonard1995large}. In this context let us mention the recent works \cite{banerjee2020new,budhiraja2020large,budhiraja2021empirical,nguyen2021large}. \bigskip \\ Within collisional kinetic theory, previous works have reported upper bounds of a similar form. The work of L\'eonard \cite{leonard1995large} already cited considers the same case of the energy-preserving Kac model, and produces a rate function exactly given by $\mathcal{I}(\mu_\bullet)=\inf_w\mathcal{I}(\mu_\bullet, w)$, albeit for a different topology. Rezakhanlou \cite{rezakhanlou1998large} considers a collisional model for a spatially inhomogeneous gas, where the positions take values in the unit circle $\rr/\mathbb{Z}$ and the velocities take values only in a finite set, and finds an upper bound and a restricted lower bound, where the infimum runs only over a subset $\mathcal{R}\cap\U$ as in Theorem \ref{thrm: main positive} rather than over the full open set $\U$ as in (\ref{eq: formal LD LB}), with a rate function analagous to the variational form (cf Lemma \ref{lemma: variational form of RF} or \cite[Theorem 7.1]{leonard1995large}). Bodineau et al. \cite{bodineau2020fluctuation} consider the full spatially inhomogeneous Boltzmann--Grad limit with random initial data and deterministic dynamics for local interactions; the rate function is again given in a variational form, and the lower bound is again restricted to sufficiently good paths.  \bigskip \\ Finally, let us mention the recent works of Basile et al. \cite{basile2021large,basile2021large',basile2022asymptotic}, of which \cite{basile2021large',basile2022asymptotic} appeared after the first version of this work. In \cite{basile2021large}, the authors consider `Kac--like' random walks, which preserve momentum but not energy; the lower bound is again of the restricted form as in Theorem \ref{thrm: main positive}. The later works \cite{basile2021large',basile2022asymptotic} deal with a discrete Kac like model with a conserved energy and the true (energy-conserving) Kac process respectively, in both cases introduce a new rate function which assigns a non-zero rate to paths with energy evaporation \cite{basile2021large'} or energy creation \cite{basile2022asymptotic} respectively, and in both cases proving an upper bound and a restricted lower bound with the new rate function. The former work \cite{basile2021large'} also establishes that there is a path of finite rate with energy evaporation, which exactly corresponds to what we do implicitly in Corollary \ref{cor: bad by evolution}. In the work \cite{basile2022asymptotic}, the new rate function $\overline{\mathcal{I}}$ agrees with the one established here on the class of paths $\mathcal{R}$, but is strictly positive on the classes $\mathcal{A}_\Theta, \mathcal{B}$ in Theorem \ref{thm: main} and Corollary \ref{cor: bad by evolution}. Moreover, this work also replaces the regularised hard spheres kernel $1+|v|$ considered here by the more important \emph{true} hard spheres kernel $|v|$.  \bigskip \\ Outside of kinetic theory, analagous rate functions have been found for large deviations of jump processes, for instance \cite{djehiche1998large}. A number of works \cite{patterson2018large,patterson2016dynamical,renger2018flux} have considered the case of `reaction networks', which formally includes the Kac/Boltzmann dynamics considered here by viewing $\rrd$ as a continuum of particle species; these works are the origin of considering the pair $(\mu^N_\bullet, w^N)$ which significantly eases the analysis. Other works \cite{dupuis2016large,kraaij2017flux} in the context of reaction networks or mean-field dynamics exploit a control representation of the dynamics, leading to an equation similar to (\ref{eq: BEK}) with random controls $K$ and the same cost function $\tau$, and using weak convergence method due to Dupuis \cite{dupuis2011weak}. In this weak converence method, it is essential that the control uniquely determines possible limiting paths (see a similar argument in the proof of Theorem \ref{thm: main}ii) in Section \ref{sec: RLB}), whereas this type of uniqueness is known not to be hold in the Boltzmann case, even in the most advantageous possible case of Maxwell molecules. In the work \cite{kraaij2017flux}, the key to removing the restriction on regular paths is an approximation argument so that paths are perturbed to lie in the \emph{interior} of the space of probability measures on the space of the finite space of species $S$, which is clearly impossible in the infinite-dimensional setting here. \bigskip \\ Let us mention that this is a very natural form for the rate function jump processes. One recognises $\tau$ as the dynamic cost of controlling a Poisson random measure: for a Poisson process $(Z_t)_{t\ge 0}$ of unit intensity, a straightforward argument using Stirling's formula shows that $ \PP(N^{-1}Z_N \approx z) \asymp \exp(-N\tau(z))$ in the same sense as (\ref{eq: formal LD UB}, \ref{eq: formal LD LB}). Similarly, if one fixes a finite space $S$ and a probability measure $\mu$ and forms $X_N$ as a Poisson random measure of intensity $N\mu$, then one has the equivalent \begin{equation} \PP\left(N^{-1}X_N\approx \nu\right)\asymp \exp\left(-N\sum_S \tau\left(\frac{\nu(x)}{\mu(x)}\right)\mu(x)\right)=\exp\left(-N\int_S \tau\left(\frac{d\nu}{d\mu}\right)\mu(dx)\right).\end{equation} The proposed rate function above would correspond to the intuition that, given $\mu^N_t\approx \mu_t$, the instantaneous distribution of jumps is approximately Poisson, with intensity $\approx B(v-v_\star)dt\mu_t(dv)\mu_t(dv_\star)d\sigma$. \bigskip \\ As remarked above, several other works \cite{djehiche1998large,rezakhanlou1998large,bodineau2020fluctuation,basile2021large} have encountered the same problem that the lower bound can only be proven over a class of good paths. Both the works \cite{djehiche1998large,rezakhanlou1998large} conjecture that a `true' lower bound should hold in the respective frameworks. In the works cited above, such a conjecture has only been proven in the cases of reaction networks with a finite set of species \cite{kraaij2017flux,patterson2018large,patterson2016dynamical,renger2018flux} or mean-field dynamics with finite state space \cite{dupuis2016large}, which are very far from the Kac/Bolzmann dynamics we consider. To the best of our knowledge, the current work represents the first time that such a hypothesis has been falsified.
\paragraph{Remarks on the Hypotheses \& Functional Framework}We make the following remarks on the Hypothesis \ref{hyp: condition on ref ms} and on the functional framework. Firstly, the hypotheses allow the very natural choice of taking $\mu^\star_0$ to be the equilibrium distribution $\gamma$ given by (\ref{eq: Gaussian}) but Hypothesis \ref{hyp: condition on ref ms}ii). disallows measures of the form $\mu_0^\star(dv)\propto (1+|v|^2)^{-m}\gamma(dv), m>\frac{d}{2}$. In general, the condition that $\mu^N_0$ be given by drawing particles independently from a reference measure $\mu_0^\star$ will not propagate in time. However, this is natural in order to ensure that $\mu^N_0$ satisfies a large deviation principle; elementary counterexamples can be found to show that the more usual conditions, that the initial data be chaotic or entropically chaotic \cite{hauray2014kac}, do not imply a large deviation principle for $\mu^N_0$. Moreover, in the most important case $\mu_0^\star=\gamma$, the independence \emph{is} propagated, as $\gamma^{\otimes N}$ is an equilibrium distribution for the $N$-particle dynamics. \bigskip \\ Regarding the functional framework, while $(\cp_2, W)$ is not complete, the choice of metric $W$ and Skorokhod space $\D$ are natural to guarantee that $(\mu^N_\bullet, w^N)$ are exponentially tight. One could alternatively equip $\cp_2$ with the Wasserstein$_2$ metric $W_2$, which makes the map $\mu \mapsto \langle |v|^2, \mu\rangle$ continuous, and one can take a limit of the pathwise energy conservation $\langle |v|^2, \mu^N_t\rangle = \langle |v|^2, \mu^N_0\rangle$ to conclude that all possible large deviation paths still conserve energy. However, carefully following the arguments leading to our counterexamples in Section \ref{sec: pf of main} proves that the initial measures $\mu^N_0$ then fail to be exponentially tight, as does the whole process $(\mu^N_\bullet, w^N)$. Since large deviations techniques rely heavily on such tightness to prove the existence of subsequential limits under the change of measures, we have been unable to determine $\mathcal{I}$ correctly determines the large deviations in this framework. In light of this, we interpret Theorem \ref{thm: main} as showing the existence of a different kind of large deviations behaviour, where macroscopic energy concentrates in $\mathfrak{o}(N)$ particles, which is not captured by convergence in $(\cp_2, W_2)$. \bigskip \\ In future works, it may be interesting to consider the large deviations in the functional framework similar to that of L\'eonard \cite{leonard1995large}. Let us write $C_\mathrm{qu}(\rrd)$ for the continuous functions of quadratic growth, we write $\mathfrak{P}_2$ for the space of linear maps $\mathfrak{m}: C_\mathrm{qu}(\rrd)\to \rr$ satisfying $\mathfrak{m}(1)=1$, $\mathfrak{m}(\varphi)\ge 0$ whenever $\varphi\ge 0$, and such that there exists $\mu=j[\mathfrak{m}]\in \cp_2$ with $\langle \varphi, \mu\rangle = \mathfrak{m}(\varphi)$ for all bounded $\varphi\in C_\mathrm{qu}(\rrd)$. We then equip $\mathfrak{P}_2$ with the product topology from the inclusion $\mathfrak{P}_2\subset \rr^{C_\mathrm{q}(\rrd)}$, and we can view $\cp_2\subset \mathfrak{P}_2$ via the identification $\iota: \cp_2\to \mathfrak{P}_2$, $\iota(\mu)(\varphi):=\langle \varphi, \mu\rangle$, so that the Kac process can be understood as taking values in $\mathfrak{P}_2$. Moreover, thanks to the classical theorems of Tychonoff or Banach-Alaoglu, the sets \begin{equation} \mathfrak{K}_a=\left\{\mathfrak{m}\in \mathfrak{P}_2: \text{ for all }\varphi\in C_\mathrm{qu}(\rrd), |\mathfrak{m}(\varphi)|\le a\sup_v \frac{|\varphi(v)|}{1+|v|^2}\right\}\end{equation} are compact for all $a\in [0,\infty)$, and cover $\mathfrak{P}_2$. In this framework, one has both exponential tightness, and continuity of the map $\mathfrak{m}\mapsto \mathfrak{m}(|v|^2)$. On the other hand, we warn the reader that elements of $\mathfrak{P}_2$ are typically \emph{not} measures, since $j[\mathfrak{m}]=\mu$ does not imply that $\mathfrak{m}=\iota(\mu)$. Indeed, following the construction of the initial data in Section \ref{sec: pf of main} leading to Theorem \ref{thm: main} produces limits with $\mathfrak{m}(|v|^2)=\Theta(T)>\langle |v|^2, j[\mathfrak{m}]\rangle = 1$. 
\paragraph{Other Models \& Generality of the Phenomenon} Let us first remark that, although the current work only focusses on the two kernels identified in (\ref{eq: form of B}), identical arguments would apply with $B(v)=1+|v|^\gamma$ in place of the regularised hard spheres kernel, at the cost of further complications. The modification by adding $1$ is also not necessary for Theorem \ref{thm: main}, but eases some technical difficulties in Theorem \ref{thrm: main positive}ii). L\'eonard \cite{leonard1995large} also obtains an upper bound in the case of non-cutoff Maxwell molecules, where $B(v,\sigma)$ has a non-integrable angular singularity as $\sigma\to \pm v/|v|$, leading to an abundance of grazing collisions.    \bigskip \\ It may be interesting in future to consider the large deviations of other collisional processes associated to the Boltzmann Equation. In the Nanbu process \cite{nanbu1983interrelations}, the rate of each collision is doubled, but only one particle at a time is updated. In this case, one can construct a measure dominating $\PP$ under which the jumps are independent, while this is impossible for the Kac process (see the comments \cite[Remark 1.9i]{leonard1995large}). One could also consider the large deviations behaviour of the Kac model coupled to a heat bath \cite{bonetto2014kac, tossounian2015partially} which may introduce enough additional randomness to avoid the kind of behaviour described here. Heuristically, one could already view the counterexamples in Theorem \ref{thm: main} as constructing a heat bath from $\mathfrak{o}(N)$ particles, which are allowed to drive energy into the system; it would also be interesting to formalise this connection. \bigskip\\   It also seems that the key ingredients of the counterexample Theorem \ref{thm: main} may generalise to other large deviation systems. Although we will not explore the general case in more detail, the key points we require generalise to an interacting particle system on a locally compact\footnote{for such a space and $f:S\to \rr$, we say that $f\to\infty$ if all preimages $\{x: f(x)\le M\}\subset S$ are compact.} state space $S$ as follows: \begin{enumerate} \item \textbf{Conserved Quantity: } There exists a continuous  $\varphi: S\to [0,\infty)$, $\varphi\to \infty$ such that, almost surely, $\langle \varphi, \mu^N_t\rangle$ is constant along sample paths; \item \textbf{Criticality: } The initial distributions are such that $F(z):=\lim_N N^{-1}\log \EE[e^{Nz\langle \varphi, \mu^N_0\rangle}]$ exists in $[0,\infty]$ for all $z$. Moreover, the function $F(z)$ is finite on a neighbourhood $I$ of the origin, but diverges to infinity as $z\uparrow \sup I<\infty$. \item \textbf{Delocalisation Mechanism:} For some continuous $\psi$ with $\psi \ge \varphi, \psi/\varphi \to \infty$, one \textbf{either} has \begin{enumerate} \item Uniformly in $N$, for all $t>0$ and all starting points $\mu^N_0$, $\EE\langle \psi, \mu^N_t \rangle < \infty$ can be controlled only in terms of $t$ and $\langle \varphi, \mu^N_0\rangle$, uniformly in compact subsets of $t\in (0,\infty)$; \textbf{or} \item For some $\mu^N_0$ and for all $\delta>0$, one can find changes of measure $\PPm^{N, \delta}\ll\PP$ by modifying only the dynamics, such that $\sup_N \EE_{\PPm^{N,\delta}} \langle \psi, \mu^N_t\rangle $ can be controlled in terms of $t, \langle \varphi,\mu^N_0\rangle, \delta$, uniformly in $N$, uniformly in compact subsets of $t\in (0,\infty)$, and the perturbation is small in the sense that $\PPm^{N, \delta}(\frac{d\PPm^{N,\delta}}{d\PP}>e^{N\delta a})<\frac{1}{2}$ for $N$ large enough, for some $a$ depending only on $\langle \varphi, \mu^N_0\rangle$. \end{enumerate}  \end{enumerate} In our case, the conserved quantity $\varphi$ is the energy $\varphi(v)=|v|^2$, and this would apply to any system with stochastic, energy-preserving dynamics. The second point is natural for initial data sampled from Gibbs distributions in statistical mechanics, where the density with respect to some Lebesgue measure is given by $\propto e^{-H} = e^{-N\langle \varphi, \mu^N_0\rangle}$, and corresponds to Hypothesis \ref{hyp: condition on ref ms}i-ii). The third point says that, potentially under a small perturbation of the dynamics, the system rapidly distributes $\varphi$ among all particles. By item 2, no exponential moments for $\langle \psi, \mu^N_0\rangle$ can be hoped for, so that bounds on $\EE_{\mathbb{Q}^N}\langle \psi, \mu^N_t\rangle$ will not hold under typical changes of measure $\mathbb{Q}^N\ll \PP$: we only ask that one such change of measure can be found. In our case, this r\^ole will be played by the moment creation property and Povzner estimates with $\psi=|v|^p, p>2$, see Proposition \ref{prop: MCP}; case a) corresponds to regularised hard spheres, and case b) to Maxwell molecules.  \bigskip \\ In either case, since $\varphi$ is necessarily unbounded for item 2 to hold, the dynamics cannot only be captured by a weakly continuous function of the empirical measure. Item 2 allows cases where a macroscopic pertubation of $\langle \varphi, \mu^N_0\rangle$ is achieved with only a small pertubation of $\mu^N_0$ in the weak topology, and playing the pathwise conservation (item 1) against the delocalisation mechanism (item 3) instantaneously spreads this pertubation to the whole empirical measure. This leads to a law of large numbers for paths $\mu_\bullet$ along which $\langle \varphi, \mu_t\rangle $ is not conserved, and is a given, nonconstant function $\Theta(t)$ which is constant aside from a jump discontinuity at $0$; more general $\Theta$ could be found with further assumptions on the dynamics. The dynamic cost required for such paths is either $0$, or $e^{\mathcal{O}(N\delta)}$, by following exactly the arguments in Section \ref{sec: pf of main}; the conclusion that the large deviation occurrence of such limit paths is not correctly predicted by the na\"ive rate function then follows by exploiting the conflict between the non-conservative limit paths and conservative finite-$N$ paths (item 1).

\paragraph{Relationship to Other Problems} We mention some other aspects of the Boltzmann/Kac dynamics which are related to the current work. \bigskip \\ As already mentioned above, large deviations give a probabilistic meaning to Boltzmann's Entropy functional $H(\cdot|\gamma)$; the $H$-Theorem, which guarantees that this decreases along solutions to (\ref{eq: BE}) or its spatially inhomogeneous version, goes back to the foundations of kinetic theory. Quantitative versions of this increase, and hence qualifying the convergence to equilibrium, have been a major topic in the analysis of the Boltzmann Equation (among many others, \cite{cercignani1982h,toscani1999sharp,bobylev1999rate,villani1999trend,villani2003cercignani,desvillettes2010celebrating}). Let us also mention the work of Mischler and Mouhot \cite{mischler2013kac}, which gives a probabilistic proof of the $H$-theorem via entropic chaos of the Kac process; however, as remarked above, entropic chaos does not lead to the large deviations considered here.\bigskip \\ Following the seminal work of Jordan, Kinderlehrer and Otto \cite{jordan1998variational}, it has been shown that many equations of mean-field type can be understood as the gradient flow of the entropy for a metric adapted to the particular problem, so that the dynamics not only increase entropy, but do so in the most efficient way possible. Further, it is known that such gradient flow properties can be derived from large deviation principles \cite{adams2011large,duong2013wasserstein,adams2013large,mielke2014relation,erbar2015large}. Since such a gradient descent formulation of the Boltzmann equation is already known \cite{erbar2016gradient,bouchet2020boltzmann,basile2021large}, we will not explore this here. \bigskip \\ Finally, let us refer to the recent work of Bouchet \cite{bouchet2020boltzmann} which discusses the classical paradox of reversibility based on large deviations, with a rate function analagous to (\ref{eq: leo rate function}) above and the r\^ole of entropy as a quasipotential driving the Boltzmann dynamics.

\section{Exponential Tightness \& Upper Bound} \label{sec: ETUB} In this section, we will prove Propositon \ref{prop: ET + UB} and the upper bound Theorem \ref{thrm: main positive}i). We first verify exponential tightness in Subsection \ref{subsec: ET}. In Subsection \ref{sec: var form} we introduce a variational form for the rate function, similar to that of L\'eonard \cite[Theorem 3.1, Theorem 7.1]{leonard1995large} and prove equivalence of the two formulations; this leads to a simple proof of lower semicontinuity in Proposition \ref{prop: ET + UB}ii), as well as some functional analytical facts which will be useful later. Finally, we use the variational formulation to  prove the upper bound in Section \ref{subsec: ET}, based on standard martingale techniques and a covering argument.
\subsection{Exponential Tightness} \label{subsec: ET} We first prove that $(\mu^N_\bullet, w^N)$ are exponentially tight in $\D\times\Mm$, which proves the first assertion Proposition \ref{prop: ET + UB}i. \begin{lemma}[Verification of Conditions for Exponential Tightness]\label{lemma: easy UT} For any $M>0$, the following hold.  \begin{enumerate}[label=\alph*).] \item For $\lambda>0$, set $\cp_{2,\lambda}=\{\mu\in \cp_2:\langle |v|^2, \mu\rangle \le \lambda\}$ and $\D_\lambda:=\{\mu_\bullet \in \D: \mu_t\in \cp_{2,\lambda} \text{ for all }t\}$. There exists $\lambda\in (0,\infty)$ such that \begin{equation}\label{eq: tightness at single time} \limsup_N \frac{1}{N}\log \PP\left(\mu^N_\bullet \not \in \D_\lambda \right)\le -M. \end{equation} \item For all $\delta>0$, define $q^N(\delta)=\sup(W(\mu^N_s, \mu^N_t): |s-t|<\delta)$. For all $\epsilon>0$, there exists $\delta>0$ such that \begin{equation} \label{eq: tightness of fluctuations} \limsup_N \frac{1}{N}\log\PP\left(q^N(\delta)>\epsilon\right)\le -M. \end{equation}  \item There exists $C>0$ such that \begin{equation}\label{eq: finite mass for flux} \limsup_N\frac{1}{N}\log \PP\left(w^N(E)>C\right)\le -M. \end{equation} \end{enumerate} \end{lemma} Let us remark that the first item proves that, for each fixed $t\in [0,T]$, $\mu^N_t$ are exponentially tight because $\cp_{2,\lambda}$ are compact for the metric $W$. Together, the first two conditions verify the well-known criteria for exponential tightness in the Skorokhod space $\D$ due to Feng and Kurz \cite[Theorem 4.1]{feng2006large}. In the third item, the sets $\{w\in \Mm: w(E)\le C\}$ are compact for the metric $d$, which induces the weak$^\star$ topology\footnote{I.e. the topology induced by $\langle f, \mu\rangle, f\in C_b(\mathbb{R}^d)$.}, and hence the third item shows that $w^N$ are exponentially tight in $\Mm$. Together, these prove that the pair $(\mu^N_\bullet, w^N)$ together are exponentially tight in $\D\times\Mm$. \begin{proof}[Proof of Lemma \ref{lemma: easy UT}] Fix $M$ throughout. We start with the first point, and begin by noting that, thanks to Hypothesis \ref{hyp: condition on ref ms}i) and a Chebychev bound, for $z_1>0$ sufficiently small and all $\lambda>0$, $$ \PP\left(\langle |v|^2, \mu^N_0\rangle>\lambda\right)\le e^{-N\lambda z_1}\EE\left[e^{Nz_1 \langle |v|^2, \mu^N_0\rangle}\right] = \exp\left(-N(\lambda z_1-\log \mathcal{E}_{z_1}(\mu^\star_0))\right) $$ where, in the right-hand side, we recall that $\mu^N_0$ is given by sampling particles independently from $\mu^\star_0$. Taking $\lambda_M=z_1^{-1}(M+\log \mathcal{E}_{z_1}(\mu^\star_0))$, we conclude that $$\PP\left(\mu^N_0\not\in \cp_{2,\lambda_M}\right)=\PP\left(\langle |v|^2, \mu^N_0\rangle > \lambda_M\right)\le e^{-MN}. $$ To extend this to the whole process we note that the kinetic energy $\langle |v|^2, \mu^N_t\rangle$ is constant in time, so that $\mu^N_t\in \D_{\lambda_M}$ if, and only if, $\mu^N_0\in \cp_{2,\lambda_M}$. Therefore, for any $t\in [0,T]$, $$ \PP\left(\mu^N_\bullet\not\in \D_{\lambda_M}\right)=\PP\left(\mu^N_0\not\in \cp_{2,\lambda_M}\right)\le e^{-MN}$$ and the first item follows. For the second item, we observe that the instantaneous rate of the Kac process, in either the hard spheres or Maxwell molecules case, is bounded by \begin{equation} \begin{split} N\int_{\rrd\times\rrd\times \ssd} B(v-v_\star)\mu^N_t(dv)\mu^N_t(dv_\star)d\sigma &   \le N\int_{\rrd\times\rrd} (3+|v|^2+|v_\star|^2)\mu^N_t(dv)\mu^N_t(dv_\star) \\[1ex] & \le 3N(1+\langle |v|^2, \mu^N_0\rangle)\end{split}\end{equation} where we note that, for either kernel, $ B(|v-v_\star|)\le 1+|v|+|v_\star|\le 3+|v|^2+|v_\star|^2$, and in the final inequality, we recall again that the second moment $\langle |v|^2, \mu^N_t\rangle$ is independent of time. It therefore follows that we can construct a time-homogenous Poisson process $\widetilde{w}^N_t$, of constant, random rate $3N(1+\langle |v|^2, \mu^N_0\rangle)$, such that $\widetilde{w}^N_t$ has jumps on a superset of the times when $\mu^N_t$ jumps. This leads to the bound, for any $s\le t$, $$(w^N_t-w^N_s)([0,T]\times \rrd\times\rrd\times\ssd)\le \frac{1}{N}(\widetilde{w}^N_t-\widetilde{w}^N_s).$$ For each $\delta$, we now pick a partition $0=t_0<t_1...<t_m$ of size $\lceil T/\delta\rceil$ by taking constant steps of size $\delta$. We now observe that $W(\mu^N_t,\mu^N_{t-})\le 4/N$ at times when $\mu^N_t$ jumps, and for any $|s-t|\le \delta$, the interval $[s,t]$ is contained in at most two adjacent intervals $[t_{i-1},t_{i+1}]$. Together, we conclude that  \begin{equation}\label{eq: dominate mod con by pp} q^N(\delta)\le 8\max_i \frac{1}{N}(\widetilde{w}^N_{t_i}-\widetilde{w}^N_{t_{i-1}}). \end{equation} For $\lambda_M$ as above and for any $z>0$, we bound \begin{equation} \begin{split}\label{eq: fluctuations in each interval} \PP\left(\frac{1}{N}(\widetilde{w}^N_{t_i}-\widetilde{w}^N_{t_{i-1}})>\frac{\epsilon}{8}\h\bigg|\h \mu^N_0\in \cp_{2,\lambda_M} \right) &\le e^{-z N\epsilon/8} \EE\left[e^{z (\widetilde{w}^N_{t_i}-\widetilde{w}^N_{t_{i-1}})}\h\bigg|\h \mu^N_0\in \cp_{2,\lambda_M}\right] \\& \le \exp\left(-N\left(\frac{z \epsilon}{8} - 3(1+\lambda_M) \delta (e^z-1) \right)\right) \end{split} \end{equation} where, in the second line, we use the bound that the rate of $\widetilde{w}^N$ is at most $3N(1+\lambda_M)$ if $\mu^N_0\in \cp_{2,\lambda_M}$. We now choose $z=8(M+1)/\epsilon$, and $\delta>0$ small enough, depending on $ z, \lambda_M$, so that $3(1+\lambda_M)\delta(e^z-1)<1$. For this choice of $\delta$, the final expression in (\ref{eq: fluctuations in each interval}) is $e^{-NM}$, for each interval. Finally, we take a union bound: $$ \{q^N(\delta)>\epsilon\}\subset \{\mu^N_0\not \in \cp_{2,\lambda_M}\}\cup \bigcup_{i\le \lceil T/\delta\rceil} \left\{\widetilde{w}^N_{t_i}-\widetilde{w}^N_{t_{i-1}}>\frac{N\epsilon}{8}, \h \mu^N_0\in \cp_{2,\lambda_M}\right\}.$$ By the choices of $\lambda_M$ and $\delta$, $$\PP\left(q^N(\delta)>\epsilon\right)\le (1+\lceil T/\delta\rceil)e^{-NM} $$ and the second item now follows. The final item follows in exactly the same way: following (\ref{eq: fluctuations in each interval}), for all $C>0$ we bound \begin{equation} \begin{split} \PP\left(w^N(E)>C\big|\mu^N_0\in \cp_{2,\lambda_M}\right) & \le \PP\left(\widetilde{w}^N_T > CN\big|\mu^N_0\in \cp_{2,\lambda_M}\right) \\[1ex]& \le e^{-CN}\EE\left[e^{\widetilde{w}^N_T}\big|\mu^N_0\in \cp_{2,\lambda_M} \right]\\[1ex] & \le \exp\bigg(-N\big(C  - 3(1+\lambda_M)T(e-1)\big)\bigg)  \end{split}  \end{equation} and choosing $C=M+3(1+\lambda_M)T(e-1)$ makes the final probability at most $e^{-MN}$. Using a union bound, \begin{equation} \PP\left(w^N(E)>C\right) \le \PP\left(w^N(E)>C\big| \mu^N_0\in \cp_{2,\lambda_M}\right) + \PP\left(\mu^N_0 \not \in \cp_{2,\lambda_M}\right) \le 2e^{-MN} \end{equation} from which (\ref{eq: finite mass for flux}) follows.  \end{proof} Let us also record, for later use, the following corollary. \begin{corollary}\label{cor: tightness} Let $\PPm^N \ll \PP$ be changes of measure such that \begin{equation}\label{eq: condition for tightness} \lim_{a\to\infty} \liminf_N\h \PPm^N\left(\frac{d\PPm^N}{d\PP} \le e^{Na}\right)=1. \end{equation} Then the laws of $(\mu^N_\bullet, w^N)$ under $\PPm^N$ are tight: for all $\epsilon>0$ there exists a compact set $\K\subset\D\times\Mm$ such that \begin{equation} \sup_N\h\PPm^N\left((\mu^N_\bullet, w^N)\not\in \K\right)<\epsilon.\end{equation} \end{corollary} \begin{proof} This follows from Proposition \ref{prop: ET + UB}i) and the hypothesis (\ref{eq: condition for tightness}) by purely general considerations. Let us fix $\epsilon>0$; thanks to (\ref{eq: condition for tightness}) we can choose $a$ such that, for all but finitely many $N$, \begin{equation}\label{eq: for all N bound} \PPm^N\left(\frac{d\PPm^N}{d\PP}> e^{Na}\right)<\frac{\epsilon}{2} \end{equation} and, changing $a$ if necessary, we can arrange that (\ref{eq: for all N bound}) holds for all $N$. We now choose $M=a-\log (\epsilon/2)$; by Proposition \ref{prop: ET + UB}i), there exists a compact set $\K$ such that \begin{equation} \limsup_N N^{-1}\log \PP((\mu^N_\bullet, w^N)\not\in \K)<-M\end{equation}  which implies that, for all $N$ sufficiently large, \begin{equation} \label{eq: for all N bound 2}\PP\left((\mu^N_\bullet, w^N)\not\in \K\right) \le e^{-MN}.\end{equation} Since the space $\D\times\Mm$ is a separable metric space, each $(\mu^N_\bullet, w^N)$ is tight so $\K$ can be replaced with a larger compact set such that (\ref{eq: for all N bound 2}) again holds for all $N$. Together, (\ref{eq: for all N bound}, \ref{eq: for all N bound 2}) imply that \begin{equation} \begin{split} \PPm^N\left((\mu^N_\bullet, w^N)\not \in \K\right) & \le \PPm^N\left(\frac{d\PPm^N}{d\PP}\le e^{Na}, (\mu^N_\bullet, w^N)\not\in \K\right)+\PPm^N\left(\frac{d\PPm^N}{d\PP}>e^{Na}\right) \\[1ex]& < e^{Na} \PP\left((\mu^N_\bullet, w^N)\not\in \K\right) + \frac{\epsilon}{2} \le e^{Na}e^{-N(a+\log (\epsilon/2))}+ \frac{\epsilon}{2}\le \epsilon\end{split}  \end{equation}  and we are done.  \end{proof}
\subsection{Variational Formulation of the Rate Function}\label{sec: var form} In preparation for the upper bound, we will now present a variational formulation of the rate function. This will also allow us to prove the lower semicontinuity in Proposition \ref{prop: ET + UB}. We are aided in this equivalence by the inclusion of the flux in the large deviation principle: the choice of $K$, if it exists, is unique, which allows us to significantly simplify the proof of L\'eonard \cite{leonard1995large}.\bigskip \\ We begin with the following construction. Let us write $C^{1,1}_{0,b}([0,T]\times\rrd)$ for those functions $f: [0,T]\times\rrd\to \rr$ which are bounded and Lipschitz in the $v$-variable, with a bounded and $v$-Lipschitz time derivative $\partial_t f$, and such that $f_0=0$. For $\varphi \in C_b(\rrd), f\in C^{1,1}_{0,b}([0,T]\times \rrd)$ and $g\in C_c(E)$ and $t\in [0,T]$, we define \begin{equation} \Xi(\mu_\bullet, w,\varphi, f,g)_t= \Xi_0(\mu_\bullet, \varphi)+ \Xi_{1,t}(\mu_\bullet, w,f) + \Xi_{2,t}(\mu_\bullet, w,g) \end{equation} where \begin{equation} \Xi_0(\mu_\bullet, \varphi)= \langle \varphi, \mu_0\rangle - \log \langle e^\varphi, \mu^\star_0\rangle ;\end{equation} \begin{equation}\begin{split} \Xi_{1,t}(\mu_\bullet, w,f):&= \langle f_t, \mu_t\rangle -\int_0^t  \langle \partial_s f_s, \mu_s\rangle ds \\& \hs - \int_{(0,t]\times\rrd\times\rrd\times\ssd} \Delta f(s,v,v_\star,\sigma) w(ds,dv,dv_\star,d\sigma)    \end{split}\end{equation} and \begin{equation}\begin{split} \Xi_{2,t}(\mu_\bullet, w,g)&:= \int_{E} 1(s\le t)(g(s,v,v_\star,\sigma) w(ds,dv,dv_\star,d\sigma) \\ & \hs -\int_E(e^g-1)(s,v,v_\star,\sigma) \overline{m}_{\mu}(ds,dv,dv_\star,d\sigma)).\end{split} \end{equation} We write $\Xi(\mu_\bullet, w,\varphi, f,g)$ for the terminal value $\Xi(\mu_\bullet, w,\varphi, f,g)=\Xi(\mu_\bullet, w,\varphi,f,g)_T$. Let us note that these processes make sense at the level of the particle system $\mu^N_\bullet, w^N$. The function $f$ here entering into $\Xi_{1,t}(\mu_\bullet, w,f)$ will play the r\^ole of a Lagrange multiplier to enforce the constraint of the continuity equation. This is made precise by the following result. \begin{lemma}\label{lemma: lagrange multiplier}  Fix $(\mu_\bullet, w)\in \D\times\Mm$. Then \begin{equation}\sup\left\{\Xi_{1,T}(\mu_\bullet, w, f): f\in C^{1,1}_{0,b}([0,T]\times \rrd)\right\}=\begin{cases} 0 &\text{if }(\mu_\bullet, w)\text{ solves (\ref{eq: CE})}; \\ \infty & \text{else.} \end{cases} \end{equation} \end{lemma} \begin{proof} For the case where $(\mu_\bullet, w)$ solves (\ref{eq: CE}), we will show that, for all $f\in C^{1,1}_{0,b}([0,T]\times\rrd)$ and all $t\in [0,T]$, we have the time-dependent equivalent of (\ref{eq: CE}): \begin{equation}\label{eq: time dependent CE} \langle f, \mu_t\rangle = \int_0^t \langle \partial_s f_s,\mu_s\rangle ds+\int_E 1(s\le t)\Delta f(s,v,v_\star,\sigma)w(ds,dv,dv_\star,d\sigma). \end{equation} This will immediately imply that $\Xi_{1,t}(\mu_\bullet, w,f)=0$ for all $f$ and all $t$, which implies the claim. The proof of this formulation is slightly complicated by the lack of regularity, since we only assume \emph{a priori} that $\mu_\bullet$ is c{\`a}dl{\`a}g rather than continuous; we will instead use the facts about c{\`a}dl{\`a}g paths from Proposition \ref{prop: Sk continuity}. Since (\ref{eq: time dependent CE}) is linear in $f$, we can assume that $f_t, \partial_t f_t$ belong to the class $\mathcal{F}$ which are $1$-bounded and $1$-Lipschitz. Fix $t\in [0,T]$, $\epsilon>0$, $\delta>0$. \bigskip \\ Let us write $L\subset [0,T]$ for those $s\in [0,T]$ with $W(\mu_{s-}, \mu_s)\ge\epsilon$; thanks to Proposition \ref{prop: Sk continuity}a), $L$ is finite, and we write $m=|L|<\infty$ for its cardinality. Possibly making $\delta>0$ smaller, Proposition \ref{prop: Sk continuity}b) guarantees that $\delta$ can be chosen so that any interval $[u,v)$ of length $\delta$ either contains a point of $L$, or for all $s\in [u,v)$, $W(\mu_s, \mu_u)<\epsilon$. Now, for such $\delta$, we decompose $(0,t]$ into intervals $(t_i, t_{i+1}]$ of length at most $\delta$, and add \begin{equation} \langle f_{t_{i+1}}, \mu_{t_{i+1}}-\mu_{t_i}\rangle = \int_E 1(s\in (t_i, t_{i+1}]) \Delta f(t_{i+1},v,v_\star,\sigma)w(ds,dv,dv_\star,d\sigma);\end{equation}\begin{equation} \langle f_{t_{i+1}}-f_{t_i}, \mu_{t_{i}}\rangle =\int_{t_i}^{t_{i+1}} \langle \partial_s f_s, \mu_{t_{i}}\rangle ds  \end{equation} to obtain \begin{equation}\begin{split} \langle f_{t_{i+1}}, \mu_{t_{i+1}}\rangle -\langle f_{t_i}, \mu_{t_i}\rangle &= \int_{t_i}^{t_{i+1}} \langle \partial_s f_s, \mu_{t_{i}}\rangle ds \\ &  +\int_E 1(s\in (t_i, t_{i+1}]) \Delta f(t_{i+1},v,v_\star,\sigma)w(ds,dv,dv_\star,d\sigma). \end{split} \end{equation} We approximate the two terms by \begin{equation} \begin{split}  \left|\int_{t_i}^{t_{i+1}} \langle \partial_s f_s, \mu_{t_{i}}\rangle ds -\int_{t_i}^{t_{i+1}} \langle \partial_s f_s, \mu_s\rangle ds \right| & \le \int_{t_i}^{t_{i+1}} W(\mu_s, \mu_{t_i})ds \\[1ex] &\le \epsilon(t_{i+1}-t_i)+2\delta \cdot 1(L\cap [t_i, t_{i+1})\neq \emptyset) \end{split} \end{equation} since, by the choice of $\delta$, either $W(\mu_s, \mu_{t_i})\le \epsilon$ for all $s\in [t_i, t_{i+1}]$, or there is a point of $L$ in $[t_i, t_{i+1})$, in which case we use the trivial bound $W(\mu_s, \mu_{t_i})\le 2$ and recall that the interval is of length at most $\delta$. For the second term \begin{equation}\begin{split}&\left|\int_E 1(s\in (t_i, t_{i+1}]) \left(\Delta f(t_{i+1},v,v_\star,\sigma)-\Delta f(s,v,v_\star,\sigma)\right)w(ds,dv,dv_\star,d\sigma)\right|    \\ & \hs \hs \hs \hs   \le \left|\int_E 1(s\in (t_i, t_{i+1}])4\|f_{t_{i+1}}-f_s\|_\infty w(ds,dv,dv_\star,d\sigma)\right| \\[1ex] & \hspace{4cm}  \le 4\delta w\left((t_i, t_{i+1}]\times\rrd\times\rrd\times\ssd\right)\end{split}  \end{equation}  where in the final line we recall that we have scaled so that $\|\partial_t f_t\|_\infty \le 1$. Adding, we conclude that \begin{equation} \begin{split}  &\bigg|\langle f_{t_{i+1}}, \mu_{t_{i+1}}\rangle -\langle f_{t_i}, \mu_{t_i}\rangle - \int_{t_i}^{t_{i+1}} \langle \partial_s f_s, \mu_s\rangle \\ & \hspace{4cm}-\int_E 1(s\in (t_i, t_{i+1}])\Delta f(s,v,v_\star,\sigma)w(ds,dv,dv_\star,d\sigma) \bigg|  \\ &  \le \epsilon(t_{i+1}-t_i)+2\delta\cdot1\left(L\cap [t_i, t_{i+1})\neq\emptyset \right)+4\delta w\left((t_i, t_{i+1}]\times\rrd\times\rrd\times\ssd\right).\end{split} \end{equation}Summing over all such intervals $(t_i, t_{i+1}]$ covering $(0,t]$, and recalling that $f_0\equiv 0$, we obtain \begin{equation} \begin{split} &\left|\langle f_t, \mu_t\rangle -\int_0^t \langle \partial_s f_s, \mu_s\rangle ds -\int_E 1(s\le t)\Delta f(s,v,v_\star,\sigma)w(ds,dv,dv_\star,d\sigma)\right| \\ & \hspace{3cm}\le \epsilon t+2m\delta+4\delta w(E) \end{split} \end{equation} and the right-hand side can be made arbitrarily small by taking $\epsilon, \delta\to 0$, recalling that $w$ is a finite measure by hypothesis, so the claim (\ref{eq: time dependent CE}) is proven, and we conclude that $\sup_f \Xi_{1,T}(\mu_\bullet,w,f)=0$ as claimed. Otherwise, if (\ref{eq: CE}) fails, there exists some $t_0\in (0,T]$ and some $g\in C^\infty_c(\rrd)$ such that \begin{equation} \label{eq: CE fails} \langle g, \mu_{t_0}\rangle -\langle g, \mu_0\rangle -\int_E 1(s\le t_0) \Delta g(v,v_\star,\sigma)w(ds,dv,dv_\star,d\sigma) > 1. \end{equation}  Let us assume that $t_0\in (0,T)$; the proof $t_0=T$ is similar and strictly easier. We now fix a smooth, increasing function $\chi:\rr\to \rr$ such that $\chi=0$ on $(-\infty, 0]$ and $\chi=1$ on $[1,\infty)$, and for $0<\delta<\min(t_0, T-t_0)$, we construct $f^\delta \in C^{1,1}_{0,b}([0,T]\times\rrd)$ by defining \begin{equation} f^\delta_t(v):=\begin{cases} \chi(t/\delta)g(v) & \text{if }t\in [0, \delta]; \\ g(v) & \text{if }\delta<t\le t_0; \\ \chi(1-(t-t_0)/\delta)g(v) & \text{if }t_0<t<t_0+\delta \\ 0 &\text{else}.\end{cases}\end{equation} Thanks to right-continuity of $\mu_t$, we observe that $\int_0^T \langle \partial_t f^\delta_t, \mu_t\rangle dt \to \langle g, \mu_0\rangle - \langle g, \mu_{t_0}\rangle$, and using dominated convergence, \begin{equation} \int_E \Delta f^\delta(s,v,v_\star,\sigma)w(ds,dv,dv_\star,d\sigma)\to \int_{E} 1(s\le t_0) \Delta g(v,v_\star,\sigma)w(ds,dv,dv_\star,d\sigma).\end{equation}  Therefore, $\Xi_{1,T}(\mu_\bullet, w,f^\delta)$ converges to \begin{equation} \Xi_{1,T}(\mu_\bullet, w,f^\delta)\to \langle g, \mu_{t_0}\rangle - \langle g, \mu_0\rangle - \int_E 1(s\le t_0)\Delta g(v,v_\star,\sigma)w(ds,dv,dv_\star,d\sigma) \end{equation} and in particular, we can choose $\delta>0$ small enough that $\Xi_{1,T}(\mu_\bullet, w,f^\delta)>1$. By linearity, for all $\lambda>0$, $\Xi_{1,T}(\mu_\bullet, w, \lambda f^\delta)>\lambda$, and so the supremum is infinite, as claimed. \end{proof} We now use this equality to show that the functions $\Xi$ above give a variational formulation of the rate function $\mathcal{I}$ given in the introduction. \begin{lemma}\label{lemma: variational form of RF} For $(\mu_\bullet, w)\in \D\times\Mm$, we have \begin{equation} \mathcal{I}(\mu_\bullet, w)=\sup\left\{\Xi(\mu_\bullet, w,\varphi, f,g): \varphi \in C_b(\rrd), f\in C^{1,1}_{0,b}([0,T]\times\rrd),g\in C_c(E) \right\}. \end{equation} \end{lemma} \begin{proof} Let us write $\widetilde{\mathcal{I}}$ for the right-hand side. Since $\Xi_{0}$ depends only on $\varphi$, $\Xi_1$ only on $f$ and $\Xi_2$ only on $g$, the supremum decomposes as \begin{equation} \label{eq: decomposition of var form} \widetilde{\mathcal{I}}(\mu_\bullet, w)=\sup_{\varphi} \Xi_0(\mu_\bullet, \varphi) +\sup_f \Xi_{1,T}(\mu_\bullet, w,f)+\sup_g \Xi_{2,T}(\mu_\bullet, w,g) \end{equation} where the suprema run over the same sets as above. Optimising over $\varphi $ produces the well-known variational formulation $\sup_{\varphi}\Xi_0(\mu_0, \varphi)=H(\mu_0|\mu_0^\star)$ of the relative entropy. This identity can be found in \cite[Appendix 1]{kipnis1998scaling}, or derived using essentially the same argument as for $\Xi_{2,T}$ below. Thanks to Lemma \ref{lemma: lagrange multiplier}, the supremum over $f$ is infinite unless the continuity equation (\ref{eq: CE}) holds, in which case this term vanishes. \bigskip\\ We now deal with the third term. If $w\not \ll \overline{m}_{\mu}$, there is a compact set $E'\subset E$ with $w(E')>0$ but $\overline{m}_{\mu}(E')=0$, and by since $E$ is a metric space and $w$ is a Borel measure, we can find open $U_n \downarrow E'$ and closed $U_n\supset A_n\supset E'$ with $\overline{m}_{\mu}(U_n)\downarrow 0$. We now choose $g_n \in C_c(E)$ so that $0\le g_n\le 1$, $g_n=1$ on $A_n$, and $=0$ except on $U_n$, and bound for $\lambda>0$, \begin{equation} \int_E \lambda g_n w(ds,dv,dv_\star,d\sigma) \ge \lambda w(E'); \end{equation}\begin{equation}\int_E (e^{\lambda g_n}-1)\overline{m}_{\mu}(ds,dv,dv_\star,d\sigma)\le (e^\lambda -1)\overline{m}_{\mu}(U_n) \end{equation} so that \begin{equation} \Xi_{2,T}(\mu_\bullet, w, \lambda g_n) \ge \lambda w(E') -(e^{\lambda}-1)\overline{m}_{\mu}(U_n).\end{equation} By taking $\lambda=\lambda_n \to \infty$ slowly enough, the right-hand side can be made arbitrarily large as $n\to \infty$, so in this case $\sup_g \Xi_{2,T}(\mu_\bullet, w,g)=\infty$. On the other hand, if $w\ll \overline{m}_\mu$, let us write $K$ for the tilting function $\frac{dw}{d\overline{m}_{\mu}}$, so that \begin{equation} \Xi_{2,T}(\mu_\bullet, w,g)=\int_E (Kg - e^g +1)(s,v,v_\star,\sigma)\overline{m}_{\mu}(ds,dv,dv_\star,d\sigma). \end{equation} Observing that, for all $x \in \rr, y\ge 0$, it holds that $xy\le (e^x-1)+(y\ln y - y+ 1)=(e^x-1)+\tau(y)$, the first term $\int Kg$ can be bounded by\begin{equation} \begin{split} \int_E Kg(s,v,v_\star,\sigma)\overline{m}_{\mu}(ds,dv,dv_\star,d\sigma)& \le \int_E \tau(K)\overline{m}_{\mu}(ds,dv,dv_\star,d\sigma) \\& + \int_E (e^g-1)\overline{m}_{\mu}(ds,dv,dv_\star,d\sigma)\end{split} \end{equation} which leads to the bound, uniformly in $g\in C_c(E)$, \begin{equation}\label{eq: upper bound for xi2} \Xi_{2,T}(\mu_\bullet, w,g)\le \int_E \tau(K(s,v,v_\star,\sigma))\overline{m}_{\mu}(ds,dv,dv_\star,d\sigma)\end{equation} whether or not the right-hand side is finite. On the other hand, let us fix $M$. By Lusin's theorem, we can construct continuous, bounded $g_n\in C_c(E)$ with $g_n\to \ln K \land M$ for $\overline{m}_{\mu}$-almost all $(t,v,v_\star,\sigma)$, so that \begin{equation} Kg_n-e^{g_n}+1\le MK+1\end{equation} and \begin{equation} Kg_n-e^{g_n}+1 \to K(\ln K \land M)-(K\land e^M)  + 1\end{equation} $ \overline{m}_\mu(ds,dv,dv_\star,d\sigma)$-almost everywhere. Since $K\in L^1(\overline{m}_\mu)$, we can use dominated convergence to obtain \begin{equation} \Xi_{2,T}(\mu_\bullet, w,g_n)\to \int_E (K(\ln K\land M)-(K\land e^M)+1)\overline{m}_{\mu}(ds,dv,dv_\star,d\sigma)\end{equation} and the supremum is at least the right-hand side. The integrand is increasing in $M$, and converges to $\tau(K)$ pointwise, so the whole integral converges to $\int_E \tau(K)d\overline{m}_\mu$. We conclude that \begin{equation}\sup\left\{\Xi_{2,T}(\mu_\bullet, w,g): g\in C_c(E)\right\} \ge \int_E \tau(K(s,v,v_\star,\sigma))\overline{m}_{\mu}(ds,dv,dv_\star,d\sigma) \end{equation} and (\ref{eq: upper bound for xi2}) shows that this is an equality. Putting everything together, we have shown that \begin{equation} \label{eq: variational J}\begin{split} & \sup_{f,g}\left\{\Xi_{1,T}(\mu_\bullet,w,f)+\Xi_{2,T}(\mu_\bullet, w,g)\right\}\\& \hspace{2cm}=\begin{cases} \int_E \tau(K)\overline{m}_{\mu}(ds,dv,dv_\star,d\sigma)& \text{if } (\mu_\bullet, w) \text{is a measure flux pair;} \\ \infty &\text{else}\end{cases}\end{split} \end{equation} and the right-hand side is exactly the definition of $\mathcal{J}(\mu_\bullet, w)$; returning to (\ref{eq: decomposition of var form}), we have proven that \begin{equation} \widetilde{\mathcal{I}}(\mu_\bullet, w)=H(\mu_0|\mu^\star_0)+\mathcal{J}(\mu_\bullet, w) = \mathcal{I}(\mu_\bullet, w) \end{equation}  as desired.\end{proof} Thanks to this variational form, we readily obtain the lower semi-continuity claimed in Proposition \ref{prop: ET + UB}. We first record, as a separate lemma, a result which will be helpful later. \begin{lemma}\label{lemma: really useful continuity result} For fixed $f\in L^\infty([0,T],C_b(\rrd))$ and $g\in C_c(E)$, the maps \begin{equation}   \mu_\bullet \mapsto \int_0^T \langle f_t, \mu_t\rangle dt;\qquad \mu_\bullet \mapsto \int_E g \overline{m}_{\mu}(ds,dv,dv_\star,d\sigma)\end{equation} are continuous in the topology of $\D$. \end{lemma} \begin{proof} Noting that the topology of $\D$ is induced by a metric, it is sufficient to prove sequential continuity: let us fix $\mu^{(n)}_\bullet \to \mu_\bullet$. By Proposition \ref{prop: Sk continuous functions}, it follows that $W(\mu^{(n)}_t, \mu_t)\to 0$ for $dt$-almost all $t$, and for all such $t$, we also have the weak convergence $\mu^{(n)}_t\otimes \mu^{(n)}_t\to \mu_t\otimes \mu_t$. Since $g$ has compact support in $E$, for any fixed $\sigma$ and for such $t$, the map $(v,v_\star)\mapsto g(t,v,v_\star,\sigma)B(v-v_\star)$ is bounded and continuous, and so we have the convergences \begin{equation} \langle  f_t, \mu^{(n)}_t\rangle \to \langle  f_t, \mu_t\rangle; \end{equation} \begin{equation} \begin{split} &\int_{\rrd\times\rrd}  g(t,v,v_\star,\sigma)B(v-v_\star)\mu^{(n)}_t(dv)\mu^{(n)}_t(dv_\star)\\& \hspace{5cm}\to \int_{\rrd\times\rrd}  g(t,v,v_\star,\sigma)B(v-v_\star)\mu_t(dv)\mu_t(dv_\star).\end{split}\end{equation} Since these hold for all $\sigma$ and $dt$-almost all $t$, we can integrate and use bounded convergence to find that \begin{equation} \int_0^T \langle f_t, \mu^{(n)}_t\rangle dt \to \int_0^T \langle f_t, \mu_t\rangle dt; \end{equation} and \begin{equation} \int_E g\overline{m}_{\mu^{(n)}}(ds,dv,dv_\star,d\sigma)\to \int_E g\overline{m}_{\mu}(ds,dv,dv_\star,d\sigma)  \end{equation} and we are done.   \end{proof}  \begin{lemma}\label{lemma: semiconntinuity} For fixed $\varphi\in C_b(\rrd), f\in C^{1,1}_{0,b}([0,T]\times\rrd)$ and $g\in C_c(\rrd)$, the maps \begin{equation} (\mu_\bullet, w)\to \Xi_{0}(\mu_\bullet, \varphi);\qquad (\mu_\bullet, w)\to \Xi_{1,T}(\mu_\bullet, \varphi);\qquad (\mu_\bullet, w)\to \Xi_{2,T}(\mu_\bullet, \varphi); \end{equation}\begin{equation} (\mu_\bullet, w)\to \Xi(\mu_\bullet, w,\varphi,f,g) \end{equation} are continuous for the topology of $\D\times\Mm$. In particular, the sub-level sets $\{\mathcal{J}\le a\}, \{\mathcal{I}\le a\}$ are closed in $\D\times\Mm$ for all $a\in [0,\infty)$, as is the set of pairs $(\mu_\bullet, w)$ for which (\ref{eq: CE}) holds, and $\{\mu\in \cp_2: H(\mu|\mu_0^\star)\le a\}$ in the topology of $(\cp_2, W)$. \end{lemma} \begin{proof} With the choices of topologies on $\D, \Mm$, the maps $$  w\mapsto \int_E g w(ds,dv,dv_\star,d\sigma) ;\qquad w\mapsto \int_E \Delta f(s,v,v_\star,\sigma)w(ds,dv,dv_\star,d\sigma)$$ are immediately continuous, and thanks to Proposition \ref{prop: Sk continuous functions}a), so are $\mu_\bullet \mapsto \langle \varphi, \mu_0\rangle; \mu_\bullet \mapsto \langle f_T, \mu_T\rangle$. Combining with Lemma \ref{lemma: really useful continuity result}, with $f_t$ replaced by $\partial_t f_t$, each expression appearing in the definitions of $\Xi_{i,T}$ is continuous, and we conclude the claimed continuity of the stated maps. For the second point, we use  Lemma \ref{lemma: variational form of RF} to write the sublevel sets, for any $a\in [0,\infty]$, as \begin{equation} \{\mathcal{I}\le a\}=\bigcap_{\varphi\in C_b(\rrd), f\in C^{1,1}_{0,b}([0,T]\times\rrd),g\in C_c(E)}\left\{(\mu_\bullet, w): \Xi(\mu_\bullet, w, \varphi,f,g)_T\le a\right\}.\end{equation} Each set in the intersection is closed, and hence so is the left-hand side, which proves lower semi-continuity; the assertions for $\mathcal{J}$ and $H(\cdot|\mu^\star_0)$ are identical, recalling (\ref{eq: variational J}) and that $H(\mu|\mu^\star_0)=\sup_\varphi \Xi_0(\varphi,\mu_0)$. The remaining assertion is similar: using Lemma \ref{lemma: lagrange multiplier}, \begin{equation} \left\{(\mu_\bullet, w): \text{(\ref{eq: CE}) holds}\right\}=\bigcap_{f\in C^{1,1}_{0,b}([0,T]\times\rrd)} \left\{(\mu_\bullet, w): \Xi_{1,T}(\mu_\bullet, w,f)=0\right\}\end{equation} which is an intersection of closed sets, and hence closed. \end{proof}

\subsection{Upper Bound} \label{subsec: UB} Using the variational formulation above, we now prove the upper bound in Theorem \ref{thrm: main positive}. We begin with a local version of the result. \begin{lemma}\label{lemma: local upper bound} Fix $(\mu_\bullet, w)\in \D\times\Mm$ with finite rate $\mathcal{I}(\mu_\bullet, w)<\infty$, and fix $\epsilon>0$. Then there exists an open set $\U\ni (\mu_\bullet, w)$ such that  \begin{equation} \limsup_N \frac{1}{N}\log \PP\left((\mu^N_\bullet, w^N)\in \U\right)\le -\mathcal{I}(\mu_\bullet, w)+\epsilon.\end{equation} If instead $\mathcal{I}(\mu_\bullet, w)=\infty$ and $M<\infty$ then there exists an open set $\U\ni (\mu_\bullet, w)$ such that \begin{equation}  \limsup_N \frac{1}{N}\log \PP\left((\mu^N_\bullet, w^N)\in \U\right)\le -M.\end{equation} \end{lemma} \begin{proof} Let us consider the first case; the second is essentially identical. Thanks to Lemma \ref{lemma: variational form of RF}, we can choose $\varphi\in C_b(\rrd),f\in C^{1,1}_{0,b}([0,T]\times\rrd)$ and $g\in C_c(E)$ such that \begin{equation} \mathcal{I}(\mu_\bullet, w)<\Xi(\mu_\bullet,w, \varphi,f,g) +\frac{\epsilon}{2}\end{equation} and, thanks to Lemma \ref{lemma: semiconntinuity}, we can find open $\U \ni (\mu_\bullet, w)$ such that, for all $(\mu'_\bullet, w')\in \U$, we have \begin{equation} \Xi(\mu'_\bullet,w' \varphi,f,g)  > \Xi(\mu_\bullet,w,\varphi,f,g) -\frac{\epsilon}{2} >\mathcal{I}(\mu_\bullet, w)-\epsilon. \end{equation}   We consider the processes \begin{equation} Z^N_t:=\exp\left(N\Xi(\mu^N_\bullet, w,\varphi,f,g)_t\right).\end{equation} We first observe that, since $\mu^N_\bullet, w^N$ satisfy the continuity equation (\ref{eq: CE}), Lemma \ref{lemma: lagrange multiplier} shows that, for all $t\ge 0$, \begin{equation} \Xi_{1,t}(\mu^N_\bullet, w^N, \varphi,f,g)=0. \end{equation}Next, we show that $Z^N$ is a martingale, following arguments of \cite{darling2008differential}. We observe that at points $(t,v,v_\star,\sigma)$ of $w^N$, $Z^N_t$ jumps by \begin{equation} Z^N_t-Z^N_{t-}=Z^N_{t-}\left(e^{g(t,v,v_\star,\sigma)}-1\right) \end{equation} while between jumps, $Z^N_t$ is differentiable and \begin{equation} \partial_t Z^N_t=-N\int_{\rrd\times\rrd\times\ssd} Z^N_t\left(e^{g(t,v,v_\star,\sigma)}-1\right)B(v-v_\star)\mu^N_t(dv)\mu^N_t(dv_\star)d\sigma.\end{equation} Together, $Z^N_t$ admits the representation \begin{equation} Z^N_t=Z^N_0+N\int_{(0,t]\times\rrd\times\rrd \times\ssd} Z^N_{s-}\left(e^{g(s,v,v_\star,\sigma)}-1\right)(w^N-\overline{m}_{\mu^N})(ds,dv,dv_\star,d\sigma). \end{equation} Recalling the generator (\ref{eq: generator}), $Z^N_t$ is a local martingale, and since it is clearly positive, a supermartingale, and at time $0$, \begin{equation} \EE\left[e^{N\Xi_0(\mu_\bullet, \varphi)}\right]= \frac{\mathbb{E}\left[e^{N\langle \varphi, \mu^N_0\rangle}\right]}{\langle e^\varphi, \mu_0^\star\rangle ^N} =1  \end{equation} where we recall that $\mu^N_0$ is formed by independent samples from $\mu_0^\star$. We now take the expectation of \begin{equation}\begin{split} 1\left((\mu^N_\bullet, w^N)\in \U\right) & \le Z^N_T \exp\left(-N\inf\left\{\Xi(\mu_\bullet, w,\varphi,f,g): (\mu_\bullet, w)\in \U\right\}\right) \\& \le Z^N_T\exp\left(-N(\mathcal{I}(\mu_\bullet, w)-\epsilon)\right)\end{split} \end{equation} to obtain \begin{equation} \begin{split} \PP\left((\mu^N_\bullet, w^N)\in \U\right)& \le \EE[Z^N_T] \exp\left(-N(\mathcal{I}(\mu_\bullet, w)-\epsilon)\right) \\ & \le \exp\left(-N(\mathcal{I}(\mu_\bullet, w)-\epsilon)\right)\end{split} \end{equation} to produce the desired result. The case where $\mathcal{I}(\mu_\bullet,w)=\infty$ is essentially identical. \end{proof}  We now give the proof of the global upper bound. \begin{proof}[Proof of Theorem \ref{thm: main}i)] Let $\A$ be any closed subset of $\D\times\Mm$ and fix $\epsilon\in (0,1]$. Let us assume that $\A$ is nonempty, and that $\inf_\A\mathcal{I}<\infty$. Choosing $M=\inf_\A\mathcal{I}+1$, by Proposition \ref{prop: ET + UB}i) there exists a compact set $\K\subset \D\times\Mm$ such that \begin{equation} \limsup_N\frac{1}{N}\PP\left((\mu^N_\bullet, w^N)\not \in \K\right)\le -M. \end{equation} Now, $\A\cap \K$ is compact, since $\A$ was assumed to be closed. For all $(\mu_\bullet, w)\in \A$, we now use Lemma \ref{lemma: local upper bound} to construct $\U(\mu_\bullet,w)\ni (\mu_\bullet, w)$: if $\mathcal{I}(\mu_\bullet, w)<\infty$ then choose $\U(\mu_\bullet,w)$ such that \begin{equation} \limsup_N \frac{1}{N}\log \PP\left((\mu^N_\bullet, w^N)\in \U(\mu_\bullet, w)\right)\le -\mathcal{I}(\mu_\bullet, w)+\epsilon \end{equation} or if $\mathcal{I}(\mu_\bullet,w)=\infty$, then choose $\U(\mu_\bullet, w)$ such that \begin{equation} \label{eq: second case in local to global}\limsup_N \frac{1}{N}\log \PP\left((\mu^N_\bullet, w^N)\in \U(\mu_\bullet, w)\right)\le -M.\end{equation} The sets $\{\U(\mu_\bullet, w): (\mu_\bullet, w)\in \A\cap\K\}$ are an open cover of $\A\cap \K$, so by compactness we can find $n<\infty$ and $(\mu^{(i)}_\bullet, w^{(i)})\in \A\cap \K$ such that $\A\cap \K$ is covered by $\U_i=\U(\mu^{(i)}_\bullet, w^{(i)}), i\le n$ and conclude that for each $N$, \begin{equation}\PP\left((\mu^N_\bullet, w^N)\in \A\right)\le \PP\left((\mu^N_\bullet, w^N)\not \in \K\right)+\sum_{i=1}^n \PP\left((\mu^N_\bullet, w^N)\in \U_i\right). \end{equation} It follows that \begin{equation} \begin{split} &\limsup_N \frac{1}{N}\log \PP\left((\mu^N_\bullet, w^N)\in \A\right) \\&\hspace{2cm} \le \max\left(\limsup\frac{1}{N}\log  \PP\left((\mu^N_\bullet, w^N)\in \V\right): \V=\K^\mathrm{c}, \U_1,...,\U_n \right).\end{split}  \end{equation}The terms appearing in the right hand side are all \emph{either} bounded by $-M \le -\inf_\A \mathcal{I}$, for the cases where $\V=\K^\mathrm{c}$ or $\V=\U_i$, for a path $(\mu^{(i)}_\bullet, w^{(i)})$ with  $\mathcal{I}(\mu^{(i)}_\bullet, w^{(i)})=\infty$, \emph{or} at most $-\mathcal{I}(\mu^{(i)}_\bullet, w^{(i)}) + \epsilon \le -\inf_\A\mathcal{I}+\epsilon$. All together, we conclude that \begin{equation}\limsup_N \frac{1}{N}\log \PP\left(\mu^N_\bullet, w^N)\not \in \A\right) \le -\inf\{\mathcal{I}(\mu_\bullet, w): (\mu_\bullet, w)\in \A\} +\epsilon \end{equation} and taking $\epsilon\to 0$ concludes the proof in the case where the infimum is finite. The case where the infimum is infinite is essentially identical: we now keep $M$ as a free parameter, choose a compact set $\K$ such that $\limsup_N N^{-1}\log \PP((\mu^N_\bullet, w^N)\not\in \K)\le -M$, and cover $\A\cap \K$ with open sets $\U(\mu_\bullet, w)\ni (\mu_\bullet, w)$ satisfying (\ref{eq: second case in local to global}). The same covering argument then gives \begin{equation} \limsup_N \frac{1}{N}\log \PP\left((\mu^N_\bullet, w^N)\not\in \A\right)\le -M\end{equation} and the conclusion follows by taking $M\to\infty$.  \end{proof}

\section{Properties of the Kac Process \& Boltzmann Equation}\label{sec: properties_and_COM}
 In order to prove Theorem \ref{thm: main} and its consequences, we will use some facts about the Kac process, including behaviour under changes of measure. We recall first the following moment creation property, both for the \emph{sample paths} of the Kac process and for the Boltmzann equation, and a result on the energy of the Boltzmann equation. \cite{norris2016consistency}.   \begin{proposition}[Moments of the Kac process and Boltzmann Equation]\label{prop: MCP} Let $B$ be a kernel of the form $B(v)=1+\delta |v|$ for some $\delta>0$. \begin{enumerate}[label=\roman*).] \item Fix $p\ge 2$, and let $\mu^N_t$ be a Kac process with collision kernel $B$ and an almost sure bound $\langle |v|^2, \mu^N_0\rangle \le a$. Then there exists a constant $C=C(p,\delta,a)<\infty$ such that, for all $0<s<t$, \begin{equation} \EE\left[\sup_{u\in [s, t]}\langle |v|^p, \mu^N_t\rangle \right] \le C(1+t)s^{2-p}.\end{equation} \item For the same $B, p$ as in item i), if $(\mu_t)_{t\ge 0}$ is a solution to the Boltzmann equation (\ref{eq: BE}) and the energy $\langle |v|^2, \mu_t\rangle=a$ is constant, then for all $s>0$,\begin{equation} \langle |v|^p, \mu_t\rangle \le Cs^{2-p}\end{equation} for some $C=C(p,\delta,a)$. \item For any solution $(\mu_t)_{t\ge 0}$ to the Boltzmann equation (\ref{eq: BE}), the energy $\langle |v|^2, \mu_t\rangle$ is nondecreasing in time. \end{enumerate} \end{proposition} Moment estimates similar to the first two items go back to Desvillettes \cite{desvillettes1993some} and Wennberg \cite{mischler1999spatially,wennberg1997entropy}, based on Povzner esimtates; the same methods were applied to the Kac process for single fixed times by Mischler and Mouhot \cite{mischler2013kac}, and the pathwise estimate on compact time intervals of the kind given here was proven by Norris \cite{norris2016consistency}. The monotonicity of the energy can be found in works by Mischler and Wennberg \cite{mischler1999spatially} and Lu \cite{lu1999conservation}. \bigskip \\ The other property we will need are the changes of probability measure necessary to perturb the initial data and dynamics. The changes of measure we will use are as follows. \begin{proposition}[Kac process under change of measure]\label{prop: com} Let $\mu^N_t$ be a Kac process with collision kernel $B$, and velocities initially sampled independently from $\mu_0^\star$, which is a Markov process on a filtered probability space $(\Omega, \mathfrak{F},(\mathfrak{F}_t)_{t\ge 0},\PP)$, and let $w^N_t$ be the associated empirical flux. Let $\varphi:\rrd\to \mathbb{R}$ be such that $\int e^{\varphi(v)} \mu_0^\star(dv)=1$, $A_0\in \mathfrak{F}_0$ such that $c_N=\EE[1_Ae^{N\langle \varphi, \mu^N_0\rangle}]>0$, and let $K:\cp_2^N\times E\to [0,\infty)$ be measurable and such that $K/(1+|v|+|v_\star|)$ is uniformly bounded. Define a new measure $\mathbb{Q}$ by \begin{equation}\begin{split}\label{eq: COM0} \frac{d\mathbb{Q}}{d\PP}=&\exp\bigg(N\langle \varphi, \mu^N_0\rangle +N\langle \log K(\mu^N_0,\cdot), w^N\rangle \\&\hspace{2cm}- N \int_{E}(K-1)(\mu^N_0,t,v,v_\star,\sigma)\overline{w}_{\mu^N}(dt,dv,dv_\star,d\sigma)\bigg)c_N^{-1}1_A\end{split}\end{equation} where we understand the right-hand side to be $0$ if $\text{supp}(w^N)\cap \{K=0\}\neq \emptyset$. Then $\mathbb{Q}$ is a probability measure, under which $\mu^N_0$ is given as the empirical measure of $N$ independent draws from $e^{\varphi(v)}\mu_0^\star(dv)$ conditioned on $A_0\in \mathfrak{F}_0$, and under which $(\mu^N_0,\mu^N_t, w^N_t)$ is a time-inhomogeneous Markov process, with time-dependent generator, for bounded $F:\cp_2^N\times\cp_2^N\times\Mm\to \rr$,\begin{equation}\begin{split} \label{eq: time dependent generator}\mathcal{G}_tF(\nu,\mu^N, w^N)=N\int_{\rrd\times\rrd\times\ssd}(F(\nu,\mu^{N,v,v_\star,\sigma},w^{N,t,v,v_\star,\sigma})-F(\nu,\mu^N, w^N))&\\ & \hspace{-6cm} \dots \times K(\nu,t,v,v_\star,\sigma)B(v-v_\star, \sigma)\mu^N(dv)\mu^N(dv_\star)d\sigma.\end{split}\end{equation}    \end{proposition} This is a version of the standard Girsanov theorem for jump processes, which is tailor-made for our purposes; see, for example, \cite[Appendix 1, Theorem 7.3]{kipnis1998scaling}. The hypotheses on the growth of $K$ are probably not the most general possible, but are sufficient for the applications in this paper in Sections \ref{sec: RLB}, \ref{sec: pf of main}. Since this particular form does not appear to be standard, a proof is given in Appendix \ref{sec: singular girsanov}.
 
 \section{Restricted Lower Bound}\label{sec: RLB} We now give a proof of the lower bound with the additional integrability hypothesis. The restricted lower bound is based on the following approximation lemma. \begin{lemma}[Approximation by Regular Paths]\label{lemma: approximation lemma} Let $(\mu_\bullet, w)$ be a measure-flux pair such that $$ \mathcal{I}(\mu_\bullet, w)<\infty; \qquad \langle 1+|v|^2+|v_\star|^2, w\rangle<\infty. $$ Then there exists a sequence $(\mu^{(n)}_\bullet, w^{(n)})$ of measure-flux pairs whose tilting functions $K^{(n)}$, are continuous, such that  and $K^{(n)}(t,v,v_\star,\sigma)B(v-v_\star,\sigma)$ is bounded and bounded away from $0$, such that $\mu^{(n)}_0$ admits a bounded density with respect to $\mu^\star_0$, and such that \begin{equation} \sup_{t\ge 0} W(\mu^{(n)}_t,\mu_t)+d(w^{(n)},w)\to 0;\qquad \limsup_n \mathcal{I}(\mu^{(n)}_\bullet,w)\le \mathcal{I}(\mu_\bullet, w). \end{equation} Moreover, each $(\mu^{(n)}_\bullet, w^{(n)})$ is the unique measure-flux pair starting from $\mu^{(n)}_0$ and with tilting function $K^{(n)}$.\end{lemma} Throughout, we write the indexes $^{(n)}$ in the superscripts in brackets, to distinguish them from similar notation for the Kac process $\mu^N_\bullet, w^N$.The proof of this lemma is rather technical, and so is deferred until Subsection \ref{subsec: approximation lemma}. Once this lemma is in hand, the restricted lower bound Theorem \ref{thrm: main positive}ii) follows straightforwardly from standard `tilting' arguments, using the change-of-measure given in Proposition \ref{prop: com} via the following law of large numbers. 

\begin{lemma}\label{lemma: restricted LLN} Let $(\mu_\bullet, w)$ be a measure-flux pair with $\mathcal{I}(\mu_\bullet, w)<\infty$, whose tilting function $K$ is continuous and such that $KB(v-v_\star)$ is bounded and bounded away from $0$, and which is the unique measure-flux pair with this tilting function and this value of $\mu_0$. Let $\PPm^N$ be the measures given by Proposition \ref{prop: com} with $\varphi=\log \frac{d\mu_0}{d\mu^\star_0}$ and $K:E\to (0,\infty)$ the tilting function for $(\mu_\bullet, w)$. Then for all open sets $\mathcal{U}\ni (\mu_\bullet, w)$ and $\epsilon>0$, we have \begin{equation} \PPm^N\left((\mu^N_\bullet, w^N)\in \U, \left|\frac{1}{N}\log \frac{d\PPm^N}{d\PP}-\mathcal{I}(\mu_\bullet, w)\right|<\epsilon\right)\to 1.\end{equation} \end{lemma} 

\begin{proof} We start by applying Proposition \ref{prop: com}. Since $K$ is a function only $E\to\rr$, $(\mu^N_t,w^N_t)$ is a Markov process with generator given by (\ref{eq: time dependent generator}) applied to functions $F: \cp_2^N\times\Mm \to \rr$. For the initial data, $\mu^N_0$ is given, under $\PPm^N$, by sampling $N$ particles independently with common law  $e^{\varphi}\mu_0^\star = e^{\log d\mu_0/d\mu_0^\star}\mu_0^\star=\mu_0$; we also remark that te $\varphi$ given in the statement is well-defined, since the finiteness of the rate $\mathcal{I}(\mu_\bullet, w)<\infty$ implies that $\mu_0\ll \mu_0^\star$.  \paragraph*{Step 1: Functional Law of Large Numbers} We begin by show that, under $\mathbb{Q}^N$, the pairs $(\mu^N_\bullet, w^N)$ converge in probability to $(\mu_\bullet, w)$. Since $K, \varphi$ are bounded, Corollary \ref{cor: tightness} applies and the laws $\mathbb{Q}^N\circ(\mu^N_\bullet, w^N)^{-1}$ are tight on $\D\times\Mm$, so every subsequence has a further subsequence converging weakly on $\D\times\Mm$. We will now prove that the only possible subsequential limit is $\delta_{(\mu_\bullet, w)}$, which implies that the whole sequence $\mathbb{Q}^N\circ (\mu^N_\bullet, w^N)^{-1}$ converges weakly to this limit, and hence \begin{equation}\label{eq: conclusion 1}\mathbb{Q}^N\left((\mu^N_\bullet, w^N)\in \U\right)\to 1. \end{equation} Let $S\subset\mathbb{N}$ be any subsequence along which $\mathbb{Q}^N\circ (\mu^N_\bullet, w^N)^{-1}$ converges weakly. Thanks to Skorokhod's representation theorem, we can realise all $(\mu^N_\bullet, w^N), N\in S$ with these laws on a common probability space, with probability measure $\mathbb{Q}$, converging $\mathbb{Q}$-almost surely to a limit $(\widetilde{\mu}_\bullet, \widetilde{w})$. For each $N$, $(\mu^N_\bullet, w^N)$ almost surely lies in the set of pairs satisfying the continuity equation, which is closed by Lemma \ref{lemma: semiconntinuity}, and hence $(\widetilde{\mu}_\bullet, \widetilde{w})$ almost surely satisfies (\ref{eq: CE}). We now show that the limit is almost surely a measure-flux pair with tilting $K$: for all $g\in C_c(E)$, the processes \begin{equation}\begin{split} M^{N,g}_t & =\langle g, w^N_t\rangle - \int_{(0,t]\times \rrd\times\rrd\times\ssd} g(s,v,v_\star,\sigma)K(s,v,v_\star,\sigma)\\ & \hspace{6cm}\dots B(v-v_\star)ds\mu^N_s(dv)\mu^N_s(dv_\star)d\sigma \end{split} \end{equation}is a c{\`a}dl{\`a}g martingale, with previsible, increasing quadratic variation \begin{equation}\begin{split} [M^{N,g}]_t&=\frac{1}{N}\int_{(0,t]\times\rrd\times\rrd\times\ssd} g^2 K(s,v,v_\star,\sigma)B(v-v_\star)ds\mu^N_s(dv)\mu^N_s(dv_\star)d\sigma \\[2ex] & \hs \le \|g\|_\infty^2 \sup_E (B(v-v_\star)K)T/N,\end{split}\end{equation}see, for instance, \cite{darling2008differential,norris2016consistency}. In particular, since $B(v-v_\star)K$ is bounded by construction, the constant in the final expression is finite. Therefore, for all such $g$, \begin{equation} \label{eq: identification of limit} \mathbb{E}_{\mathbb{Q}}\left|\langle g, w^N\rangle - \int_E g K(s,v,v_\star,\sigma)B(v-v_\star)ds\mu^N_s(dv)\mu^N_s(dv_\star)d\sigma\right|\le\frac{C_g}{\sqrt{N}}  \end{equation} for some constant $C_g$. Taking $N\to \infty$ through $S$, the first term in the expectation converges almost surely to $\langle g, \widetilde{w}\rangle$, and the second term converges to $\int gK d\overline{m}_{\widetilde{\mu}}$ by Lemma \ref{lemma: really useful continuity result} applied to $gK$. We now take $N\to \infty$ through $S$ to obtain \begin{equation} \label{eq: identification of limit 3} \mathbb{E}_\mathbb{Q}\left|\langle g, \widetilde{w}\rangle -\int_E g(s,v,v_\star,\sigma)K(s,v,v_\star,\sigma)\overline{m}_{\widetilde{\mu}}(ds,dv,dv_\star,d\sigma) \right|=0\end{equation} and so the integrand is $0$, $\mathbb{Q}$-almost surely. Taking a union bound over a countable dense set in $C_c(E)$, we conclude that $\widetilde{w}=K\overline{m}_{\widetilde{\mu}}$ almost surely, and the limit is a measure-flux pair with the prescribed rate function $K$, and the convergence $\mu^N_0\to \widetilde{\mu}_0$ implies that $\mu_0=\widetilde{\mu}_0$. By hypothesis, these properties uniquely characterise the desired limit $(\mu_\bullet, w)$, so $\mathbb{Q}((\widetilde{\mu}_\bullet,\widetilde{w})=(\mu_\bullet, w))=1$ and the step is complete.\paragraph*{Step 2: Law of Large Numbers for the Dynamic Cost} We will now show that the (random) exponential cost induced by the change of measure (\ref{eq: COM0}) converges under $\mathbb{Q}^N$: for all $\epsilon>0$, \begin{equation} \label{eq: desired convergence of dynamic cost}\mathbb{Q}^N\left(\left|\frac{1}{N}\log \frac{d\mathbb{Q}^N}{d\mathbb{P}}-\mathcal{I}(\mu_\bullet, w)\right|>\epsilon\right)\to 0.\end{equation} We begin by using the definitions of $\mathcal{I}$ and (\ref{eq: COM0}) to rewrite the difference as \begin{equation} \label{eq: difference of dynamic cost} \begin{split}\frac{1}{N}\log \frac{d\mathbb{Q}^N}{d\mathbb{P}}-\mathcal{I}(\mu_\bullet, w)&=\langle \varphi, \mu^N_0\rangle - H(\mu_0|\mu^\star_0) + \langle \log K, w^N-w\rangle \\&\hspace{-2cm} -\int_E(K-1)B(v-v_\star)(\mu^N_t(dv)\mu^N_t(dv_\star)-\mu_t(dv)\mu_t(dv_\star))d\sigma   \end{split} \end{equation} and examine the terms one by one. Fix, throughout, $\epsilon, \epsilon'>0$. \paragraph*{Step 2a: Cost of the Initial Data} For the cost of the initial data, $\langle \varphi, \mu^N_0\rangle$ is the empirical mean of $\log \frac{d\mu'_0}{d\mu^\star_0}$, sampled at $N$ independent draws from $\mu_0$. The mean of each draw is exactly $\int_{\rrd} \log \frac{d\mu_0}{d\mu^\star_0}(v) \mu_0(dv)=:H(\mu_0|\mu_0^\star)$, so by the weak law of large numbers, for all $N$ large enough\begin{equation}\label{eq: term 1 of dynamic cost} \mathbb{Q}^N\left(|\langle \varphi, \mu^N_0\rangle-H(\mu_0|\mu^\star_0)|>\epsilon/4\right)<\epsilon'/3. \end{equation} \paragraph*{Step 2b: Integral against Empirical Flux}Let us now examine the second term. By the choice of $K$, $\log K$ is continuous, bounded above, and bounded below by $\log(c/B(v-v_\star))$ for some constant $c>0$. We can further bound this below by\begin{equation} \log \frac{1}{B(v-v_\star)}\ge \log \frac{1}{(1+|v|)(1+|v_\star)} \ge -c(|v|+|v_\star|)\end{equation} for a new constant $c$: in particular, $|\log K|\le C(1+|v|+|v_\star|)$ is continuous, and of at most linear growth. Recalling that $B(v-v_\star)K$ is bounded, we also estimate, uniformly in $N$, \begin{equation}\begin{split}\mathbb{E}_{\mathbb{Q}^N}\left[\langle 1+|v|^2+|v_\star|^2, w^N\rangle\right]&=\mathbb{E}_{\mathbb{Q}^N}\left[\int_E (1+|v|^2+|v_\star|^2)B(v-v_\star)K\mu^N_t(dv)\mu^N_t(dv_\star)d\sigma\right] \\[1ex] & \le C\mathbb{E}_{\mathbb{Q}^N}\langle 1+|v|^2, \mu^N_0\rangle =C\langle 1+|v|^2, \mu_0\rangle. \end{split} \end{equation} Elementary Chebychev estimates produce $R<\infty$ such that, uniformly in $N$,\begin{equation} \label{eq: truncate flux} \mathbb{Q}^N\left(\langle (1+|v|+|v_\star|)1(|v|>R \text{ or }|v_\star|>R), w^N\rangle > \epsilon/12C\right) <\epsilon'/6;\end{equation} and similarly, using the boundedness of $B(v-v_\star)K$ and finiteness of the second moments, the second moment $\langle 1+|v|^2+|v_\star|^2, w\rangle <\infty$ is also finite, and so we can additionally choose $R$ so that \begin{equation} \langle (1+|v|+|v_\star|)1(|v|>R\text{ or }|v_\star|>R),w\rangle <\frac{\epsilon}{12C}\end{equation} and construct a continuous, compactly supported function $g:E\to \mathbb{R}$ such that $|g-\log K|\le C(1+|v|+|v_\star|)$ and which agrees with $\log K$ when both $|v|, |v_\star|\le R$. We therefore find from (\ref{eq: truncate flux}) that \begin{equation}\label{eq: truncate flux 2} \mathbb{Q}^N\left(\langle |g-\log K|, w^N\rangle>\epsilon/12\right)<\epsilon'/6; \qquad \langle |g-\log K|, w\rangle <\frac{\epsilon}{12}.\end{equation} Thanks to the convergence in distribution, for $N$ large enough, \begin{equation}\label{eq: truncate flux 3}\mathbb{Q}^N\left(|\langle g, w^N-w\rangle|>\epsilon/12\right)<\epsilon'/6 \end{equation}  and we find from (\ref{eq: truncate flux 2},\ref{eq: truncate flux 3}) that \begin{equation} \label{eq: 2 term of dynamic cost} \mathbb{Q}^N\left(\left|\langle \log K, w^N-w\rangle\right|>\epsilon/4\right)<\epsilon'/3. \end{equation}  \paragraph*{Step 2c: Integral against Compensator} We finally deal with the third term in (\ref{eq: difference of dynamic cost}). Since $B(v-v_\star)K$ is bounded and $K$ is continuous, it follows that $B(v-v_\star)(K-1)$ is of at most linear growth, so there exists $C$ such that $|B(v-v_\star)(K-1)|\le C(1+|v|+|v_\star|)$, and as in the previous step, we can choose $R$ such that, uniformly in $N$, \begin{equation}\label{eq: truncate compensator} \mathbb{Q}^N\left(\sup_{t\le T} \int_{\rrd\times \rrd} (1+|v|+|v_\star|)1(|v|>R\text{ or }|v_\star|>R)\mu^N_t(dv)\mu^N_t(dv_\star)>\epsilon/12CT\right)<\epsilon'/6\end{equation} and similarly such that \begin{equation}\label{eq: truncate compensator 2} \sup_{t\le T} \int_{\rrd\times \rrd} (1+|v|+|v_\star|)1(|v|>R\text{ or }|v_\star|>R)\mu'_t(dv)\mu'_t(dv_\star)\le \epsilon/12CT.\end{equation} We again truncate, with a proxy $h: E\to \mathbb{E}$ which is continuous, compactly supported, agrees with $(K-1)B(v-v_\star)$ if both $|v|, |v_\star|\le R$, and such that $|(K-1)B(v-v_\star)-h|\le C(1+|v|+|v_\star|)$ for the same constant $C$. Using Lemma \ref{lemma: really useful continuity result} again, \begin{equation} \label{eq: convergence of compensator 1} \mathbb{Q}^N\left(\left|\int_E h(t,v,v_\star,\sigma)dt(\mu^N_t(dv)\mu^N_t(dv_\star)-\mu_t(dv)\mu_t(dv_\star))d\sigma\right|>\epsilon/12\right)<\epsilon'/6\end{equation} for $N$ large enough, while (\ref{eq: truncate compensator}) implies that \begin{equation} \label{eq: convergence of compensator 2}\mathbb{Q}^N\left(\left|\int_E (h-B(v-v_\star)(K-1))dt\mu^N_t(dv)\mu^N_t(dv_\star)d\sigma\right|>\frac{\epsilon}{12}\right)<\epsilon'/6\end{equation} and (\ref{eq: truncate compensator 2}) implies that \begin{equation} \label{eq: convergence of compensator 3} \left|\int_E (h-B(v-v_\star)(K-1))dt\mu_t(dv)\mu_t(dv_\star)\right|<\frac{\epsilon}{12}.\end{equation} Gathering (\ref{eq: convergence of compensator 1}, \ref{eq: convergence of compensator 2}, \ref{eq: convergence of compensator 3}), we conclude that, for $N$ large enough,\begin{equation}\label{eq: term 3 of dynamic cost}\mathbb{Q}^N\left(\left|\int_E(K-1)B(v-v_\star)dt(\mu^N_t(dv)\mu^N_t(dv_\star)-\mu_t(dv)\mu_t(dv_\star))d\sigma\right|>\epsilon/4\right)<\epsilon'/3. \end{equation} Returning to (\ref{eq: difference of dynamic cost}), we combine (\ref{eq: term 1 of dynamic cost}, \ref{eq: 2 term of dynamic cost}, \ref{eq: term 3 of dynamic cost}) to obtain, for all $\epsilon, \epsilon'>0$, and all $N$ large enough, depending on $\epsilon, \epsilon'$, \begin{equation} \mathbb{Q}^N\left(\left|\frac{1}{N}\log \frac{d\mathbb{Q}^N}{d\PP}-\mathcal{I}(\mu_\bullet, w)\right|>\epsilon\right)<\epsilon'\end{equation} and we have proven the desired convergence (\ref{eq: desired convergence of dynamic cost}). Together with the previous step, the proof is complete. \end{proof} 

We can now prove the restricted lower bound. 
\begin{proof}[Proof of Theorem \ref{thrm: main positive}ii)] Let us fix a Skorokhod-open set $\U$, a path $(\mu_\bullet, w)\in \U\cap \mathcal{R}$, and $\epsilon>0$. Let us assume that $\mathcal{I}(\mu_\bullet, w)<\infty$. Thanks to Lemma \ref{lemma: approximation lemma}, there exists a measure-flux pair $(\mu'_\bullet, w')\in \U$ with overall cost $\mathcal{I}(\mu'_\bullet, w')<\mathcal{I}(\mu_\bullet, w)+\epsilon$, satisfying the conclusions of Lemma \ref{lemma: approximation lemma}. For the changes of measure $\PPm^N\ll \PP$ as in Lemma \ref{lemma: restricted LLN}, we then have, for all $N$ large enough, \begin{equation} \mathbb{Q}^N\left((\mu^N_\bullet, w^N)\in \U, \frac{d\mathbb{Q}^N}{d\mathbb{P}}\le \exp\left(N(\mathcal{I}(\mu'_\bullet,w')+\epsilon)\right)\right)\ge \frac{1}{2}\end{equation} which implies that \begin{equation} \begin{split} \mathbb{P}\left((\mu^N_\bullet, w^N)\in \U\right)&\ge \mathbb{E}_{\mathbb{Q}^N}\left[\left(\frac{d\mathbb{Q}^N}{d\mathbb{P}}\right)^{-1}1\left(\frac{d\mathbb{Q}^N}{d\mathbb{P}}\le \exp\left(N(\mathcal{I}(\mu'_\bullet,w')+\epsilon)\right),\h (\mu^N_\bullet, w^N)\in \U\right)\right] \\[2ex] &\ge \frac{1}{2}\exp\left(-N(\mathcal{I}(\mu'_\bullet, w')+\epsilon)\right) \ge \frac{1}{2}\exp\left(-N(\mathcal{I}(\mu_\bullet, w)+2\epsilon)\right).\end{split} \end{equation}  Taking the logarithm and the limit $N\to \infty$, and then $\epsilon\to 0$, \begin{equation} \label{eq: nearly end of RLB}\liminf_N\frac{1}{N}\log \PP\left((\mu^N_\bullet, w^N)\in \U\right)\ge -\mathcal{I}(\mu_\bullet, w). \end{equation} Of course, (\ref{eq: nearly end of RLB}) trivially holds if $\mathcal{I}(\mu_\bullet,w)=\infty$, and so applies to any $(\mu_\bullet, w)\in \U\cap \mathcal{R}$, and the result is proven.\end{proof}

\subsection{Proof of Approximation Lemma} \label{subsec: approximation lemma} 
We will now present the proof of the approximation lemma as a number of intermediate steps. We will present the statements here, to give an overview of the proof of the overall approximation lemma, and the proofs in Subsection \ref{subsec: polypheme}. We begin with the following definition. \begin{definition} Let $(\mu_\bullet, w)$ be a measure-flux pair, and $\lambda>0$. Let $g_\lambda$ be the Gaussian in $\rrd$ \begin{equation}\label{eq: gdelta} g_\lambda(x):=\frac{1}{(2\pi \lambda)^{d/2}}\exp\left(-|x|^2/2\lambda\right).\end{equation}  We define the convolutions $(g_\lambda\star \mu_\bullet), g_\lambda\star w$ by \begin{equation} \label{eq: convolution 1} (g_\lambda\star\mu_\bullet )_t(dv)=(g_\lambda\star \mu_t)(dv)=\int_{\rrd} g_\lambda(v-u)\mu_t(du)dv; \end{equation} \begin{equation} \label{eq: convolution 2}  (g_\lambda\star w)(dt,dv,dv_\star,d\sigma)=\int_{\rrd\times \rrd}g_\lambda(v-u) g_\lambda(v_\star-u_\star) w(dt,du,du_\star,d\sigma).\end{equation}  The measures $g_\lambda\star \mu_t$ are absolutely continuous with respect to the Lebesgue measure; we will alternatively use the notation $g_\lambda \star \mu_t$ for their density on $\rrd$.  \end{definition} \begin{remark} Let us note that the choice of Gaussian mollification is essential here, as it is the unique mollifier which is invariant under changing between the pre- and post- collisional velocities.\end{remark}\begin{lemma}[Approximation by Convolution]\label{lemma: convolution} Suppose $(\mu_\bullet, w)$ is a measure-flux pair with a bounded tilting function $K$,  such that \begin{equation}\label{eq: controlling quantities} \langle |v|^2+|v_\star|^2, w\rangle < \infty; \qquad \mathcal{J}(\mu_\bullet, w)<\infty;\qquad \sup_{t\le T} \langle |v|^2, \mu_t\rangle=\langle |v|^2, \mu_0\rangle<\infty.\end{equation} Then, for all $\lambda>0$, $(\mu_\bullet\star g_\lambda, w\star g_\lambda)$ is a measure-flux pair. Furthermore, there exists a continuous function $\vartheta: [0,1]\to [0,\infty)$, which is continuous at $0$ and $\vartheta(0)=0$, and a constant $C$, \emph{which only depends on upper bounds for the quantities in (\ref{eq: controlling quantities})} and not the boundedness of $K$,  such that for all $\lambda\in (0,1]$, \begin{equation} \label{eq: dynamic cost conclusion} \mathcal{J}(g_\lambda\star \mu_\bullet, g_\lambda\star w) \le \mathcal{J}(\mu_\bullet, w)+C\vartheta(\lambda).\end{equation} Finally, the tilting function $K^\lambda$ satisfies \begin{equation} \label{eq: contraction on supremum}\sup_{t,v,v_\star,\sigma} K^\lambda(t,v,v_\star,\sigma)B(v-v_\star)\le \sup_{t,v,v_\star,\sigma} K(t,v,v_\star,\sigma)B(v-v_\star)\end{equation} \end{lemma}  We now apply this to produce some approximation results. \begin{lemma}\label{lemma: approximation 1} Let $\mu_\bullet, w$ be as in Lemma \ref{lemma: approximation lemma}. Then there exist measure-flux pairs $\mu^{(n)}_\bullet, w^{(n)}$ such that \begin{equation} \label{eq: Approx 1 1} \mathcal{J}(\mu^{(n)}_\bullet, w^{(n)})\to \mathcal{J}(\mu_\bullet, w);\qquad \sup_{t\le T} \|(1+|v|^2)(\mu^{(n)}_t-\mu_t)\|_{\mathrm{TV}}+\|(1+|v|^2+|v_\star|^2)(w^{(n)}-w)\|_{\mathrm{TV}}\to 0\end{equation} and, for each $n$, the tilting function $K^{(n)}$ is such that $K^{(n)}(t,v,v_\star,\sigma)B(v-v_\star)$ is bounded and $K^{(n)}$ is continuous in $v,v_\star$. Furthermore, the starting points $\mu^{(n)}_0$ can be taken to be of the form \begin{equation} \label{eq: form of starting point} \mu^{(n)}_0(dv)=c_n^{-1}\left(\mu_0(dv)+\nu^{(n)}(dv)\right) \end{equation} for a suitable normalising constant $c_n \to 1$ and measures $\nu^{(n)}$ on $\rrd$ with $\langle 1+|v|^2, \nu^{(n)}\rangle \to 0$.\end{lemma}
\begin{lemma}\label{lemma: approximation 1.5} Let $\mu_\bullet, w$ be as in Lemma \ref{lemma: approximation lemma}. Then there exist measure-flux pairs $\mu^{(n)}_\bullet, w^{(n)}$ such that \begin{equation} \label{eq: Approx 1.5 1} \mathcal{I}(\mu^{(n)}_\bullet, w^{(n)})\to \mathcal{I}(\mu_\bullet, w);\qquad \sup_{t\le T} \h W\left(\mu^{(n)}_t, \mu_t\right)+d(w^{(n)},w)\to 0\end{equation} and, for each $n$, the tilting function $K^{(n)}$ is such that $K^{(n)}(t,v,v_\star,\sigma)B(v-v_\star)$ is bounded and $K^{(n)}$ is continuous in $v,v_\star$.  \end{lemma} \begin{lemma}\label{lemma: approximation 2}Let $\mu_\bullet, w$ be a measure-flux pair with finite rate $\mathcal{I}(\mu_\bullet, w)<\infty$, such that $K(t,v,v_\star,\sigma)B(v-v_\star)$ is bounded and $K$ is continuous in $v,v_\star$. Then there exist measure-flux pairs $\mu^{(n)}_\bullet, w^{(n)}$ with \begin{equation} \label{eq: Approx 2 1}\mathcal{I}(\mu^{(n)}_\bullet, w)\to \mathcal{I}(\mu_\bullet, w);\qquad \sup_{t\le T}\|\mu^{(n)}_t-\mu_t\|_\mathrm{TV}+\|w^{(n)}-w\|_\mathrm{TV}\to 0 \end{equation} and additionally, for each $n$, \begin{equation}\label{eq: Approx 2 2} \sup_{(t,v,v_\star,\sigma)} K^{(n)}(t,v,v_\star,\sigma)B(v-v_\star)\le \sup_{(t,v,v_\star,\sigma)} K(t,v,v_\star,\sigma)B(v-v_\star)+1; \end{equation} \begin{equation}\label{eq: Approx 2 3} \inf_{(t,v,v_\star,\sigma)} K^{(n)}(t,v,v_\star,\sigma)B(v-v_\star)>0; \end{equation} \begin{equation} \sup \frac{d\mu^{(n)}_0}{d\mu^\star_0}<\infty \end{equation} and such that $K^{(n)}$ are continuous functions on $E$. Moreover, the approximations are uniquely characterised among measure-flux pairs by the initial value $\mu^{(n)}_0$ and the tilting function $K^{(n)}$.  \end{lemma}  Equipped with these lemmas, the stated result Lemma \ref{lemma: approximation lemma} follows by a standard diagonal argument. \begin{proof}[Proof of Lemma \ref{lemma: approximation lemma}] Let us fix $\mu_\bullet, w$ as given, and construct a sequence of approximating measure-flux pairs $\mu^{(n)}_\bullet, w^{(n)}$ as follows. By Lemma \ref{lemma: approximation 1.5}, there exists a pair $\mu^{(n,1)}_\bullet, w^{(n,1)}$ whose tilting function $K^{(n,1)}$ is continuous in $v,v_\star$ and $K^{(n,1)}B(v-v_\star)$ is bounded, such that $\mu^{(n,1)}_0$ admits a continuous and positive density, and such that \begin{equation}\label{eq: diagonal approximation 1} \sup_{t\le T}W\left(\mu^{(n,1)}_t,\mu_t\right)+d(w^{(n,1)},w)+\left|\mathcal{I}(\mu^{(n,1)}_\bullet, w^{(n,1)})-\mathcal{I}(\mu_\bullet, w)\right|<\frac{1}{n}.\end{equation} Thanks to Lemma \ref{lemma: approximation 2}, we can approximate $\mu^{(n,1)}_\bullet, w^{(n,1)}$ by a further pair $\mu^{(n,2)}_\bullet, w^{(n,2)}$, whose tilting function $K^{(n,2)}$ is continuous and so that $K^{(n,2)}B(v-v_\star)$ is still bounded and bounded away from $0$, and where $\mu^{(n,2)}_0$ has a bounded density with respect to $\mu^\star_0$, which is uniquely characterised among measure-flux pairs by the initial data $\mu^{(n,2)}_0$ and tilting function $K^{(n,2)}$, with further error \begin{equation}\label{eq: diagonal approximation 2} \sup_{t\le T}\|\mu^{(n,2)}_t-\mu^{(n,1)}_t\|_\mathrm{TV}+\|w^{(n,2)}-w^{(n,1)}\|_{\mathrm{TV}}+\left|\mathcal{I}(\mu^{(n,2)}_\bullet, w^{(n,2)})-\mathcal{I}(\mu_\bullet^{(n,1)}, w^{(n,1)})\right|<\frac{1}{n}.\end{equation} Combining (\ref{eq: diagonal approximation 1}, \ref{eq: diagonal approximation 2}), we recall that the total variation distance on measures on $\rrd$, respectively $E$, dominates the Wasserstein$_1$ distance $W$, respectively $d$, so the sequence $\mu^{(n)}_\bullet = \mu^{(n,2)}_\bullet, w^{(n)}=w^{(n,2)}$ has the desired properties. \end{proof} 
 
\subsection{Proof of Lemmas}\label{subsec: polypheme} We start with the convolution lemma, which is the most difficult step. \begin{proof}[Proof of Lemma \ref{lemma: convolution}] We divide the proof into several steps. Throughout, $C$ will denote a constant, which may vary from line to line, but is allowed to depend \emph{only} on the quantities specified in (\ref{eq: controlling quantities}). \paragraph*{Step 1: $(g_\lambda\star \mu_\bullet, g_\lambda\star w)$ solves the continuity equation} This property is fairly well-known, see Erbar \cite{erbar2016gradient} or Basile \cite{basile2021large}, and we include a proof for completeness. If we fix $f \in C_b(\rrd)$, let us denote $g_\lambda\star f$ the convolution $(g_\lambda\star f)(v):=\int_{\rrd} g_\lambda(v-w)f(w)dw$, and observe that $\langle f, g_\lambda\star \mu_t\rangle =\langle g_\lambda\star f, \mu_t\rangle$ for all $t\in [0,T].$ Now, using the continuity equation for $(\mu_\bullet, w)$ with the test function $g_\lambda\star f$, for all $t\in [0,T]$, \begin{equation}\begin{split} \label{eq: stab under convolution 1}\langle f, g_\lambda\star \mu_t\rangle -\langle f, g_\lambda \star \mu_0\rangle &=\langle g_\lambda\star f, \mu_t\rangle - \langle g_\lambda \star f, \mu_0\rangle\\ &=\int_E \Delta (g_\lambda \star f)(v,v_\star,\sigma)1_{s\le t}w(ds,dv,dv_\star,d\sigma).\end{split} \end{equation} Let us now fix $v, v_\star,\sigma$, and observe that the map $\mathcal{T}_\sigma:(v,v_\star)\to (v',v'_\star)$ is a linear isometry of Euclidean distance on $(\rrd)^2$; for variables $(u', u_\star')\in (\rrd)^2$, let us write $(u, u_\star)$ for the preimage under $\mathcal{T}_\sigma$. We therefore have \begin{equation} \begin{split} &(g_\lambda\star f)(v')+(g_\lambda \star f)(v'_\star)\\ & \hspace{1cm}=\frac{1}{(2\pi\lambda)^d}\int_{\rrd\times\rrd} ( f(u')+ f(u_\star'))\exp\left(-\frac{|u'-v'|^2+|u_\star'-v_\star'|^2}{2\lambda}\right)du' du_\star' \\[2ex] & \hspace{1cm}=\frac{1}{(2\pi\lambda)^d}\int_{\rrd\times\rrd} ( f(u')+ f(u_\star'))\exp\left(-\frac{|u-v|^2+|u_\star-v_\star|^2}{2\lambda}\right)du' du_\star' \\[2ex] &\hspace{1cm}=\frac{1}{(2\pi\lambda)^d}\int_{\rrd\times\rrd} ( f(u')+ f(u_\star'))\exp\left(-\frac{|u-v|^2+|u_\star-v_\star|^2}{2\lambda}\right)du du_\star \\[2ex] &\hspace{1cm}=\int_{\rrd\times\rrd} ( f(u')+ f(u_\star'))g_\lambda(u-v)g_\lambda(u_\star-v_\star)dudu_\star \end{split} \end{equation} where the penultimate line makes the change of variables $(u', u_\star')=\mathcal{T}_\sigma(u, u_\star)$, with unit determinant. We now substitute the resulting identity \begin{equation}\begin{split} \Delta (g_\lambda\star  f)(v,v_\star,\sigma)&=\int_{\rrd\times\rrd}( f(u')+ f(u_\star')- f(u)- f(u_\star))g_\lambda(u-v)g_\lambda(u_\star-v_\star)dudu_\star \\ &=\int_{\rrd\times\rrd} (\Delta  f)(u,u_\star,\sigma)g_\lambda(u-v)g_\lambda(u_\star-v_\star)dudu_\star \end{split}\end{equation} into (\ref{eq: stab under convolution 1}) to obtain \begin{equation} \begin{split} \langle  f, g_\lambda\star \mu_t\rangle-\langle  f, g_\lambda\star \mu_0\rangle &= \int_E \int_{\rrd\times \rrd} (\Delta  f)(u,u_\star,\sigma)1_{s\le t}g_\lambda(u-v)g_\lambda(u_\star-v_\star)w(ds,dv,dv_\star,d\sigma)\\&=:\int_E (\Delta  f)(u,u_\star,\sigma)1_{s\le t} (g_\lambda \star w)(ds,du,du_\star,d\sigma)\end{split}  \end{equation} and, since $f$ is arbitrary, we conclude that $(g_\lambda \star \mu_\bullet, g_\lambda \star w)$ satisfies the continuity equation (\ref{eq: CE}) as desired. \paragraph*{Step 2: Identification of the Tilting Function} To show that $(\mu_\bullet,w)$ is a measure-flux pair, and in preparation for estimating the dynamic cost, we will now explicitly find a tilting function.  Let us write $K$ for the tilting function for the pair $(\mu_\bullet, w)$, so that $w=K\overline{m}_{\mu}$. For any Borel subset $A\subset E$, we observe that \begin{equation}\begin{split} \label{eq: identification of Kdelta 1} (g_\lambda \star w)(A)=&\int_E \int_{\rrd\times\rrd} 1_A(t,u,u_\star,\sigma)g_\lambda(u-v)g_\lambda(u_\star-v_\star)K(t,v,v_\star,\sigma)\\ & \hspace{6cm}\dots\times \overline{m}_\mu(dt,dv,dv_\star,d\sigma)dudu_\star \\ & \hspace{-2cm}= \int_E 1_A(t,u,u_\star,\sigma)\left(\int_{\rrd\times\rrd}g_\lambda(u-v)g_\lambda(u_\star-v_\star)K(t,v,v_\star,\sigma)B(v-v_\star)\mu_t(dv) \mu_t(dv_\star)\right)\\ & \hspace{6cm}\dots\times dudu_\star dt d\sigma.\end{split}\end{equation} Now, let us define $K^\lambda$ by \begin{equation}\label{eq: identification of Kdelta 2} K^\lambda(t,u,u_\star,\sigma):=\frac{\int_{\rrd\times\rrd}g_\lambda(u-v)g_\lambda(u_\star-v_\star)K(t,v,v_\star,\sigma)B(v-v_\star)\mu_t(dv) \mu_t(dv_\star)}{B(u-u_\star)(g_\lambda\star \mu_t)(u)(g_\lambda \star\mu_t)(u_\star)} \end{equation} and observe that this is a well-defined function, since $K$ was assumed to be bounded and $B$ is bounded away from $0$, and where in the denominator $(g_\lambda\star \mu_t)(u)>0$ denotes the density of the measure $g_\lambda\star\mu_t$ with respect to the Lebesgue measure. Returning to (\ref{eq: identification of Kdelta 1}), this definition yields \begin{equation} \label{eq: identification of Kdelta 3} \begin{split} (g_\lambda \star w)(A)&=\int_E 1_A(t,u,u_\star,\sigma)K^\lambda(t,u,u_\star,\sigma)B(u-u_\star)(g_\lambda\star\mu_t)(u)du (g_\lambda\star \mu_t)(u_\star)du_\star dtd\sigma \\ & 
  =\int_E 1_A(t,u,u_\star,\sigma)K^\lambda(t,u,u_\star,\sigma) \overline{m}_{g_\lambda\star \mu}(dt,du,du_\star,d\sigma).  \end{split} \end{equation} We conclude that $K^\lambda$ is a tilting function for $(g_\lambda\star \mu_\bullet, g_\lambda\star w)$, and so this is a measure-flux pair as claimed, and the bound (\ref{eq: contraction on supremum}) is immediate. For future convenience, we will now define \begin{equation} \label{eq: identification of Kdelta 4} \overline{r}^\lambda_t(u,u_\star):=\int_{\rrd\times \rrd} g_\lambda(u-v)g_\lambda(u_\star-v_\star)B(v-v_\star)\mu_t(dv)\mu_t(dv_\star)\end{equation} and introduce the proxy to $K^\lambda$ given by \begin{equation}\label{eq: identification of Kdelta 5} \begin{split}\overline{K}^\lambda(t,u,u_\star,\sigma):&=\frac{\int_{\rrd\times\rrd}g_\lambda(u-v)g_\lambda(u_\star-v_\star)K(t,v,v_\star,\sigma)B(v-v_\star)\mu_t(dv) \mu_t(dv_\star)}{\overline{r}^\lambda_t(u,u_\star)}\\[2ex]& =\int_{\rrd\times\rrd} K(t,v,v_\star,\sigma)\nu^\lambda_{(t,u,u_\star)}(dv,dv_\star)\end{split}\end{equation} which we have written in terms of the integral against the probability measure \begin{equation} \label{eq: identification of Kdelta 6} \nu^\lambda_{(t,u,u_\star)}(dv,dv_\star)=\frac{B(v-v_\star)g_\lambda(v-u)\mu_t(dv)g_\lambda(v_\star-u_\star)\mu_t(dv_\star)}{\overline{r}^\lambda_t(u,u_\star)}.\end{equation} For any $u,u_\star,\sigma$, the quotient is given by the function \begin{equation} \begin{split} \label{eq: identification of Kdelta 7} \frac{K^\lambda}{\overline{K}^\lambda}(t,u,u_\star,\sigma)&= \psi_\lambda(t,u,u_\star):=\frac{\overline{r}^\lambda_t(u,u_\star)}{B(u-u_\star) (g_\lambda\star \mu_t)(u)(g_\lambda\star \mu_t)(u_\star)} \\[3ex] &=\frac{1}{B(u-u_\star)}\int_{\rrd\times \rrd} B(v-v_\star)\frac{g_\lambda(v-u)\mu_t(dv)}{(g_\lambda\star \mu_t)(u)} \frac{g_\lambda(v_\star-u_\star)\mu_t(dv_\star)}{(g_\lambda\star \mu_t)(u_\star)}.\end{split} \end{equation} which depends only on $\mu_t$, and not on $w$. \paragraph*{Step 3: Decomposition of the Rate Function}  We now break the rate function up into several parts which can be more easily manipulated; for the rest of the proof, let us consider only $0<\lambda\le 1$. We start from \begin{equation} \label{eq: decomposition of dynamic cost 1} \mathcal{J}(g_\lambda\star \mu_\bullet, g_\lambda \star w)=\int_E (\tau-1)(K^\lambda(t,v,v_\star,\sigma))\overline{m}_{g_\lambda\star \mu}(dt,dv,dv_\star,d\sigma)+\overline{m}_{g_\lambda\star\mu}(E)\end{equation} where we recall that $(\tau-1)(x)=x\log x-x$ is a convex function on $[0,\infty)$. Next, we observe that \begin{equation} \label{eq: decomposition of dynamic cost 2} (\tau-1)(K^\lambda)=K^\lambda \log  \psi_\lambda +  \psi_\lambda(\tau-1)(\overline{K}^\lambda).\end{equation}For the first term, we start with the observation that \begin{equation}\label{eq: bad term 1}\begin{split} &K^\lambda(t,u,u_\star,\sigma)\overline{m}_{g_\lambda\star \mu}(dt,du,du_\star,d\sigma)\\& \hs =\int_{\rrd\times\rrd} K(t,v,v_\star,\sigma)B(v-v_\star)\mu_t(dv)\mu_t(dv_\star)g_\lambda(u-v)g_\lambda(u_\star-v_\star)dt d\sigma du du_\star\end{split}\end{equation} from which it follows that we can rewrite the integral of the first term as an error term \begin{equation} \label{eq: bad term 2}\begin{split} \mathcal{T}_\lambda(\mu_\bullet, w)&:=\int_E K^\lambda(t,u,u_\star,\sigma) \log  \psi_\lambda(t,u,u_\star)\overline{m}_{g_\lambda\star\mu}(dt,du,du_\star,d\sigma) \\& =\int_{E\times\rrd\times\rrd}K(t,v,v_\star,\sigma)B(v-v_\star)g_\lambda(u-v)g_\lambda(u_\star-v_\star)\log  \psi_\lambda(t, u,u_\star)\\ & \hspace{5cm}\dots \times dt\mu_t(dv)\mu_t(dv_\star)d\sigma du du_\star. \end{split}\end{equation} To integrate the second term, we note that \begin{equation}\label{eq: decomposition of dynamic cost 3}\begin{split}  \psi_\lambda(t,u,u_\star)\overline{m}_{g_\lambda\star \mu}(dt,du,du_\star,d\sigma)& =\overline{r}_t(u,u_\star) dt du du_\star d\sigma. \end{split} \end{equation} Since $\nu^\lambda_{(t,u,u_\star)}$ are probability measures, we can apply Jensen to the convex function $\tau-1$ to find \begin{equation}\label{eq: decomposition of dynamic cost 4} \begin{split} (\tau-1)(\overline{K}^\lambda(t,u,u_\star,\sigma))&\le \int_{\rrd\times\rrd} (\tau-1)(K(t,v,v_\star,\sigma))\nu_{(t,u,u_\star)}(dv,dv_\star) \\ &\hspace{-1cm}= \frac{1}{\overline{r}^\lambda_t(u,u_\star)}\int_{\rrd\times\rrd} (\tau-1)(K(t,v,v_\star,\sigma))B(v-v_\star)g_\lambda(u-v)g_\lambda(u_\star-v_\star)\\& \hspace{6cm}\dots\times\mu_t(dv)\mu_t(dv_\star).\end{split} \end{equation} Gathering (\ref{eq: decomposition of dynamic cost 3}, \ref{eq: decomposition of dynamic cost 4}), we obtain \begin{equation} \label{eq: decomposition of dynamic cost 5} \begin{split} &\int_E  \psi_\lambda(t,u,u_\star)(\tau-1)(\overline{K}^\lambda)(t,u,u_\star,\sigma)\overline{m}_{g_\lambda\star \mu}(dt,du,du_\star,d\sigma)\\[1ex]&\hspace{1cm} \le \int_{E\times \rrd\times\rrd} (\tau-1)(K(t,v,v_\star,\sigma))B(v-v_\star)g_\lambda(u-v)g_\lambda(u_\star-v_\star)dt\mu_t(dv)\mu_t(dv_\star)d\sigma \\[1ex]& \hspace{1cm} =\int_E (\tau-1)(K(t,v,v_\star,\sigma))B(v-v_\star)dt\mu_t(dv)\mu_t(dv_\star)d\sigma \\[1ex] & \hspace{1cm} =\int_E(\tau-1)(K(t,v,v_\star,\sigma))\overline{m}_{\mu}(dt,dv,dv_\star,d\sigma).\end{split} \end{equation} Returning to (\ref{eq: decomposition of dynamic cost 1}) and using the analagous equation for $\mu_\bullet$, we finally obtain the decomposition \begin{equation}\label{eq: decomposition of dynamic cost 6} \mathcal{J}(g_\lambda\star \mu_\bullet, g_\lambda \star w) \le \mathcal{J}(\mu_\bullet, w)+(\overline{m}_{g_\lambda\star \mu}-\overline{m}_\mu)(E)+\mathcal{T}_\lambda(\mu_\bullet, w).\end{equation} It is very straightforward to show that the second term converges to $0$ as $\lambda\to 0$, with a rate depending only on $\sup_t \langle |v|^2, \mu_t\rangle$. \paragraph*{Step 4: Analysis of $ \psi_\lambda$} We now turn to the error term $\mathcal{T}_\lambda$ identified in (\ref{eq: bad term 2}), which depends on the continuity of $\log B$. We remark first that this term cannot be avoided purely on general considerations; consider, for example, the kernel $B(v)=1_{|v|\ge 1}$ in which case $\mathcal{J}(g_\lambda\star \mu_\bullet, g_\lambda\star w)$ can become infinite due to contributions in the region $\{|v-v_\star|<1\}$. \bigskip \\ We start with an upper bound for $ \psi_\lambda$. We define first the measures \begin{equation} \label{eq: analysis of fd 1}  \xi^\lambda_{(t,u)}(dv):=\frac{g_\lambda(v-u)\mu_t(dv)}{(g_\lambda\star \mu_t)(u)}.\end{equation} Setting $R:=\sqrt{2\sup_s \langle |v|^2, \mu_s\rangle}$, a simple Chebychev inequality shows that $\mu_t(|v|\le R)\ge \frac{1}{2}$ for all $t$, which leads to the lower bound \begin{equation} (g_\lambda\star \mu_t)(u)\ge \frac{1}{2}\exp\left(-(|u|+R)^2)/2\lambda\right)/(2\pi \lambda)^{d/2}.\end{equation} We now estimate \begin{equation}\label{eq: analysis of fd 2} \begin{split} \left \langle |v|^2 1[|v-u|>R+|u|], \xi^\lambda_{(t,u)}\right \rangle &=\frac{\int_{|v-u|>R+|u|} e^{-|v-u|^2/2\lambda}|v|^2 \mu_t(dv)/(2\pi \lambda)^{d/2}}{(g_\lambda\star \mu_t)(u)}\\[2ex] &\le \frac{e^{-(|u|+R)^2/2\lambda}/(2\pi \lambda)^{d/2}}{(g_\lambda\star \mu_t)(u)} \int_{|v-u|>R+|u|}|v|^2\mu_t(dv)\\[3ex] &\le 2\sup_s \langle |v|^2, \mu_s\rangle = R^2.\end{split} \end{equation} Together with a trivial bound for the remaining region, we conclude that \begin{equation}\label{eq: analysis of fd 3} \begin{split} \left\langle |v|^2, \xi^{\lambda}_{(t,u)}\right\rangle \le 3R^2 +4|u|^2 \le C(1+|u|^2) \end{split} \end{equation} since $R\ge 1$ depends only on the quantities in (\ref{eq: controlling quantities}). Using the lower bound $B(u-u_\star)\ge 1$ and the upper bound $B(v-v_\star)\le 1+|v|+|v_\star| \le C(1+|v|^2+|v_\star|^2)$, we now return to (\ref{eq: identification of Kdelta 7}) to obtain \begin{equation}\label{eq: analysis of fd 4}  \psi_\lambda(t,u,u_\star) \le C(1+|u|^2+|u_\star|^2).\end{equation} This bound will be useful in general, for $(u,u_\star)$ where $ \psi_\lambda$ cannot be shown to be close to $1$. We complement this with a bound which will show that, for most points $(u, u_\star)$ in the support of $(g_\lambda\star \mu_t)^{\otimes 2}$, $ \psi_\lambda$ is not too much bigger than 1.  We will exploit, repeatedly, the observation that, for the kernels given in (\ref{eq: form of B}), \begin{equation} \label{eq: analysis of fd 5}  B(v-v_\star)\le (1+|v-u|+|v_\star-u_\star|)B(u-u_\star).\end{equation} Fix $t\ge 0$, and suppose that $v^0, v^0_\star$ are such that there exists a sets $U\ni v^0, U_\star \ni v^0_\star$ of diameter $c\sqrt{\lambda}$ and $\mu_t(U), \mu_t(U_\star)\ge \epsilon \lambda^{d/2}$, and $(u, u_\star)$ is such that $|(u,u_\star)-(v^0,v^0_\star)|<x\sqrt{\lambda}$, for some constant $c$ and parameters $\epsilon>0, x<\infty$ to be chosen later. In this case, we bound the denominator below by observing that \begin{equation} \label{eq: analysis of fd 6} (g_\lambda \star \mu_t)(u)\ge \frac{1}{(2\pi \lambda)^{d/2}} e^{-(x\sqrt \lambda+c\sqrt{\lambda})^2/2\lambda} \mu(U)\ge \frac{\epsilon}{(2\pi)^{d/2}} \exp\left(-x^2-c^2\right) \end{equation} and similarly for $u_\star$. We now return to (\ref{eq: identification of Kdelta 7}) and split the integral defining $ \psi_\lambda$ into the regions $E_1=\{(v,v_\star):|(v,v_\star)-(u, u_\star)|< \lambda^{1/3}+3x\sqrt{\lambda}\}$ and its complement $E_2$. On $E_1$, \begin{equation}\label{eq: analysis of fd 6} B(v-v_\star)\le (1+2\lambda^{1/3}+6x\sqrt{\lambda})B(u-u_\star)\end{equation} and so \begin{equation}\begin{split} \label{eq: analysis of fd 7} &\int_{E_1} \h \frac{B(v-v_\star)}{B(u-u_\star)} \h \frac{g_\lambda(v-u)\mu_t(dv)}{(g_\lambda\star \mu_t)(u)}\h  \frac{g_\lambda(v_\star-u_\star)\mu_t(dv_\star)}{(g_\lambda\star \mu_t)(u_\star)}\le 1+2\lambda^{1/3}+6x\sqrt{\lambda}.\end{split} \end{equation} On the other hand, using (\ref{eq: analysis of fd 6}) and the trivial bound $B(v-v_\star)/B(u-u_\star)\le C(1+|v|+|v_\star|)$, we bound the term from the second region by \begin{equation}\label{eq: analysis of fd 8}\begin{split} &  \int_{E_2} \h \frac{B(v-v_\star)}{B(u-u_\star)} \h \frac{g_\lambda(v-u)\mu_t(dv)}{(g_\lambda\star \mu_t)(u)}\h  \frac{g_\lambda(v_\star-u_\star)\mu_t(dv_\star)}{(g_\lambda\star \mu_t)(u_\star)} \\[1ex] & \hspace{3cm}  \le C \int_{\rrd\times \rrd} (1+|v|+|v_\star|)\lambda^{-d}\exp\left(-\frac{1}{2\lambda}(9x^2\lambda+\lambda^{2/3})\right)\epsilon^{-2}  \\ &\hspace{5cm}\dots\times \exp(2x^2+2c^2)\mu_t(dv)\mu_t(dv_\star)\\[2ex]&\hs\hs\hs \le C\exp\left(2c^2-x^2\right)\epsilon^{-2}. \end{split}  \end{equation} where, in the final line, we recall that $\lambda^{-d}\exp(-\lambda^{-1/6}/{2})$ is uniformly bounded on $(0,\infty)$, and that the remaining integral is controlled in terms of the second moments of $\mu_t$, so can be absorbed into $C$. Gathering (\ref{eq: analysis of fd 7}, \ref{eq: analysis of fd 8}), we conclude that for $(u, u_\star), x, \epsilon$ as above, \begin{equation} \label{eq: analysis of fd 9} \begin{split} \psi_\lambda(t, u,u_\star)&\le 1+2\lambda^{1/3}+6x\sqrt{\lambda}+C\exp(2c^2-x^2)\epsilon^{-2}\\&\le (1+2\lambda^{1/3}+6x\sqrt{\lambda})\left(1+C\exp(2c^2-x^2)\epsilon^{-2}\right). \end{split}\end{equation} \paragraph*{Step 5: Analysis of $\mathcal{T}_\lambda$} Equipped with this preliminary analysis of $ \psi_\lambda$ in the previous step, we bound the final term $\mathcal{T}_\lambda$ appearing in (\ref{eq: decomposition of dynamic cost 6}). Together with the observation under (\ref{eq: decomposition of dynamic cost 6}), this suffices to prove (\ref{eq: dynamic cost conclusion}) and finish the proof of the lemma. \bigskip \\  We break up the integration space $E\times \rrd\times \rrd$ in the definition of $\mathcal{T}_\lambda$ as follows. For $M \ge 2, R\in [3,\infty)\cap \sqrt{\lambda}\mathbb{N}$,  $x \in (0,\infty), \epsilon\in (0,\infty)$ to be chosen later, we form a partition $\mathfrak{P}$ of $(-R,R]^d$ into $(2R/\sqrt{\lambda})^d$ translates of $(0,\sqrt{\lambda}]^d$, and for $v\in (-R,R]^d$, write $\mathcal{B}(v)$ for the unique $\mathcal{B}\in \mathfrak{P}$ containing $v$. We now consider the partition of $E\times\rrd\times\rrd$ given by \begin{equation*}\hspace{-0.5cm} A_1:=\left\{v, v_\star\in (-R, R]^d, K\le M, \mu_t(\mathcal{B}(w))\ge \epsilon \lambda^{d/2}, \mu_t(\mathcal{B}(v_\star))\ge \epsilon \lambda^{d/2}, |(u,u_\star)-(v,v_\star)|<x\sqrt{\lambda} \right\};\end{equation*}
\begin{equation*}\hspace{-0.5cm} A_2:=\left\{v, v_\star\in (-R, R]^d, K\le M, \mu_t(\mathcal{B}(w))\ge \epsilon \lambda^{d/2}, \mu_t(\mathcal{B}(v_\star))\ge \epsilon \lambda^{d/2}, |(u,u_\star)-(v,v_\star)|\ge x\sqrt{\lambda} \right\};\end{equation*} \begin{equation*}A_3:=\left\{v,v_\star\in (-R,R]^d, K\le M, \mu_t(\mathcal{B}(v))< \epsilon \lambda^{d/2}\text{ or } \mu_t(\mathcal{B}(v_\star))< \epsilon \lambda^{d/2}\right\};\end{equation*} \begin{equation*} A_4:=\left\{v,v_\star\in (-R,R]^d, K(t,v,v_\star,\sigma)> M\right\};\end{equation*} 
 \begin{equation}\label{eq: decomposition 5} A_5:=\left\{(v,v_\star)\not\in (-R,R]^{2d}\right\}.\end{equation} We analyse the contributions from these regions one-by-one. Roughly, $A_1$ is the `good' region, containing most of the contributions from the integrating measure, where $\log  \psi_\lambda$ is small by (\ref{eq: analysis of fd 9}), and the remaining terms are small, depending on the parameters $M, R, x, \epsilon$; at the end, we will optimise, so that $M, R, x\to \infty$ and $\epsilon\to 0$ as functions of $\lambda \to 0$. \paragraph*{Step 5a: Contribution from $A_1$} For the region $A_1$, we observe that the hypotheses leading to (\ref{eq: analysis of fd 9}) hold, with $U=\mathcal{B}(v), U_\star=\mathcal{B}(v_\star)$ and $c=\sqrt{d}$ is an absolute constant. Further, $K<M$ and $B(v-v_\star)\le 1+2R \le CR$, so we integrate (\ref{eq: analysis of fd 9}) to find \begin{equation} \label{eq: contribution from A1}\begin{split}& \int_{A_1}K(t,v,v_\star,\sigma)B(v-v_\star)g_\lambda(u-v)g_\lambda(u_\star-v_\star)\log  \psi_\lambda(t,u,u_\star)dt\mu_t(dv)\mu_t(dv_\star)d\sigma du du_\star \\& \hs \hs \hs \hs \hs \le CMR\left(\log (1+2\lambda^{1/3}+6x\sqrt{\lambda})+\log (1+\epsilon^{-2}e^{-x^2})\right).\end{split}\end{equation}  \paragraph*{Step 5b: Contribution from $A_2$} For $A_2$, we use the general upper bound (\ref{eq: analysis of fd 4}), which is valid without restriction on $u, u_\star$. For fixed $v,v_\star$ we use H\"older's inequality to see that \begin{equation}\label{eq: contribution from A2 1} \begin{split} &\int_{\rrd\times\rrd} \log  \psi_\lambda(t,u,u_\star)1[|(u,u_\star)-(v,v_\star)|\ge x\sqrt{\lambda}]g_\lambda(u-v)g_\lambda(u_\star-v_\star)dudu_\star \\& \hs \hs \hs \le C\int_{\rrd\times\rrd} \log(1+|u|^2+|u_\star|^2)1[|(u,u_\star)-(v,v_\star)|\ge x\sqrt{\lambda}] \\& \hspace{5cm}\dots\times g_\lambda(u-v)g_\lambda(u_\star-v_\star)dudu_\star \\ & \hs \hs \le C\left(\int_{\rrd\times\rrd} (1+|u|^2+|u_\star|^2)g_\lambda(u-v)g_\lambda(u_\star-v_\star)dudu_\star\right)^{1/2} \\ & \hs\hs\hs\hs\hs\dots\times \left(\int_{\rrd\times\rrd} 1[|(u', u'_\star)|>x\sqrt{\lambda}]g_\lambda(u')g_\lambda(u'_\star)du'du'_\star\right)^{1/2}\\[2ex] & \hs\hs \le C(1+|v|+|v_\star|) \exp\left(-x^2/16d\right)\end{split} \end{equation} where for the first factor in the final line, we integrated $\int |u|^2 g_\lambda(u-v)du =|u|^2+d\lambda^2\le |v|^2+d$, and for the second factor we used standard tail estimates for the normal distribution, absorbing constants into the prefactor $C$. Bounding $B(v-v_\star), K$ as above, and integrating over $t, v, v_\star,\sigma$, we find \begin{equation} \label{eq: contribution from A2} \begin{split}&\int_{A_2}K(t,v,v_\star,\sigma)B(v-v_\star)g_\lambda(u-v)g_\lambda(u_\star-v_\star)\log  \psi_\lambda(t, u, u_\star)dt \mu_t(dv)\mu_t(dv_\star)d\sigma du du_\star \\ & \hspace{5cm} \le CMR\exp\left(-x^2/16d\right).\end{split} \end{equation} \paragraph*{Step 5c: Contribution from $A_3$} Similarly to the previous step, we start from a bound on the integrals over $u, u_\star$, with $t$ and $v,v_\star$ fixed. We split the integral over $u, u_\star$ into $\{|(u, u_\star)-(v, v_\star)|\le (1+|v|+|v_\star|)\}$ and $\{|(u, u_\star)-(v, v_\star)|> (1+|v|+|v_\star|)\}$. In the first region, thanks to (\ref{eq: analysis of fd 5}), \begin{equation} \label{eq: contribution from A3 1} \log  \psi_\lambda(t, u, u_\star)\le C\log (1+|u|^2+|u_\star|^2) \le C\log (1+|v|^2+|v_\star|^2)\end{equation} while the contribution from the second region is controlled by using H\"older's inequality in the same way as (\ref{eq: contribution from A2 1}) to obtain \begin{equation}\begin{split}\label{eq: contribution from A3 2} &\int_{|(u,u_\star)-(v,v_\star)|>(1+|v|+|v_\star|)} \log  \psi_\lambda(t,u,u_\star) g_\lambda(u-v)g_\lambda(u_\star-v_\star)dudu_\star \\ & \hspace{5cm}\le C(1+|v|+|v_\star|)\exp(-(1+|v|+|v_\star|)^2/16d).\end{split} \end{equation} This term can be absorbed into the contribution from (\ref{eq: contribution from A3 1}), and we conclude that, for all $t$ and  $v,v_\star \in \rrd$, \begin{equation}\label{eq: contribution from A3 3} \int_{\rrd\times\rrd} \log  \psi_\lambda(t,u,u_\star)g_\lambda(u-v)g_\lambda(u_\star-v_\star)dudu_\star\le C\log (1+|v|^2+|v_\star|^2). \end{equation} In particular, when $v,v_\star\in (-R,R]^d$, the right-hand side can be replaced by $C\log R$, and $B(v-v_\star)\le CR$. We now integrate over $A_3$ to find \begin{equation} \label{eq: contribution from A3 4} \begin{split} &\int_{A_3} K(t,v,v_\star,\sigma)B(v-v_\star)g_\lambda(u-v)g_\lambda(u_\star-v_\star)\log  \psi_\lambda(t,u,u_\star)dt \mu_t(dv)\mu_t(dv_\star)d\sigma du du_\star \\ &  \le CM(R\log R)\int_0^T \mu_t^{\otimes 2}\left((v, v_\star)\in (-R, R]^{2d}: \mu_t(\mathcal{B}(v))<\epsilon \lambda^{d/2} \text{ or }\mu_t(\mathcal{B}(v_\star))<\epsilon \lambda^{d/2}\right)dt \\&\hs\hs \le CM(R\log R)\int_0^T\mu_t\left(v\in (-R,R]^d: \mu_t(\mathcal{B}(v))<\epsilon \lambda^{d/2}\right)dt\end{split} \end{equation} where the last line follows using a union bound, absorbing the factor of $2$ into $C$. The integrand is now \begin{equation} \label{eq: contribution from A3 5} \begin{split} \mu_t\left(v\in (-R,R]^d: \mu_t(\mathcal{B}(v))<\epsilon \lambda^{d/2}\right)&=\sum_{\mathcal{B}\in \mathfrak{P}} \mu_t(\mathcal{B})1(\mu_t(\mathcal{B})<\epsilon \lambda^{d/2}) \\[2ex] & \le \epsilon \lambda^{d/2} (\#\mathfrak{P})=\epsilon (2R)^{d}\end{split} \end{equation} recalling that $\#\mathfrak{P}=(2R/\sqrt{\lambda})^d$. Substituting this bound back into (\ref{eq: contribution from A3 4}) we conclude that \begin{equation} \label{eq: contribution from A3} \begin{split} &\int_{A_3} K(t,v,v_\star,\sigma)B(v-v_\star)g_\lambda(u-v)g_\lambda(u_\star-v_\star)\log  \psi_\lambda(t,u,u_\star)dt \mu_t(dv)\mu_t(dv_\star)d\sigma du du_\star \\ & \hs \hs \le CM(R^{d+1}\log R)\epsilon.\end{split} \end{equation} \paragraph*{Step 5d: Contribution from $A_4$} In $A_4$, we use the same bound (\ref{eq: contribution from A3 3}) on $\int \log  \psi_\lambda g_\lambda(u-v)g_\lambda(u_\star-v_\star)dudu_\star$, and observe that, on $A_4$, $K(t,v,v_\star,\sigma)\le \frac{M}{\tau(M)} \tau(K)$. Integrating, it follows that \begin{equation} \label{eq: contribution from A4} \begin{split} & \int_{A_4}K(t,v,v_\star,\sigma)B(v-v_\star)g_\lambda(u-v)g_\lambda(u_\star-v_\star)\log  \psi_\lambda(t,u,u_\star)dt \mu_t(du)\mu_t(dv_\star)d\sigma du du_\star \\& \hs \le C(\log R)\left(\frac{M}{\tau(M)}\right)\int_E \tau(K(t,v,v_\star,\sigma))\overline{m}_\mu(dt,dv,dv_\star,d\sigma) \\[2ex] & \hs =C(\log R)\left(\frac{M}{\tau(M)}\right) \mathcal{J}(\mu_\bullet, w)=C(\log R)\left(\frac{M}{\tau(M)}\right)\end{split} \end{equation} since $C$ is allowed to depend on an upper bound for $\mathcal{J}(\mu_\bullet, w)$. \paragraph*{Step 5e: Contribution from $A_5$} We finally turn to the contribution from $A_5$. Thanks to (\ref{eq: contribution from A3 3}), for any $(v,v_\star)\not\in (-R,R]^{2d}$, we have \begin{equation} \label{eq: contribution from A5 1}\begin{split} \int_{\rrd\times\rrd} \log  \psi_\lambda(t,u,u_\star)g_\lambda(u-v)g_\lambda(u_\star-v_\star)du du_\star & \le C \log (1+|v|^2+|v_\star|^2) \\ & \hspace{-1cm} \le C \frac{\log (R^2)}{R^2}(1+|v|^2+|v_\star|^2)\end{split} \end{equation} as $(\log x)/x$ is decreasing on $[R,\infty)\subset [e,\infty)$, and $1+|v|^2+|v_\star|^2\ge R^2$. Integrating over $t,v,v_\star,\sigma$, we obtain \begin{equation} \label{eq: contribution from A5} \begin{split} &\int_{A_5}K(t,v,v_\star,\sigma)B(v-v_\star)g_\lambda(u-v)g_\lambda(u-v_\star)\log  \psi_\lambda(t,u,u_\star)dt \mu_t(dv)\mu_t(dv_\star)d\sigma du du_\star \\& \hs \le C\left(\frac{\log R^2}{R^2}\right)\int_E (|v|^2+|v_\star|^2)K(t,v,v_\star,\sigma)\overline{m}_\mu(dt,dv,dv_\star,d\sigma) \\[1ex]& \hs =C\left(\frac{\log R}{R^2}\right)\langle |v|^2+|v_\star|^2, w\rangle = C\left(\frac{\log R}{R^2}\right)  \end{split} \end{equation} recalling again that the second moment of $w$ is one of the quantities (\ref{eq: controlling quantities}) on which $C$ is allowed to depend. \paragraph*{Step 5f: Conclusion} Gathering (\ref{eq: contribution from A1}, \ref{eq: contribution from A2}, \ref{eq: contribution from A3}, \ref{eq: contribution from A4}, \ref{eq: contribution from A5}), we conclude that, for any $M,R,x, \epsilon$ as above, \begin{equation}\begin{split}  \mathcal{T}_\lambda(\mu_\bullet, w)& \le C\bigg(MR \log (1+2\lambda^{1/3}+6x\sqrt{\lambda})+MR\log\left(1+e^{-x^2}\epsilon^{-2}\right)\\ & \hs \dots + MR e^{-x^2/16d}+M(R^{d+1}\log R)\epsilon +(\log R)\left(\frac{M}{\tau(M)}\right)+\frac{\log R}{R^2}\bigg).\end{split}\end{equation} We now define $\vartheta_1(\lambda)$ to be the infimum of the term in parantheses over the possible choices of $M,R,\epsilon, x$ described at the start of Step 5 for $\lambda>0$, and $\vartheta_1(0)=0$. Although this expression is somewhat complicated to optimise directly, it is straightforward to see that $\vartheta_1(\lambda)\to 0$ as $\lambda\downarrow 0$: given a target $\eta>0$, we can choose $R$ such that the last term is at $<\eta/6$, independently of $M, \epsilon, x, \lambda$; with $R$ thus fixed, we choose $M$ such that the second-last term is $<\eta/6$ for all $\epsilon, x, \lambda$, and so on. Returning to (\ref{eq: decomposition of dynamic cost 6}), one obtains an additional error, corresponding to the term $\overline{m}_{g_\lambda\star \mu}(E)-\overline{m}_\mu(E)$, which can easily be controlled, giving another term $C\vartheta_2(\lambda)$. Adding the two, the lemma is proven, with a new function $\vartheta$.  \end{proof} \begin{proof}[Proof of Lemma \ref{lemma: approximation 1}] We now prove Lemma \ref{lemma: approximation 1} based on the following truncation argument. \paragraph*{Step 1: Definition} For $n\ge 1$, let $B_n$ be the set $\{v\in \rrd: |v|\le n\}$ and \begin{equation} E_n=\{(t,v,v_\star,\sigma): B(v-v_\star)K(t,v,v_\star,\sigma) \le n,\text{ and }v, v_\star, v', v'_\star \in B_n\}\end{equation}  and define measures $\nu^{(n)}$ on $\rrd$ by specifying, for all bounded and measureable $f:\rrd\to\rr$ \begin{equation} \langle f, \nu^{(n)}\rangle=\int_{E_n^\mathrm{c}} (f(v')1_{v'\in B_n}+f(v_\star')1_{v_\star'\in B_n})\hspace{0.1cm}w(dt,dv,dv_\star,d\sigma) \end{equation} and\begin{equation} c_n=(\mu_0+\nu^{(n)})(\rrd). \end{equation} We define a new flux\begin{equation}\begin{split} w^{(n)}(dt,dv,dv_\star,d\sigma)&=c_n^{-1}\hspace{0.1cm}K(t,v,v_\star,\sigma)\hspace{0.1cm}1_{E_n}\hspace{0.1cm}\overline{m}_\mu(dt,dv,dv_\star,d\sigma)\\& =c_n^{-1}1_{E_n}w(dt,dv,dv_\star,d\sigma) \end{split}\end{equation}  Let $\mu^{(n)}_\bullet$ be given by  \begin{equation} \mu^{(n)}_t=\frac{\mu_0+\nu^{(n)}}{c_n}+\int_E \Delta(v,v_\star,\sigma) \hspace{0.1cm}1_{s\le t}\hspace{0.1cm}w^{(n)}(ds,dv,dv_\star,d\sigma). \end{equation} This definition gives a signed measure with $\mu^{(n)}_t1_{B_n^\mathrm{c}}=c_n^{-1}\mu_01_{B_n^\mathrm{c}}\ge 0$, and we further observe that for any Borel $A\subset B_n$ and $t\le T$, \begin{equation} \begin{split} \mu^{(n)}_t(A) &\ge c_n^{-1}\bigg(\mu_0(A)+\int_{E}\left((\Delta 1_A)(v,v_\star,\sigma)1_{E_n}+(1_A(v')+1_A(v'_\star))1_{E_n^\mathrm{c}}\right)\\& \hspace{7cm}\dots\times 1_{s\le t} \hspace{0.1cm} w(ds,dv,dv_\star,d\sigma)\bigg)   \\[1ex] & \ge c_n^{-1}\left(\mu_0(A)+\int_{E}\Delta1_A(v,v_\star,\sigma)1_{s\le t} \hspace{0.1cm} w(ds,dv,dv_\star,d\sigma)\right) = c_n^{-1}\mu_t(A). \end{split} \end{equation} It follows that $\mu^{(n)}_t$ is a positive measure for all $t \le T$, and thanks to the normalisation by $c_n$, it follows that $\mu^{(n)}_t$ is a probability measure. Moreover, it also follows that  $\mu_t1_{B_n}$ is absolutely continuous with respect to $\mu^{(n)}_t$, and that \begin{equation}\label{eq: bound on density} \frac{d(\mu_t1_{B_n})}{d\mu^{(n)}_t} \le c_n \hs \mu^{(n)}_t\text{-almost everywhere}\end{equation} and the form (\ref{eq: form of starting point}) of $\mu^{(n)}_0$ is immediate by construction.  Further, the conntinuity equation (\ref{eq: CE}) follows immediately by construction.  \paragraph*{Step 2: Convergence of the Truncated Measure-Flux} Firstly, we show that $\mu^{(n)}_\bullet$ approximates $\mu_\bullet$ uniformly in a weighted total variation norm. At time $0$, \begin{equation} \label{eq: weighted TV error at 0}\begin{split} \|(1+|v|^2)(\mu^{(n)}_0-\mu_0)\|_\mathrm{TV} \le c_n^{-1}\langle 1+|v|^2, \nu^{(n)}\rangle  +\frac{|1-c_n|}{c_n}\langle 1+|v|^2, \mu_0\rangle. \end{split} \end{equation} In the first term,\begin{equation} \langle 1+|v|^2, \nu^{(n)}\rangle = \int_{E} ((1+|v'|^2)1_{v'\in B_n}+(1+|v_\star'|^2)1_{v_\star\in B_n})1_{E^\mathrm{c}_n}w(dt,dv,dv_\star,d\sigma)\to 0\end{equation} by applying dominated convergence: the integrand is at most $2(1+|v|^2+|v_\star|^2)$ by energy conservation, and by hypothesis, $\langle 2+|v|^2+|v_\star|^2, w\rangle<\infty$. It follows already from these estimates that $c_n\to 1$, and the second term on the right-hand side of (\ref{eq: weighted TV error at 0}) converges to $0$.  Similarly, we estimate \begin{equation} \begin{split}\left\|(1+|v|^2)((\mu^{(n)}_t-\mu^{(n)}_0)-(\mu_t-\mu_0))\right\|_\mathrm{TV}&\le 4c_n^{-1}\int_{E_n^\mathrm{c}} (1+|v|^2+|v_\star|^2)1_{E^\mathrm{c}_n}w(ds,dv,dv_\star,d\sigma) \\& +4|1-c_n^{-1}|\int_E (1+|v|^2+|v_\star|^2)w(ds,dv,dv_\star,d\sigma) \\ & \hs\hs \rightarrow 0 \end{split} \end{equation} and we conclude that $\sup_{t\le T} \|(1+|v|^2)(\mu^{(n)}_t-\mu_t)\|_\mathrm{TV}\rightarrow 0$. A similar argument shows that \begin{equation}\|(1+|v|^2+|v_\star|^2)(w^{(n)}-w)\|_\mathrm{TV}\to 0\end{equation}  as desired. \paragraph*{Step 3: Tilting Function for the Truncated Pair}We now construct the tilting function $K^{(n)}$, which completes the proof that $\mu^{(n)}_\bullet, w^{(n)}$ is a measure-flux pair. By construction, we have \begin{equation}\mu^{(n)}_t-\mu^{(n)}_0=\int_E \Delta(v,v_\star,\sigma)1_{s\le t}w^{(n)}(ds,dv,dv_\star,d\sigma) \end{equation} and\begin{equation} \begin{split} w^{(n)}(dt,dv,dv_\star,d\sigma)&=c_n^{-1}K(t,v,v_\star,\sigma)1_{E_n}(t,v,v_\star,\sigma)B(v-v_\star,d\sigma)\mu_t(dv)\mu_t(dv_\star)dt  \\ &\hspace{-1cm}=c_n^{-1}K(t,v,v_\star,\sigma)1_{E_n}(t,v,v_\star,\sigma)B(v-v_\star,d\sigma)(\mu_t1_{B_n})(dv)(\mu_t1_{B_n})(dv_\star)dt\end{split}\end{equation} where, in the last line, we observe that $1_{E_n}=1_{E_n}1_{B_n}(v)1_{B_n}(v_\star)$. Recalling the absolute continuity (\ref{eq: bound on density}), we have  \begin{equation} w^{(n)}(dt,dv,dv_\star,d\sigma)=K^{(n)}(t,v,v_\star,\sigma)\overline{m}_{\mu^{(n)}}(dt,dv,dv_\star,d\sigma) \end{equation} where $K^{(n)}$ is given by \begin{equation} K^{(n)}(t,v,v_\star,\sigma)=c_n^{-1}K(t,v,v_\star,\sigma)1_{E_n}(t,v,v_\star,\sigma)\left(\frac{d(\mu_t1_{B_n})}{d\mu^{(n)}_t}\right)(v)\left(\frac{d(\mu_t1_{B_n})}{d\mu^{(n)}_t}\right)(v_\star).\end{equation} From (\ref{eq: bound on density}) and the definition of $E_n$, \begin{equation} \label{eq: bound on Kn} B(v-v_\star)K^{(n)}(t,v,v_\star,\sigma)\le nc_n\end{equation} is bounded, as claimed. \paragraph*{Step 4: Convergence of the Dynamic Cost} It remains to show that $\mathcal{J}(\mu^{(n)}, w^{(n)})\rightarrow \mathcal{J}(\mu, w)$. From the total variation convergence proven above, it follows that $(\mu^{(n)}_\bullet, w^{(n)})\to (\mu_\bullet,w)$ in the topology of $\D\times \Mm$. This implies that $\liminf_n \mathcal{J}(\mu^{(n)}, w^{(n)})\ge \mathcal{J}(\mu, w)$ by lower semicontinuity (Lemma \ref{lemma: semiconntinuity}), and so it suffices to prove an upper bound. We start by observing that, by construction \begin{equation}\begin{split} (\tau-1)(K^{(n)})\overline{m}_{\mu^{(n)}}(dt,dv,dv_\star,d\sigma)&=(\log K^{(n)}-1)w^{(n)}(dt,dv,dv_\star,d\sigma) \\ & =c_n^{-1} 1_{E_n}(\log K^{(n)}-1)w(dt,dv,dv_\star,d\sigma)\end{split}\end{equation} and that, on $E_n$, $K^{(n)}\le c_n K$, so \begin{equation}\begin{split} (\tau-1)(K^{(n)})\overline{m}_{\mu^{(n)}}(dt,dv,dv_\star,d\sigma) &\le c_n^{-1}(\log K-1+\log c_n)1_{E_n}w(dt,dv,dv_\star,d\sigma) \\[1ex]&=c_n^{-1}((\tau-1)(K)+K\log c_n)1_{E_n}\overline{m}_{\mu}(dt,dv,dv_\star,d\sigma).\end{split} \end{equation} Integrating, and recalling the definition of $\mathcal{J}$, we see that \begin{equation}\begin{split} \mathcal{J}(\mu^{(n)}_\bullet, w^{(n)})-\overline{m}_{\mu^{(n)}}(E)&=\int_E(\tau-1)(K^{(n)})\overline{m}_{\mu^{(n)}}(dt,dv,dv_\star,d\sigma) \\[1ex] & \le c_n^{-1}\left(\mathcal{J}(\mu_\bullet, w)-\overline{m}_{\mu}(E)-\int_{E_n^\mathrm{c}}(\tau-1)(K)\overline{m}_\mu(dt,dv,dv_\star,d\sigma)\right)\\&\hs +(c_n^{-1}\log c_n )w(E).\end{split}\end{equation} We now take the limit superior of both sides. On the left-hand side, it is straightforward to see, using the weighted total variation convergence, that $\overline{m}_{\mu^{(n)}}(E)\to \overline{m}_{\mu}(E)$, while the third term on the first line of the right-hand side converges to $0$ by dominated convergence, and the final term converges to $0$ because $c_n\to 1$ and $w(E)<\infty$. We conclude that $\limsup_n \mathcal{J}(\mu^{(n)}_\bullet, w^{(n)})\le \mathcal{J}(\mu, w)$ and we are done.  \end{proof} Combining the previous two results, we prove Lemma \ref{lemma: approximation 1.5}. The main difficulty with the construction above is that the presence of $\nu^{(n)}$ may make the cost of the initial data large: \emph{a priori} $\nu^{(n)}$ could be singular, which would give $H(\mu^{(n)}_0|\mu_0^\star)=\infty$. To avoid this, we will convolve with the mollifiers $g_\lambda$, at a scale $\lambda=\lambda_n$ to be chosen. For this reason, it is important to have the uniform convergence of the cost function in Lemma \ref{lemma: convolution}. \begin{proof}[Proof of Lemma \ref{lemma: approximation 1.5}] Let $\mu_\bullet, w$ be as given, and let $\mu^{(n,0)}, w^{(n,0)}$ be the approximations produced by Lemma \ref{lemma: approximation 1}. We observe first that, thanks to the strong convergence (\ref{eq: Approx 1 1}), \begin{equation} \label{eq: controlling quantities uniform in trunctation} \sup_n \langle 1+|v|^2+|v_\star|^2, w^{(n,0)}\rangle;\qquad \sup_n\mathcal{J}(\mu^{(n,0)}_\bullet, w^{(n,0)}); \qquad \sup_n \sup_{t\le T} \langle |v|^2, \mu^{(n,0)}_t\rangle \end{equation} are all finite, uniformly in $n$. For any $\lambda>0$, let $\mu^{(n,\lambda)}_\bullet, w^{(n,\lambda)}$ be the convolutions \begin{equation} \mu^{(n,\lambda)}_\bullet:=g_\lambda\star \mu^{(n,0)}_\bullet;\qquad w^{(n,\lambda)}:=g_\lambda\star w^{(n,0)}.\end{equation} Thanks to Lemma \ref{lemma: convolution} and (\ref{eq: controlling quantities uniform in trunctation}), there exists some $C$, uniform in $n$, such that \begin{equation}\label{eq: uniform approximation of dynamic cost} \mathcal{J}(\mu^{(n,\lambda)}, w^{(n,\lambda)})\le \mathcal{J}(\mu^{(n,0)},w^{(n,0)})+C\vartheta(\lambda).\end{equation} We consider now the cost due to the initial data. Firstly, we write \begin{equation} \mu^{(n,0)}_0=(1-p_n)(g_\lambda\star \mu_0)+p_n(g_\lambda\star \xi^{(n)})\end{equation} with $p_n=\nu^{(n)}(\rrd)/c_n \to 0$, and $\xi^{(n)}:=\nu^{(n)}/\nu^{(n)}(\rrd)$. Using the convexity of $H(\cdot|\mu_0^\star)$, we immediately have \begin{equation} \label{eq: initial cost under convolution  1} H\left(\mu^{(n,\lambda)}\big|\mu_0^\star\right)\le (1-p_n)H(g_\lambda \star \mu_0|\mu_0^\star)+p_nH(g_\lambda\star \xi^{(n)}|\mu_0^\star). \end{equation} We investigate these terms one at a time. \paragraph*{Step 1: Entropy of $H(g_\lambda\star\mu_0|\mu_0^\star)$} We first show that \begin{equation}\label{eq: entropy of convolved term conclusion} \limsup_{\lambda \to 0} H(g_\lambda\star \mu_0|\mu_0^\star)\le H(\mu_0|\mu_0^\star)\end{equation} where we recall that, since $\mathcal{I}(\mu_\bullet, w)<\infty$ by hypothesis, the right-hand side is a finite limit. Since $\mu_0$ is absolutely continuous with respect to  $\mu_0^\star$, it is absolutely continuous with respect to the Lebesgue measure; let us write $f_0$ for its density, and recall the notation $f_0^\star$ for the density of $\mu_0^\star$. We can then write $H(\mu_0|\mu_0^\star)=\int f_0 \log (f_0/f_0^\star)<\infty$, and, recalling that $f_0^\star \ge ce^{-z_3|v|^2}$ by Hypothesis \ref{hyp: condition on ref ms}iii), $\log f_0\le \log (f_0/f_0^\star) -\log c + z_3|v|^2$. Since $\mu_0$ has a finite second moment, we see that $\int f_0 \log f_0 dv<\infty$. Further, bounding $-\log f_0^\star \le \log c+ z_3|v|^2$ and $(\log f_0^\star)1(f_0^\star\ge 1)f_0\le f_0^\star \exp(1(f_0^\star\ge 1))+f_0\log f_0$, we conclude that $\int |\log f_0^\star|f_0<\infty$. We now write, as a difference of finite integrals,  \begin{equation}\begin{split} H(\mu_0|\mu_0^\star)&=\int_{\rrd} f_0 \log f_0 dv + \int_{\rrd} (-\log f^\star_0)f_0(v)dv\\ &=\int_{\rrd}f_0\log f_0 dv+\int_{\rrd} (-\log f_0^\star)\mu_0(dv)\end{split} \end{equation} and similarly\begin{equation}\begin{split}\label{eq: entropy of convolved term} H(g_\lambda\star \mu_0|\mu^\star_0)&=\int_{\rrd} (g_\lambda\star f_0)\log (g_\lambda \star f_0)dv+\int_{\rrd} (g_\lambda\star f_0)(-\log f_0^\star) dv\\&= \int_{\rrd} (g_\lambda\star f_0)\log (g_\lambda \star f_0)dv+\int_{\rrd} (-\log f_0^\star) (g_\lambda\star \mu_0)(dv).\end{split}\end{equation} Let us fix $\epsilon>0$. For the first term, we recall that the function $x\log x$ is convex on $[0,\infty)$, which implies that, for all $\lambda>0$, \begin{equation} \label{eq: convergence of entropy 1} \int_{\rrd} (g_\lambda\star f_0)\log (g_\lambda \star f_0)dv\le \int_{\rrd} f_0\log f_0 dv. \end{equation} For the second term, we recall that $-\log f_0^\star$ is continuous, and $-\log f_0^\star \le -\log c + z_3|v|^2$ for some $c>0$ and $z_3<\infty$. Using the fact $\langle |v|^2, g_\lambda\star \mu_0\rangle = \langle |v|^2, \mu_t\rangle + d\lambda \to \langle |v|^2, \mu_0\rangle <\infty$ and $g_\lambda\star \mu_0\to \mu_0$ weakly, one can check the uniform integrability \begin{equation}\limsup_M \limsup_{\lambda \to 0} \langle |v|^2 1(|v|\ge M), g_\lambda\star \mu_0\rangle =0\end{equation} whence there exists $M<\infty$ and $\lambda_0\in(0,1]$ such that, for all $\lambda<\lambda_0$, \begin{equation}\label{eq: control from outside box}\int_{|v|>M}(-\log f_0^\star)_+(g_\lambda\star \mu_0)(dv)\le  \int_{|v|>M}\left|\log c+z_3|v|^2\right|(g_\lambda \star \mu_0)(dv)<\frac{\epsilon}{4}\end{equation} where $_+$ denotes the positive part. Using the weak convergence $g_\lambda\star \mu_0 \to \mu_0$, there exists $\lambda_1<\lambda_0$ such that, for all $\lambda<\lambda_1$, \begin{equation}\label{eq: convergence of entropy 2}  \left|\int_{|v|\le M} (-\log f_0^\star)(g_\lambda\star \mu_0)(dv)-\int_{|v|\le M} (-\log f_0^\star)\mu_0(dv)\right|<\frac{\epsilon}{3}\end{equation} since the indicator $1_{|v|\le M}$ is discontinuous on a $\mu_0$-measure set, by absolute continuity. Finally, observe that the map \begin{equation} \mu\mapsto \langle 1[|v|>M, f_0^\star \ge 1](\log f_0^\star), \mu\rangle  \end{equation} is lower semicontinuous for the weak convergence, since the integrand is nonnegative, and is finite for $\mu=\mu_0^\star$ as noted above. Therefore, we can find $\lambda_2<\lambda_1$ such that, for all $\lambda<\lambda_2$, \begin{equation}\label{eq: convergence of entropy 3} \int_{|v|>M, f_0^\star\ge 1} (\log f_0^\star)(g_\lambda\star \mu_0)(dv)>\int_{|v|>M, f_0^\star\ge 1} (\log f_0^\star)\mu_0(dv)-\frac{\epsilon}{3}.\end{equation} For such $\lambda$, we split the second integral in (\ref{eq: entropy of convolved term}) into the regions $\{|v|\le M\}$, $\{|v|>M, f_0^\star\ge 1\}$ and $\{|v|>M, f_0^\star<1\}$ to obtain \begin{equation}\label{eq: entropy of convolved term}\begin{split} H(g_\lambda\star \mu_0|\mu_0^\star) & \le \int_{\rrd}f_0\log f_0 dv+ \int_{|v|\le M}(-\log f_0^\star)(g_\lambda\star \mu_0)(dv) \\&  \hs -\int_{|v|>M, f_0^\star \ge 1}(\log f_0^\star)(g_\lambda\star \mu_0)(dv)+\int_{|v|>M, f_0^\star<1} |\log f_0^\star|(g_\lambda\star \mu_0)(dv) \\[2ex]& \hspace{-1cm}\le \int_{\rrd} f_0 \log f_0 (dv)+\left(\int_{|v|<M} (-\log f_0^\star)\mu_0(dv)+\frac{\epsilon}{4}\right) \\ & \hs -\left(\int_{|v|>M, f_0^\star \ge 1}(\log f_0^\star)(g_\lambda\star \mu_0)(dv)-\frac{\epsilon}{4}\right) +\frac{\epsilon}{4} \\[2ex]& =\int_{\rrd} f_0\log f_0 dv-\int_{|v|\le M\text{ or } f_0^\star \ge 1} (\log f_0^\star)\mu_0(dv) +\epsilon  \\[2ex] & \le H(\mu_0|\mu_0^\star)+\epsilon \end{split} \end{equation} and we have proven (\ref{eq: entropy of convolved term conclusion}). \paragraph*{Step 2: Entropy of Remainder Term} We next turn to the convolution $g_\lambda\star \xi^{(n)}$. On the one hand, the density of $g_\lambda \star \xi^{(n)}$ is at most $g_\lambda(0)$; on the other hand, taking $R^2=2\langle |v|^2, \xi^{(n)}\rangle$, it follows by Chebychev that \begin{equation}\label{eq: entropy of remainder} \xi^{(n)}(|v|\le R)\ge 1-R^{-2}\langle |v|^2, \xi^{(n)}\rangle =\frac{1}{2}.\end{equation} For any fixed $u$, if $|v|\le R$ then $g_\lambda(u-v)\ge g_\lambda(0)\exp(-(|u|^2+R^2)/\lambda)$, and integrating over this region gives \begin{equation} \label{eq: lower bound on remainder density} (g_\lambda \star \xi^{(n)})(v)\ge \frac{1}{2}\exp\left(-(|u|^2+2\langle |v|^2, \xi^{(n)}\rangle)/\lambda\right)g_\lambda(0). \end{equation}  Together with Hypothesis \ref{hyp: condition on ref ms}iii), there exists a constant $a_\lambda$  such that \begin{equation} \label{eq: density of convolved error term} \left|\log \frac{d(g_\lambda \star \xi^{(n)})}{d\mu_0^\star} (u)\right|\le a_\lambda \left(1+|u|^2+\langle |v|^2, \xi^{(n)}_0\rangle\right).\end{equation}  We now integrate, and recall that the second moment of $\langle |v|^2, g_\lambda\star \xi^{(n)}\rangle = d\lambda + \langle |v|^2, \xi^{(n)}\rangle$, to find \begin{equation} \label{eq: convolved error term}H(g_\lambda\star \xi^{(n)}|\mu_0^\star)=\int_{\rrd} \log \frac{d(g_\lambda \star \xi^{(n)})}{d\mu_0^\star} (v)(g_\lambda\star \xi^{(n)})(dv) \le a_\lambda\left\langle 1+|v|^2, \xi^{(n)}\right\rangle\end{equation} potentially for a new choice of $a_\lambda$. \paragraph*{Step 3: Control of Overall Cost} We now combine (\ref{eq: uniform approximation of dynamic cost}, \ref{eq: initial cost under convolution  1}, \ref{eq: entropy of convolved term conclusion}, \ref{eq: convolved error term}) to see that, for some constant $C$ and $a_\lambda$, \begin{equation}\label{eq: total cost} \begin{split}\mathcal{I}(\mu^{(n,\lambda)}_\bullet, w^{(n,\lambda)})&\le \mathcal{I}(\mu_\bullet, w)+\left(\mathcal{J}(\mu^{(n,0)}_\bullet, w^{(n,0)})-\mathcal{J}(\mu_\bullet, w)\right) + C\vartheta(\lambda)\\& \hs \hs +\bigg(H(g_\lambda\star \mu_0|\mu_0^\star)-H(\mu_0|\mu_0^\star)\bigg)+ a_\lambda \langle 1+|v|^2, p_n\xi^{(n)}\rangle. \end{split}\end{equation} By the definitions of $p_n, \xi^{(n)}$, it follows that $p_n\xi^{(n)}= \nu^{(n)}/c_n$; by Lemma \ref{lemma: approximation 1}, $\langle 1+|v|^2, \nu^{(n)}\rangle \to 0$, $c_n\to 1$, so for fixed $\lambda>0$, the last term converges to $0$ as $n\to \infty$. We can therefore choose a sequence $\lambda_n\in (0,1], \lambda_n \to 0$ which decays slowly enough that $a_{\lambda_n}\langle 1+|v|^2, \nu^{(n)}\rangle \to 0$. We now define $\mu^{(n)}_\bullet:=\mu^{(n,\lambda_n)}_\bullet, w^{(n)}:=w^{(n,\lambda_n)}$. Every term except the first on the right-hand side of (\ref{eq: total cost}) converges to $0$, and in  particular $\limsup_n \mathcal{I}(\mu^{(n)}_\bullet, w^{(n)})\le \mathcal{I}(\mu_\bullet, w)$. \paragraph*{Step 4: Conclusion} We now check that the diagonal sequence extracted has all the desired properties. First, thanks to (\ref{eq: contraction on supremum}), the convolution with $g_{\lambda_n}$ preserves the boundedness, so \begin{equation} \sup_{(t,v,v_\star,\sigma)} K^{(n)}(t,v,v_\star,\sigma)B(v-v_\star)\le \sup_{(t,v,v_\star,\sigma)} K^{(n,0)}(t,v,v_\star,\sigma)B(v-v_\star)<\infty.\end{equation} To see convergence of the overall sequence, note that $W(\mu, g_\lambda\star \mu)\le C\sqrt{\lambda}$ for all measures $\mu$, and since $W$ is dominated by the total variation distance,\begin{equation} \begin{split}\sup_{t\le T} W(\mu^{(n)}_t,\mu_t)&\le \sup_{t\le T} W(g_{\lambda_n}\star \mu^{(n,0)}_t,\mu^{(n,0)}_t)+\sup_{t\le T} \left\|\mu^{(n,0)}_t-\mu_t\right\|_\mathrm{TV} \\[2ex]& \le C\sqrt{\lambda_n}+\sup_{t\le T} \left\|(1+|v|^2)(\mu^{(n,0)}_t-\mu_t)\right\|_\mathrm{TV} \to 0. \end{split}\end{equation} Similarly, $d(g_{\lambda_n}\star w^{(n,0)},w^{(n,0)})\le C\sqrt{\lambda_n}$, so that $d(w^{(n)},w)\to 0$.    \end{proof}   
 Finally, we prove Lemma \ref{lemma: approximation 2}, which allows us to impose an asymptotic lower bound on $K$, so we control how fast $|\log K|$ grows as $v,v_\star \to \infty$. 
 \begin{proof}[Proof of Lemma \ref{lemma: approximation 2}] Let us consider the space $\mathcal{MS}_2(\rrd)$ of signed measures with finite second moment $\langle 1+|v|^2, |\xi|\rangle<\infty$, equipped with the complete distance given by the weighted total variation norm $\|\xi\|_{\mathrm{TV}+2}:=\|(1+|v|^2)\xi\|_\mathrm{TV}$. We start from a measure-flux pair $(\mu_\bullet, w)$ as in the statement, so that the tilting function $K$ is continuous in $v,v_\star$, and $B(v-v_\star)K$ is bounded; since $B$ is bounded away from $0$, this implies the same for $K$. \paragraph*{Step 1: Construction of $K$} We begin with a family of mollifiers. Let us fix a smooth function $\eta: \rr \to [0,\infty)$, supported on $[-1,1]$ and such that $\int \eta ds=1$, and for $t\in [0,T], \lambda>0$, define \begin{equation} \eta_\lambda(s,t)=\frac{\eta((s-t)/\lambda)}{\int_0^T \eta((u-t)/\lambda)du}\end{equation} so that $\eta_\lambda$ is continuous in both arguments, $\eta_\lambda(\cdot, t)$ is supported on $[0,T]\cap [t-\lambda, t+\lambda]$, and $\int_0^T \eta_\lambda(s,t)ds=1$. For the spherical directions, let $h_\lambda(\cdot, \cdot)$ be the heat kernel on $\ssd$, so that $h_\lambda(\cdot, \sigma)$ is a smooth function on $\ssd$ which integrates to $1$, and $h_\lambda(\sigma', \sigma)d\sigma'\to \delta_{\sigma}(d\sigma')$ weakly as $\lambda\to 0$. With these fixed, we define $K^{(n)}$ by \begin{equation} K^{(n)}(t,v,v_\star,\sigma):=\int_{[0,T]\times \ssd} K(s,v,v_\star,\sigma')\eta_{1/n}(s,t)h_{1/n}(\sigma',\sigma)ds d\sigma'+\frac{1}{nB(v-v_\star)}. \end{equation} From the construction, the continuity of $K$ in $v,v_\star$ implies that each $K^{(n)}$ is continuous on $E$. $K^{(n)}$ also inherit the upper bound: there exists $M$ such that \begin{equation} \label{eq: growth of K 47}\sup_n \sup_{t,v,v_\star,\sigma} B(v-v_\star)K^{(n)}(t,v,v_\star,\sigma)\le M;\qquad \sup_{t,v,v_\star,\sigma} B(v-v_\star)K(t,v,v_\star,\sigma)\le M\end{equation} and by construction $\inf_E B(v-v_\star)K^{(n)}\ge \frac{1}{n}>0.$ Finally, for all $(v,v_\star)$ fixed, $dt d\sigma$ almost everywhere, $K^{(n)}(t,v,v_\star,\sigma)\to K(t,v,v_\star,\sigma)$. \paragraph*{Step 2: Construction of $\mu^{(n)}_\bullet$ by Picard-Lindel\"of} We now construct processes $\mu^{(n)}_\bullet$, which at this stage may be signed measures, via the machinery of the Picard-Lindel\"of theorem. We consider the space $\mathcal{MS}_2$ of signed measures for which the quadratic total variation norm $\|\xi\|_{\mathrm{TV}+2}:=\langle 1+|v|^2, |\xi|\rangle$ is finite. For any $t\in [0,T]$, $\xi\in \mathcal{MS}_2$, define the signed measures \begin{equation}\Phi(t,\xi):=\int_{\rrd\times\rrd\times\ssd} \Delta(v,v_\star,\sigma)B(v-v_\star)K(t,v,v_\star,\sigma)\xi(dv)\xi(dv_\star)d\sigma; \end{equation} \begin{equation}\Phi_n(t,\xi):=\int_{\rrd\times\rrd\times\ssd} \Delta(v,v_\star,\sigma)B(v-v_\star)K^{(n)}(t,v,v_\star,\sigma)\xi(dv)\xi(dv_\star)d\sigma. \end{equation} Using the uniform boundedness of $B(v-v_\star)K, B(v-v_\star)K^{(n)}$, it is easy to see that \begin{equation} \label{eq: bilinear TV estimate}\|(\Phi_n(t,\xi)-\Phi_n(t,\xi')\|_\mathrm{TV} \le C\|(\xi-\xi')\|_{\mathrm{TV}+2}(\|\xi\|_{\mathrm{TV}+2}+\|\xi'\|_{\mathrm{TV}+2}) \end{equation} for some constant $C$, uniformly in $n$, and similarly for $\Phi$. It then follows from the Picard-Lindel\"of theorem that, for any $\xi^{(n)}_0, \xi_0$, there exist unique local solutions to the integral equations \begin{equation}\label{eq: integral equations} \xi^{(n)}_t=\xi^{(n)}_0+\int_0^t \Phi_n(s,\xi^{(n)}_s)ds;\qquad \xi_t=\xi_0+\int_0^t \Phi(s,\xi_s)ds.\end{equation} Further, observing that $\|\Phi_n(t,\xi)\|_{\mathrm{TV}}\le C\|\xi\|_\mathrm{TV}$, it follows that $\|\xi^{(n)}_t\|_\mathrm{TV}$ grows at most exponentially in time; using the similar estimate that $\|\Phi_n(t,\xi)\|_{\mathrm{TV}+2}\le C \|\xi\|_{\mathrm{TV}+2}\|\xi\|_{\mathrm{TV}}$ by taking $\xi'=0$ above, the same holds for $\|\xi^{(n)}_t\|_{\mathrm{TV}+2}$, so the solutions are globally defined. \bigskip \\ Let us now consider these equations to construct our approximations. For the initial data, we recall that the finiteness of the entropy $H(\mu_0|\mu^\star_0)\le \mathcal{I}(\mu_\bullet, w)<\infty$ implies that $\mu_0$ has a density with respect to $\mu^\star_0$, and we set  \begin{equation}\label{eq: bounded density} \mu^{(n)}_0(dv)=c_n \left(\frac{d\mu_0}{d\mu^\star_0}\land n\right)\mu^\star_0(dv)\end{equation} for a normalising constant $c_n\to 1$ which makes $\mu^{(n)}_0$ a probability measure. We now take $\mu^{(n)}_\bullet$ to be the unique solution $\xi^{(n)}_t$ produced to $\partial_t\xi^{(n)}_t=\Phi_n(t, \xi^{(n)}_t)$ for this choice of initial data. It follows by definition of $K$ that the process $\mu_t$ given satisfies $\partial_t \mu_t=\Phi(t,\mu_t)$, which must then by the unique solution.  \paragraph*{Step 3: Positivity of $\mu^{(n)}_\bullet$}  To see that this gives positive measures, we use an integrating factor introduced by Norris \cite{norris1999smoluchowski} in the context of a similar construction for the Smolouchowski equation. We define \begin{equation} \theta^{(n)}_t(v):=\exp\left(\int_0^t\int_{\rrd\times\ssd} (K^{(n)}(s,v,v_\star,\sigma)+K^{(n)}(s,v_\star,v,\sigma))B(v-v_\star)\mu^{(n)}_s(dv_\star)d\sigma ds \right)\end{equation} and \begin{equation} \Phi^+_n(t,\xi)=\int_{\rrd\times\rrd\times\ssd} (\theta^{(n)}_t(v')\delta_{v'}+\theta^{(n)}_t(v'_\star)\delta_{v'_\star})B(v-v_\star)K^{(n)}(t,v,v_\star,\sigma)\xi(dv)\xi(dv_\star)d\sigma.\end{equation} Thanks to the boundedness, $\theta^{(n)}_t$ is bounded and bounded away from $0$, uniformly on compact time intervals, and $\Phi^+_n(\xi)\ge 0$ whenever $\xi\ge 0$. These integrating factors are such that $\partial_t (\theta^{(n)}_t \mu^{(n)}_t)=\Phi^+_n(t,\theta^{(n)}_t\mu^{(n)}_t)$, while applying the same arguments as above in the smaller space $(\cp_2, \|\cdot\|_{\mathrm{TV}+2})$ shows that the unique solution to $\partial_t \nu_t=\Phi_n^+(t, \nu_t)$ remains positive if $\nu_0\in \cp_2$ is a positive measure. It follows that $\theta^{(n)}_t\mu^{(n)}_t \ge 0$ are positive measures, and hence so are $\mu^{(n)}_t$; recalling again the boundedness, energy conservation implies that $\langle |v|^2, \mu^{(n)}_t\rangle=\langle |v|^2, \mu^{(n)}_0\rangle<\infty$ is constant for each $n$, and we conclude that $\mu^{(n)}_\bullet \in \D$.  We define the corresponding flux $w^{(n)}$ by \begin{equation} w^{(n)}(dt,dv,dv_\star,d\sigma):=K^{(n)}(t,v,v_\star,\sigma)\overline{m}_{\mu^{(n)}}(dt,dv,dv_\star,d\sigma)\end{equation} so that $\mu^{(n)}_\bullet, w^{(n)}$ is a measure-flux pair. Moreover, if $(\mu'_\bullet, w')$ is any measure-flux pair with $\mu'_0=\mu^{(n)}_0$ and with tilting function $K^{(n)}$, then $\partial_t\mu'_t=\Phi_n(t,\mu'_t)$, which implies that $\mu'_t=\mu^{(n)}_t$ by the uniqueness in step 2, and $w'=K\overline{m}_{\mu'}=K\overline{m}_{\mu^{(n)}}=w$, so each approximating pair $(\mu^{(n)}_\bullet, w^{(n)})$ is uniquely characterised by the initial value and tilting function, as claimed. \paragraph*{Step 4: Convergence of the Approximations} Let us now show that the measure-flux pairs constructed above converge as $n\to \infty$. We start from \begin{equation} \begin{split} \|\Phi_n(t,\mu^{(n)}_t)-\Phi(t,\mu_t)\|_{\mathrm{TV}+2}& \le \|\Phi_n(t,\mu^{(n)}_t)-\Phi_n(t,\mu_t)\|_{\mathrm{TV}+2}+ \|\Phi_n(t,\mu_t)-\Phi(t,\mu_t)\|_{\mathrm{TV}+2} \\ &\hspace{-3cm} \le C\|\mu^{(n)}_t-\mu_{t}\|_{\mathrm{TV}+2}(\|\mu^{(n)}_t\|_{\mathrm{TV}+2}+\|\mu_t\|_{\mathrm{TV}+2}) + \|\Phi_n(t,\mu_t)-\Phi(t,\mu_t)\|_{\mathrm{TV}+2}\end{split}\end{equation} using (\ref{eq: bilinear TV estimate}). In the first term, we observe that $\|\mu^{(n)}_t\|_{\mathrm{TV}+2}=\langle 1+|v|^2, \mu^{(n)}_0\rangle$ is bounded uniformly in $n, t$, thanks to energy conservation and the construction of $\mu^{(n)}_0$, and we absorb this constant factor into $C$. We can now use Gr\"onwall's Lemma to obtain \begin{equation}\label{eq: Gr 47}\sup_{t\le T} \|\mu^{(n)}_t-\mu_t\|_{\mathrm{TV}+2}\le e^{CT}\left(\|\mu^{(n)}_0-\mu_0\|_{\mathrm{TV}+2}+\int_0^T \|\Phi_n(t,\mu_t)-\Phi(t,\mu_t)\|_{\mathrm{TV}+2}dt\right).\end{equation}The first term is readily seen to converge to $0$ using the construction (\ref{eq: bounded density}) of $\mu^{(n)}_0$, recalling that $\mu_0$ has finite second moment. For the second term, we return to the definition of $\Phi, \Phi_n$ to see that \begin{equation} \begin{split} \label{eq: error from K Kn} \|\Phi_n(t,\mu_t)-\Phi(t,\mu_t)\|_{\mathrm{TV}+2}&\le 2\int_{\rrd\times\rrd\times\ssd}(1+|v|^2+|v_\star|^2)(K^{(n)}-K)(t,v,v_\star,\sigma)\\ & \hspace{4cm}\dots\times B(v-v_\star)\mu_t(dv)\mu_t(dv_\star)d\sigma \end{split}\end{equation} and integrating over $t\in [0,T]$ produces \begin{equation} \begin{split}& \int_0^T  \|\Phi_n(t,\mu_t)-\Phi(t,\mu_t)\|_{\mathrm{TV}+2} dt \\& \hs \hs\le 2\int_E  (1+|v|^2+|v_\star|^2)(K^{(n)}-K)B(v-v_\star)dt\mu_t(dv)\mu_t(dv_\star)d\sigma.\end{split} \end{equation} We now apply dominated convergence to see that the right-hand side converges to $0$, since $B(v-v_\star)(K^{(n)}-K)$ is bounded by (\ref{eq: growth of K 47}), and converges to $0$ for $dt \mu_t(dv)\mu_t(dv_\star)d\sigma$ almost all $(t,v,v_\star,\sigma)$, while $\mu_t$ has constant, finite second moment. Returning to (\ref{eq: Gr 47}), we conclude that $\|\mu^{(n)}_t-\mu_t\|_{\mathrm{TV}+2}\to 0$, which is stronger than the required convergence.  For the flux, we estimate \begin{equation}\label{eq: convergence of flux plus extra}\begin{split}  \|w^{(n)}-w\|_\mathrm{TV} &\le \int_E |K^{(n)}-K|B(v-v_\star)dt\mu^{(n)}_t(dv)\mu^{(n)}_t(dv_\star)d\sigma \\ & \hs  + \int_E K B(v-v_\star)dt|\mu_t(dv)\mu_t(dv_\star)-\mu^{(n)}_t(dv)\mu^{(n)}_t(dv_\star)|d\sigma. \end{split}\end{equation} The first term converges to $0$ as above, and recalling (\ref{eq: growth of K 47}), the second term is bounded by\begin{equation}\label{eq: convergence of flux plus extra 2}\begin{split}   \int_E K B(v-v_\star)dt|\mu_t(dv)\mu_t(dv_\star)-\mu^{(n)}_t(dv)\mu^{(n)}_t(dv_\star)|d\sigma \le 2C \int_0^T \|\mu^{(n)}_t-\mu_t\|_\mathrm{TV}dt \to 0 \end{split}\end{equation}  and we have proven that $\|w^{(n)}-w\|_{\mathrm{TV}}\to 0$ as desired. \paragraph*{Step 5: Convergence of the Cost Function} We finally check the convergence of the rate function $\mathcal{I}$ along our subsequence. First, $\mu^{(n)}_0=\mu_0$ by construction, so it suffices to prove the same thing for the dynamic cost $\mathcal{J}$; by the usual lower semicontinuity, it suffices to prove $\limsup_n \mathcal{J}(\mu^{(n)}_\bullet, w^{(n)})\le \mathcal{J}(\mu_\bullet, w)$. We start by writing \begin{equation} \begin{split} \mathcal{J}(\mu^{(n)}_\bullet, w^{(n)})&=\int_E \tau(K^{(n)})\overline{m}_{\mu^{(n)}}(dt,dv,dv_\star,d\sigma)\\[1ex]  &\le \mathcal{J}(\mu_\bullet, w)+\int_E \tau(K^{(n)})(\overline{m}_{\mu^{(n)}}-\overline{m}_\mu)(dt,dv,dv_\star,d\sigma) \\ &\hs\hs  +\int_E(\tau(K^{(n)})-\tau(K))\overline{m}_\mu(dt,dv,dv_\star,d\sigma).\end{split} \end{equation} In the second term, $B(v-v_\star)K^{(n)}\le M$ everywhere, and since $B\ge 1$, this implies the same bound for $K^{(n)}$ and hence the bound  $\tau(K^{(n)})\le 1+\tau(M)$, uniformly in $n$. The second term is now at most $(1+\tau(M))\|w^{(n)}-w\|_{\mathrm{TV}}\to 0$. Similarly, $\tau(K^{(n)})\to \tau(K)$ converges $dt \mu_t(dv)\mu_t(dv_\star)d\sigma$ almost everywhere, and hence $\overline{m}_\mu$ almost everywhere, with the same uniform bound as above. Since $\overline{m}_\mu(E)<\infty$, we can apply dominated convergence to see that the second term $\to 0$, and we are done. \end{proof}
 
 \section{Proof of Theorem \ref{thm: main}} \label{sec: pf of main}

 We now turn to the proof of the main result Theorem \ref{thm: main}. Let us fix, throughout, $\Theta$ and $P$ as in the theorem. We first present the proofs in detail in the case of the regularised hard sphere potential $B=1+|v|$: we will first carefully construct a change of measures $\PPm^N$ using the general form in Proposition \ref{prop: com}. We then prove a law of large numbers for the modified measures in Lemmas \ref{lemma: bad LLN main}, showing that any subsequential limit in distribution under the new measures almost surely lands in $\A_\Theta$; the proof is further broken down into Lemmas \ref{lemma: bad LLN 2} - \ref{lemma: bad LLN 4}, and we finally show how this implies the stated conclusion.  We will discuss at the end the necessary modifications for the Maxwell Molecule case $B=1$. \subsection{Construction of a change of measure $\mathbb{Q}^N$}\label{subsec: construction of com} Throughout, let us fix $(\Omega, \mathfrak{F}, (\mathfrak{F}_t)_{t\ge 0}, \mathbb{P})$ on which are defined regularised hard sphere Kac processes $\mu^N_\bullet$ and their empirical fluxes $w^N$. We now use the Girsanov formula recalled in Proposition \ref{prop: com}, and construct the tilting $\varphi$ of the initial data and dynamic modification of $K$ of the dynamics separately.  \paragraph*{Step 1. Construction of Initial Data} Let us consider the random variables $X_M, M\ge 0$, which describe the initial localisation of the energy in the initial data: \begin{equation} X_M=\left\langle |v|^21[|v|\ge M],\mu^N_0\right\rangle. \end{equation} Since the particles are sampled independently from $\mu_0^\star$, we can write $X_M$ as the mean of $N$ independent variables, which each have the distribution $ Y_M=|V|^21[|V|\ge M]; V\sim \mu_0^\star$. We write $\psi_M$ for the cumulant generating function for $Y_M$, and $\psi^\star_M$ for the associated Legendre transform: $$ \psi_M(\lambda)=\log \EE\left[e^{\lambda Y_M}\right]; \qquad \psi^\star_M(a)=\sup\left\{a\lambda-\psi_M(\lambda)\right\}.$$ By Hypothesis \ref{hyp: condition on ref ms}ii), it follows that $\psi_M(\lambda)=\infty$ for all $\lambda\ge z_2$ and all $M$, which implies that $\psi^\star_M(a)\le az_2$, uniformly in $M$.\bigskip \\  For $M>0$ and $\lambda\in [0,z_2)$ to be chosen later, we will take $\varphi_{N,M}$ to be the function \begin{equation} \varphi_{M,\lambda}(v)=\lambda |v|^21[|v|\ge M] - \psi_M(\lambda)\end{equation} so that, under any change of measure of the form (\ref{eq: COM0}) for this choice of $\varphi$, each initial velocity $V^i_0$ is distributed independently with law \begin{equation} \mu^\star_{0,\lambda,M}(dv)=\exp\left(\lambda|v|^21[|v|\ge M]-\psi_M(\lambda)\right)\mu_0^\star(dv). \end{equation} \paragraph*{Step 2: Choice of $\lambda$} We now choose $\lambda$ as a function of $M$. For fixed $M$, it is standard to check that $$E_M(\lambda)=\int_{\rrd} |v|^2\exp\left(\lambda |v|^21[|v|\ge M]-\psi_M(\lambda)\right)\mu_0^\star(dv)=\langle |v|^2, \mu^\star_{0,\lambda,M}\rangle $$ is continuous on $[0, z_2)$, $E_M(0)=1$, and $E_M$ diverges to infinity as $\lambda \uparrow z_2$ thanks to Hypothesis \ref{hyp: condition on ref ms}ii). In particular, we can choose $\lambda=\lambda(M)\in (0,z_2)$ such that $$ E_M(\lambda_M)=\langle |v|^2, \mu^\star_{0,\lambda_M,M}\rangle=\Theta(T)\in (1,\infty). $$With this choice of $\lambda$, we write $\varphi_M=\varphi_{M,\lambda_M}$ and $\mu^\star_{0,M}$ for $\mu^\star_{0,\lambda_M,M}$. \paragraph*{Step 3: Choice of $K$} We next choose the dynamic tilting function $K=K^{M,r,N}$, depending on the same parameter $M$ and an additional parameter $r$, to be chosen later. Given $r\in \mathbb{N}$, let $0=t^{(r)}_0\le t^{(r)}_1\le ...\le t^{(r)}_r<T$ be the partition given by \begin{equation}\label{eq: definition of time discretisation} \tri=\inf\left\{t\in [0,T]: \Theta(t)\ge \left(1-\frac{i}{r}\right)\Theta(0)+\frac{i}{r}\Theta(T)\right\} \in P\end{equation} and set $t^{(r)}_{r+1}=T$. By Hypothesis \ref{hyp: condition on ref ms}iii), $\mu^\star_{0,M}$ has a density, and in particular the function $r\mapsto \langle |v|^21[|v|\le r], \mu^\star_{0,M}\rangle$ is continuous. We can therefore choose $M_0\le M_1\le M_2\le \dots \le M_{r-1}\le M_r=\infty$ such that, for all $i=0,1,..r-1$,\begin{equation} \left \langle |v|^2 1[|v|\le M_i], \mu^\star_{0,M}\right \rangle =\Theta(t_i+) \end{equation} and observe that $M_0>M$. We now construct a tilting function $K=K^{M,r,N}$ by setting, for $\trim\le t<\tri$, \begin{equation} K^{M,r,N}(\mu^N_0,t,v,v_\star,\sigma)=\begin{cases} 0 &\text{if either }v,v_\star \in \text{Supp}(1(|v|\ge M_{i-1})\mu_0^N)=S_t; \\ N 1(N_t\ge 1)/N_t & \text{else} \end{cases} \end{equation} where $N_t=N_t(M,r)$ is the number of particles not in the special set, which is constant on $[\trim, \tri)$: \begin{equation} N_t=N-N\langle 1[|v|\ge M_{i-1}], \mu^N_0\rangle =N\langle 1[|v|<M_i], \mu^N_0\rangle=N-\#S_t.  \end{equation} Throughout, we will suppress the dependence of $K^{M,r,N}$ on the initial data $\mu^N_0$. In this way, particles with initial velocity $|v|\in [M_{i-1}, M_i)$ are `frozen' until time $\trim$. Moreover, since the special set $S_t$ is finite and is random only though the dependence on $\mu^N_0$, almost surely, no particles ever enter $S_t$, and so under the new measures, all particles whose initial velocity is $\le M_{i-1}$ interact as a Kac process on $N_t$ particles on $[\trim, \tri)$. We highlight that only particles whose \emph{initial} velocities are at least $M_{i-1}$ in magnitude are frozen, so that particles whose velocity increases to above $M_{i-1}$ by collision remain unfrozen, except on the probability $0$ event where the post collisional velocity coincides with that of a still-frozen particle. Let us also remark that $K^{M,r,N}$ satisfies the hypotheses of Proposition \ref{prop: com}, since $N_t$ depends only on $\mu^N_0$, with the uniform bound $K^{M,r,N}\le N$. \bigskip \\ With this choice of $K$ and $\varphi=\varphi_{M}$ as in steps 1-2, we now take $\mathbb{Q}^N_{M,r}$ to be the change of measure given by Proposition \ref{prop: com}. \paragraph*{Step 4: Choice of $M,r$} By the law of large numbers, as $N\rightarrow \infty$ with $M$ fixed,\begin{equation}\mathbb{Q}^N_{M,r}\left(W(\mu^N_0,\mu^\star_{0,M})>\delta\right)\to 0\end{equation} for any $\delta>0$ and, with $M,r$ fixed, for all $\delta>0$, \begin{equation} \mathbb{Q}^N_{M,r}\left(\text{For all }i=0,...,r, \h \left|\langle |v|^21[|v|\le M_{i}, \mu^N_0\rangle -\Theta(\tri+)\right|<\delta \right) \to 1;\end{equation} \begin{equation} \mathbb{Q}^N_{M,r}\left(\left|\frac{N_0(M,r)}{N}-\langle 1[|v|\le M_0], \mu^\star_{0,M}\rangle \right|<\delta \right)\to 1. \end{equation}  Now, we compare the equations \begin{align} e^{\psi_M(\lambda_M)}=\int_{\mathbb{R}^d}e^{\lambda_M|v|^21[|v|\ge M]}\mu^\star_0(dv);\\\Theta(T) e^{\psi_M(\lambda_M)}=\int_{\mathbb{R}^d}e^{\lambda_M|v|^21[|v|\ge M]}|v|^2\mu^\star_0(dv)\end{align} to obtain\begin{equation}\begin{split}  e^{\psi_M(\lambda_M)}& =\int_{|v|<M} \mu_0^\star(dv) + \int_{|v|\ge M} e^{\lambda_M |v|^2} \mu_0^\star(dv) \\[1ex]& \hs\hs \le 1 +\frac{1}{M^2}\int_{\rrd} |v|^2 e^{\lambda_M |v|^2 1[|v|\ge M]} \mu_0^\star(dv) \end{split}  \end{equation} which implies that  \begin{equation} e^{\psi_M(\lambda_M)}\le 1+\frac{\Theta(T)}{M^2}e^{\psi_M(\lambda_M)}\end{equation} and hence $\psi_M(\lambda_M)\rightarrow 0$ as $M\rightarrow \infty$. This implies the convergence of $\mu^\star_{0,M}$ to $\mu^\star_0$: for any $f$ with $|f|\le 1$, we estimate \begin{align*} \left|\langle f, \mu^\star_0-\mu^\star_{0,M}\rangle \right|&\le \int_{|v|<M}|f(v)|\left|1-e^{-\psi_M(\lambda_M)}\right|\mu^\star_0(dv) +  \frac{\langle |v|^2|f|, \mu^\star_0+\mu^\star_{0,M}\rangle}{M^2} \\[1ex] & \le \left|1-e^{-\psi_M(\lambda_M)}\right|+\frac{\Theta(T)+1}{M^2} \to 0\end{align*} and, since $f$ was arbitrary, the right-hand side is a bound for $\|\mu^\star_0-\mu^\star_{0,M}\|_\mathrm{TV}\ge W(\mu^\star_0,\mu^\star_{0,M})$. Similarly, we observe that \begin{equation} \langle 1[|v|\le M_0], \mu^\star_{0,M}\rangle \ge 1-\frac{\langle |v|^2, \mu_{0,M}^\star\rangle}{M_0^2} \ge 1-\frac{\Theta(T)}{M^2} \end{equation} using that $M_0\ge M$. Combining everything, and using a diagonal argument, we can construct a sequence $M_N\rightarrow \infty, r_N\to \infty$ slowly enough that, for all $\delta>0$, \begin{align} \label{eq: construction of QN1} &\mathbb{Q}^N_{M_N,r_N}\left(\max_{i\le {r_N}}\left|\langle |v|^21[|v|\le M_i], \mu^N_0\rangle -\Theta(\tri+)\right|<\delta \right)\rightarrow 1 \\[2ex]\label{eq: construction of QN2} &\mathbb{Q}^N_{M_N,r_N}\left(\inf_{t\in [0,T]} \frac{N_t(M_N, r_N)}{N}<1-\delta\right)\rightarrow 0\\[2ex]\label{eq: construction of QN3} &\mathbb{Q}^N_{M_N,r_N}(W(\mu^N_0,\mu^\star_{0, M_N})>\delta)\rightarrow 0. \end{align} We now take $\widetilde{\PPm}^N=\PPm^N_{M_N,r_N}$, and define $\PPm^N$ by conditioning: \begin{equation} \PPm^N(A):=c_N^{-1}\widetilde{\PPm}^N\left(A\cap\left\{\langle |v|^2, \mu^N_0\rangle \le 2\Theta(T), \frac{N_0}{N}\ge \frac{1}{2}\right\}\right)\end{equation} where $c_N$ is the $\widetilde{\PPm}^N$-probability of the event in the conditioning, which converges to $1$ by (\ref{eq: construction of QN1}- \ref{eq: construction of QN2}). We write throughout $K^N$ for $K^{N,M_N,r_N}$, and we remark that, since $\PPm^N$ is the conditioning of $\widetilde{\PPm}^N$ to events of high $\widetilde{\PPm}^N$-probability, the same convergences (\ref{eq: construction of QN1}- \ref{eq: construction of QN3}) hold with $\PPm^N$ in place of $\PPm^N_{M_N, r_N}=\widetilde{\PPm^N}$. By Proposition \ref{prop: com}, under these new measures, the particles are initially sampled from $\mu^\star_{0, M_N}$, conditional on the $\mathfrak{F}_0$- event $\frac{N_0}{N}\ge \frac{1}{2}$ and $\langle |v|^2, \mu^N_0\rangle \le 2\Theta(T)$, and the dynamics are then goverened by the inhomogeneous generator (\ref{eq: time dependent generator}). We begin with the following preparatory lemma. \begin{lemma}[Estimate on the Radon-Nidoykm Derivative]\label{lemma: estimate on change of measure} For the changes of measure $\PPm^N\ll \PP$ constructed above, and for all $\epsilon>0$, \begin{equation} \PPm^N\left(\frac{1}{N}\log \frac{d\PPm^N}{d\PP} > z_2 \Theta(T) +\epsilon \right)\to 0. \end{equation} \end{lemma} \begin{proof} By definition, the change of measure is \begin{equation} \label{eq: reiteration of cost} \begin{split} \frac{1}{N}\log\frac{d\PPm^N}{d\PP}&=\langle \varphi_{M_N},\mu^N_0\rangle +\langle \log K^{N}, w^N\rangle - \int_E (K^{N}-1)(t,v,v_\star,\sigma)\overline{m}_{\mu^N}(dt,dv,dv_\star,d\sigma) \\ &\hs -\frac{1}{N}\log c_N. \end{split}  \end{equation} The final term converges to $0$ because, as already noted above, $c_N\to 1$. For the first term, recall that $\phi_{M_N}=\phi_{M_N,\lambda_{M_N}}\le \lambda_{M_N}|v|^2$, and that $\lambda_{M_N}\le z_2$ is bounded, uniformly in $N$, so that (\ref{eq: construction of QN1}) gives \begin{equation}\label{eq: 1st term in COM}\PPm^N\left(\langle \varphi_{M_N}, \mu^N_0\rangle >z_2 \Theta(T)+\epsilon/3\right) \le \PPm^N\left(\langle |v|^2, \mu^N_0\rangle > \Theta(T)+\epsilon/3z_2\right)\to 0. \end{equation} In the second term, observe that $\log K^{N, M_N, r_N}\le \log N/N_0(M_N,r_N)$ on the support of $w^N$, $\PPm^N$-almost surely, since by definition of $w^N$, no points in the support of $w^N$ have either $v,v_\star$ belonging to the special set $S_t$. Thanks to the conditioning in the definition of $\PPm^N$, we have, $\PPm^N$-almost surely, \begin{equation} \label{eq: } \sup K^N\le 2, \hspace{1cm} \langle |v|^2, \mu^N_0\rangle \le 2\Theta(T)\end{equation} and the same arguments as in Section \ref{sec: ETUB} bound $Nw^N((0,T]\times\rrd\times\rrd\times\ssd)$ by a Poisson process of rate $NC$, for some constant $C$. All together, there exists a new constant $C$, depending only on $\Theta(T)$, on such that \begin{equation}\label{eq: upper bound on mass} \PPm^N\left(w^N(E)>C\right)\to 0. \end{equation} Using (\ref{eq: construction of QN3}) again, \begin{equation} \label{eq: upper bound on K} \PPm^N \left(\log N/N_0(M_N, r_N) > \epsilon/3C\right)\to 0 \end{equation} and, together with (\ref{eq: upper bound on mass}), \begin{equation} \label{eq: 2nd term in COM}\PPm^N\left(\langle \log K^N, w^N\rangle > \epsilon/3 \right) \to 0. \end{equation} For the final term, we will find an upper bound for $\int_E |K^{N,M_N, r_N}-1| d\overline{m}_{\mu^N}$. We split the integral into cases where neither $v,v_\star\in S_t$ and its complement. In the first case $(v,v_\star,\sigma)\in S_t^\mathrm{c}\times S_t^\mathrm{c}\times \ssd$, we have $$|1-K^N(t,v,v_\star,\sigma)|\le \frac{N}{N_t(M_N,r_N)}-1.$$ On the other hand, observing  that $S_t\subset\{|v|\ge M_N\}$, the contributions from $v\not \in S_t$ and $v_\star \not \in S_t$ can be controlled by straightforward Markov inequalities: for some constant $C$,\begin{equation} \int_{\rrd\times\rrd} (1+|v|+|v_\star)(1(v\in S_t)+1(v_\star\in S_t))\mu^N_t(dv)\mu^N_t(dv_\star) \le C M_N^{-1}\langle 1+|v|^2, \mu^N_0\rangle^2. \end{equation} Together we obtain\begin{equation} \begin{split} \label{eq: difference of K and 1}&\int_{\rrd\times\rrd\times\ssd} \left|1-K^N(t,v,v_\star,\sigma)\right|B(v-v_\star)\mu^N_t(dv)\mu^N_t(dv_\star)d\sigma \\& \hspace{6cm}\le C\left(\left(\frac{N}{N_t(M_N,r_N)}-1\right)+\frac{1}{M_N}\right) \langle 1+|v|^2, \mu^N_t\rangle^2. \end{split} \end{equation} We recall that $M_N\to \infty$, use  (\ref{eq: construction of QN2}) to bound the first factor and using the conditioning in the definition of $\PPm^N$ to bound the moment factor by $(1+2\Theta(T))^2$, $\PPm^N$-almost surely. We conclude that \begin{equation}\label{eq: thrd term in COM} \PPm^N\left(\int_E |1-K^N(t,v,v_\star,\sigma)|\overline{m}_{\mu^N}(dt,dv,dv_\star,d\sigma)>\epsilon/3\right) \to 0.\end{equation} Gathering (\ref{eq: 1st term in COM}, \ref{eq: 2nd term in COM}, \ref{eq: thrd term in COM}) and returning to (\ref{eq: reiteration of cost}), the lemma is proven.  \end{proof} \subsection{Law of Large Numbers} We next prove the following law of large numbers under the new measures for the sets $\A_\Theta$ given in the theorem. \begin{lemma}\label{lemma: bad LLN main} Let $\PPm^N$ be the probability measures constructed above, and suppose $S\subset N$ is an infinite subsequence such that, the laws $\PPm^N \circ (\mu^N_\bullet, w^N)^{-1} $ converges weakly on $\D\times\Mm$, and let $(\mu_\bullet, w)$ be a random variable, defined with respect to a new probability space $(\Omega, \mathfrak{F},\PPm)$ and whose distribution is the limit. Then, for $\A_\Theta$ as in Theorem \ref{thm: main} for some $\alpha>0$ to be chosen, \begin{equation} \PPm\left((\mu_\bullet, w)\in \A_\Theta\right) =1.\end{equation} \end{lemma} It will be convenient, throughout, to realise all $(\mu^N_\bullet, w^N), N\in S$ and $(\mu_\bullet,w)$ on a common probability space with probability measure $\PPm$, such that the law of $(\mu^N_\bullet, w^N)$ under $\PPm$ is the same as under $\PPm^N$, and such that $\mu^N_\bullet \rightarrow \mu_\bullet$ and $w^N \to w$ almost surely, and such that (\ref{eq: construction of QN1} - \ref{eq: construction of QN3}) hold with almost sure convergence as $N\to \infty$ through $S$:  \begin{align} \label{eq: construction of QN1'} &\PPm\left(\max_{i\le {r_N}}\left|\langle |v|^21[|v|\le M_i], \mu^N_0\rangle -\Theta(\tri+)\right|\to 0 \right)= 1; \\[2ex]\label{eq: construction of QN2'} &\PPm\left(\inf_{t\in [0,T]} \frac{N_t(M_N, r_N)}{N} \to 1\right)= 1\\[2ex]\label{eq: construction of QN3'} &\PPm(W(\mu^N_0,\mu^\star_{0, M_N})\to 0)=1. \end{align} We will write $\EEm$ for the expectation under this probability measure. For clarity, we will subdivide the argument into three smaller steps. We also observe that each $\mu^N_\bullet$ has jumps of size at most $4/N$, and so the limit $\mu_\bullet$ is $\PPm$-almost surely continuous and the Skorokhod convergence can be upgraded to uniform convergence by Proposition \ref{prop: Sk continuous functions}c), \begin{equation}\label{eq: uniform convergence} \PPm\left(\sup_{t\le T} W(\mu^N_t, \mu_t)\to 0\text{ as }N\to \infty \text{ through }S\right)=1. \end{equation} We also note immediately from (\ref{eq: construction of QN3}) that $W(\mu^N_0, \mu^\star_0) \le W(\mu^N_0, \mu^\star_{0,M_N})+W(\mu^\star_{0,M_N}, \mu^\star_0)\to 0$, $\PPm$-almost surely, so that $\mu_0=\mu_0^\star$ almost surely. We now prove the remaining properties defining $\mathcal{A}_\Theta$ one by one. We next prove that the limit process $(\mu_\bullet, w)$ is almost surely a measure-flux pair. \begin{lemma}[Limiting Path as a Measure-Flux Pair]\label{lemma: bad LLN 2} Continue in the notation following Lemma \ref{lemma: bad LLN main}. Then \begin{equation}\label{eq: limiting object is ms flux pair} \PPm\left((\mu_\bullet, w)\text{ is a measure-flux pair,} \h w=\overline{m}_{\mu}\right)=1.\end{equation} In particular \begin{equation}\label{eq: AS BE} \PPm\left(\mu_\bullet \text{ is a solution to  (\ref{eq: BE}) with }\mu_0=\mu_0^\star\right)=1.\end{equation} \end{lemma} \begin{proof} This lemma is similar to Step 1 in the proof of the lower bound in Section \ref{sec: RLB}. As in the cited proof, the continuity equation (\ref{eq: CE}) holds for the finite paths $(\mu^N_\bullet, w^N)$ $\PPm$-almost surely, and since the set of pairs $(\mu_\bullet, w)$ satisfying the continuity equation is closed by Lemma \ref{lemma: semiconntinuity}, it follows that $(\mu_\bullet, w)$ solves (\ref{eq: CE}) almost surely. \bigskip \\ We next show that $w=\overline{m}_\mu$, almost surely. Let us fix $g:E\to\rr$ continuous and compactly supported, and start by observing that the process \begin{equation} M^{N,g}_t=\int_E 1(s\le t) g(s,v,v_\star,\sigma)(w^N-K^N\overline{m}_{\mu^N})(ds,dv,dv_\star,d\sigma) \end{equation} is a $\PPm$-martingale, with previsible quadratic variation at most \begin{equation}\begin{split} [M^{N,g}]_t & = N^{-1} \int_E 1(s\le t) g^2 K^N(s,v,v_\star,\sigma)B(v-v_\star)\mu^N_s(dv)\mu^N_s(dv_\star)dsd\sigma \\[2ex] & \le 2N^{-1} T\|g\|_\infty^2 (3+2\langle |v|^2, \mu^N_0\rangle)\\ & \le 2N^{-1} T\|g\|_\infty^2 (3+4\Theta(T)). \end{split} \end{equation} We now observe that $M^{N,g}_0=0$ and $[M^{N,g}]_T\to 0$ $\PPm$-almost surely, which implies by standard martingale estimates that \begin{equation}\label{eq: first term in bad LLN} \sup_{t\le T} M^{N,g}_t\to 0\hs \text{in }\PPm\text{-probability}.\end{equation} We now investigate the difference between these martingales and the equivalent processes with $K\equiv 1$:  \begin{equation} \begin{split} &M^{N,g}_T-\int_E g(s,v,v_\star,\sigma)(w^N-\overline{m}_{\mu^N})(ds,dv,dv_\star,d\sigma)\\& \hspace{2cm} =\int_E g(s,v,v_\star,\sigma)(1-K^N)\overline{m}_{\mu^N}(ds,dv,dv_\star,d\sigma)\end{split} \end{equation} Returning to (\ref{eq: difference of K and 1}) and integrating over time, \begin{equation}\begin{split}\label{eq: second term in bad LLN}&\left|\int_E  g(w^N-\overline{m}_{\mu^N})(ds,dv,dv_\star,d\sigma)-M^{N,g}_t \right| \\& \hs\hs \hs \le C_gT\left(\left(\frac{N}{N_0(M_N,r_N)}-1\right)+\frac{1}{M_N}\right) \langle 1+|v|^2, \mu^N_0\rangle^2  \end{split}  \end{equation}  and, by the choice of $r_N, M_N$, the right-hand side converges to $0$, almost surely. Finally, using Lemma \ref{lemma: really useful continuity result}, $\PPm$-almost surely, \begin{equation} \label{eq: third tirm in bad LLN}\left|\int_E g(t,v,v_\star,\sigma)(\overline{m}_{\mu^N}-\overline{m}_\mu)(ds,dv,dv_\star,d\sigma)\right|\to 0. \end{equation} We now gather (\ref{eq: first term in bad LLN}, \ref{eq: second term in bad LLN}, \ref{eq: third tirm in bad LLN}) to conclude that, $\PPm$-almost surely, $\langle g, w-\overline{m}_\mu\rangle=0$. This extends to all $g\in C_c(E)$ simultaneously by taking a union bound over a countable dense subset of $C_c(E)$ to conclude that $\PPm(w=\overline{m}_\mu)=1$, and the lemma is proven. \end{proof} Next, we prove that the second moment $\langle |v|^2, \mu_t\rangle $ coincides everywhere with the function $\Theta(t)$ given.  \begin{lemma}[Second Moment of Limiting Path]\label{lemma: bad LLN 3} We continue in the notation following Lemma \ref{lemma: bad LLN main}. Then \begin{equation} \label{eq: conclusion of 2d} \PPm\left(\langle |v|^2, \mu_t\rangle =\Theta(t) \text{ for all }t\le T\right)=1.\end{equation}\end{lemma} \begin{proof} We start by decomposing \begin{equation} \label{eq: frozen plus time dependent Kac} \mu^N_t= \xi^N_t + \frac{N_t}{N}\nu^N_t\end{equation} where $\xi^N_t$ is the empirical measure of frozen particles $\xi^N_t= 1_{S_t}\mu^N_0$, which is constant on each time interval $[t^{(r_N)}_{i-1}, t^{(r_N)}_i)$, and on each such time interval, $\nu^N_t$ is a Kac process on $N_t$ particles. Moreover, thanks to the conditioning in the definition of $\PPm^N$, we have the almost sure energy bound: \begin{equation} \langle |v|^2, \nu^N_{\trim}\rangle \le \frac{N}{N_0}\langle |v|^2,\mu^N_0\rangle \le 4\Theta(T).\end{equation}   We now fix an interval $I\subset [0,T]\setminus P$ and $\ell>0$ such that \begin{equation} \inf_{t\in I, s\in P, s<t} (t-s)\ge \ell >0. \end{equation}  For each $N$, the points $t^{(r_N)}_i, i\le r_N$ all belong to $P$, and so do not lie in $I$; we may therefore apply the moment creation property in Proposition \ref{prop: MCP}i) to obtain, for all $N$ large enough,  \begin{equation} \label{eq: moment bound on finite path} \EEm\left[ \sup_{t\in I}\langle |v|^4, \nu^N_t\rangle\right] \le C\ell^{-2}\end{equation} uniformly in $N$, for some $C$ depending on $\Theta(T)$. We next observe that $\|\xi^N_t\|_{\mathrm{TV}}\le 1-N_0/N \to 0$, uniformly in time $\PPm$-almost surely, and so it follows from (\ref{eq: uniform convergence}) that \begin{equation}\label{eq: uniformity 2}\PPm\left(\sup_{t\le T} W(\nu^N_t, \mu_t)\to 0 \text{ as $N\to \infty$ through $S$}\right)=1.\end{equation} Using Fatou's lemma and the lower semicontinuity of moments, we may now send $N\rightarrow \infty$ through $S$ in (\ref{eq: moment bound on finite path}) to obtain, for $I$ as before, the same estimate on $\sup_I \langle |v|^4, \mu_t\rangle$, and together \begin{equation} \EEm\left[\sup_{t\in I}\left(\langle |v|^4, \nu^N_t\rangle + \langle |v|^4, \mu_t\rangle\right) \right] \le C\ell^{-2}\end{equation} so we can find a large $M$, depending on $\ell$, such that, for all $N$, \begin{equation} \label{eq: localise second moment} \PPm\left(\sup_{t\in I}\langle |v|^2 1[|v|\ge M], \nu^N_t+\mu_t\rangle \ge \epsilon/3 \right)<\epsilon'/3. \end{equation} Now, let $f_M$ be a continuous, compactly supported function with $0\le f_M\le |v|^2$ and $f_M=|v|^2$ when $|v|\le M$. By the uniform convergence (\ref{eq: uniform convergence}), for $N\in S$ large enough,\begin{equation}\label{eq: convergence of truncated 2nd moment} \PPm\left(\sup_{t\in I}|\langle f_M, \nu^N_t-\mu_t\rangle |\ge \epsilon/4\right)<\epsilon'/3\end{equation} and thanks to (\ref{eq: localise second moment}), \begin{equation} \PPm\left(\sup_{t\in I} \langle |v|^2-f_M, \nu^N_t+ \mu_t \rangle \ge \epsilon/3 \right) <\epsilon'/3\end{equation} and together, for $N\in S$ large enough, \begin{equation}\label{eq: convergence of second moments 1} \PPm\left(\sup_{t\in I} \left|\langle |v|^2, \nu^N_t \rangle - \langle |v|^2, \mu_t\rangle\right|>2\epsilon/3 \right) \le 2\epsilon'/3. \end{equation}  For each $N$, the interval $I$ lies in some $[\trim, \tri)$, $r=r_N, i=i_N$, as the endpoints of all such intervals always belong to $P$, and in particular, $\nu^N_t$ is a conservative Kac process on this interval, with \begin{equation} \langle |v|^2, \nu^N_{\trim}\rangle = \frac{N}{N_{\trim}} \langle |v|^2 1[|v|\le M_{i(N)-1}, \mu^N_0\rangle  \end{equation} Thanks to (\ref{eq: construction of QN2}), the first factor converges to $1$, $\PPm$-almost surely, and using (\ref{eq: construction of QN3}), for all $\epsilon>0$, we obtain\begin{equation}\label{eq: convergence of second moment of unfrozen} \PPm\left(\sup_{t\in I} \left|\langle |v|^2, \nu^N_t\rangle - \Theta(t^{(r_N)}_{i(N)-1}+)\right|>\epsilon\right)\to 0.\end{equation} Since $I$ is an interval disjoint from $P$, $\Theta$ is constant on $I$, and by the construction of the points $\tri$, we have the nonrandom bound \begin{equation} \sup_{t\in I}\left|\Theta(t)-\Theta(t^{(r_N)}_{i(N)-1}+)\right|\le \frac{1}{r_N} \to 0\end{equation} and we conclude that, for $N\in S$ large enough, \begin{equation} \label{eq: convergence of second moments 2} \PPm\left( \sup_{t\in I}\left|\langle |v|^2, \nu^N_t\rangle -\Theta(t)\right| \ge \epsilon/3\right) <\epsilon'/3.\end{equation} Combining (\ref{eq: convergence of second moments 1}, \ref{eq: convergence of second moments 2}), we have shown that, for all $\epsilon, \epsilon'>0$, \begin{equation} \PPm\left(\sup_{t\in I} \left|\langle |v|^2, \mu_t\rangle - \Theta(t)\right|>\epsilon\right)\le \epsilon' \end{equation} so that $\langle |v|^2, \mu_t\rangle = \Theta(t)$ for all $t\in I$, $\PPm$-almost surely. We can now cover $[0,T]\setminus P$ by a countable collection of intervals of this form, so that this conclusion holds for all $t\not \in P$ almost surely. \bigskip \\ We now show that, on a single almost sure event, this also holds for $t\in P$. As remarked in Lemma \ref{lemma: bad LLN 2}, there is a $\PPm$-almost sure event on which $\mu_\bullet$ is a solution to (\ref{eq: BE}) with $\mu_0=\mu^\star_0$, and in particular $\mu_\bullet$ is continuous and the energy $\langle |v|^2, \mu_t\rangle$ is nondecreasing by Proposition \ref{prop: MCP}iii). On this event, the equality $\langle |v|^2, \mu_0\rangle=\langle |v|^2, \mu^\star_0\rangle=1=\Theta(0)$ certainly holds at time $t=0$, and on the intersection of this event and the event where the second moment equality holds for $t\not \in P$, then any $t\in P\setminus \{0\}$ can be approached from below by $s\in [0,T]\setminus P$. By left-continuity of $\Theta$, \begin{equation} \Theta(t)=\limsup_{s\uparrow t, s\not \in P} \Theta(s) = \limsup_{s\uparrow t, s\not \in P} \h \langle |v|^2, \mu_s\rangle  \le \langle |v|^2, \mu_t\rangle. \end{equation} For the other inequality, on the same almost sure event as above, fix $t\in P$ and $\epsilon>0$. By monotone convergence, we can find a continuous, compactly supported function $0\le f\le |v|^2$ and $\langle |v|^2, \mu_t\rangle <\langle f, \mu_t\rangle +\epsilon$. Using continuity in $W$, \begin{equation} \langle |v|^2, \mu_t\rangle < \langle f, \mu_t\rangle +\epsilon = \limsup_{s\uparrow t, s\not\in P} \h \langle f, \mu_s\rangle + \epsilon \le \limsup_{s\uparrow t} \Theta(s) + \epsilon = \Theta(t)+\epsilon \end{equation} and, since $\epsilon>0$ was arbitrary, we have equality at $t$. We emphasise again that the almost sure event used here does not depend on $t\in P$, and so the equality holds for all $t$ simultaneously with $\PPm$-probability $1$, as desired. \end{proof} Finally we check the fourth moment conditions. \begin{lemma}[Fourth Moment of Limiting Path]\label{lemma: bad LLN 4} Continue in the notation above. For $A(t)$ as in (\ref{eq: At}), for some $\alpha>0$ to be chosen, we have \begin{equation} \PPm\left(\langle |v|^4, \mu_t\rangle \le A(t)\text{ for all }t\in [0,T]\right) =1. \end{equation} \end{lemma}\begin{proof} Since $A=\infty$ on $P$, there is nothing to prove for such times. Let us fix $I=[u,v]\subset (0,T]$ disjoint from $P$, and let $u'=\max(s: s\in P, s<u)$, which always exists, belongs to $P$ and is strictly less than $u$, because $P$ is closed and $0\in P, 0<u$.  For any $J_a=[a,v]\supset I$ with $u'<a\le u$, we apply Lemma \ref{lemma: bad LLN 2} to see that,$\PPm$-almost surely, $(\mu_t)_{t\in J_a}$ is a solution to (\ref{eq: BE}), with energy given by $\langle |v|^2, \mu_t\rangle = \Theta(t)=\Theta(a)\le \Theta(T)$, because $J_a$ is disjoint from $P$. Proposition \ref{prop: MCP}ii) now applies pathwise, and for some absolute constant $C$, \begin{equation} \PPm\left(\langle |v|^4, \mu_t\rangle \le C\Theta(T)(t-a)^{-2} \text{ for all }t\in I\right)=1. \end{equation} We now take $a \downarrow u'$ to obtain \begin{equation} \PPm\left(\langle |v|^4, \mu_t\rangle \le C\Theta(T)(t-u')^{-2} \text{ for all }t\in I\right)=1. \end{equation} Choosing $\alpha=C\Theta(T)$, the bound is exactly $A(t)$, because $t-u'=\min(t-s: s\in P, s<t)$ for all $t\in I$. We now cover $[0,T]\setminus P$ with countably many such $I$, and the claim is proven. \end{proof}   Together, Lemmas \ref{lemma: bad LLN 2}, \ref{lemma: bad LLN 3}, \ref{lemma: bad LLN 4} prove the conclusions of Lemma \ref{lemma: bad LLN main}\subsection{Proof of Theorem}  We now give the proof in the case of the regularised hard spheres kernel. \begin{proof}[Proof of Theorem \ref{thm: main}a]  We first check that $\A_\Theta\subset\D\times \Mm$ are compact. This follows almost exactly the same argument as Lemma \ref{lemma: bad LLN main} above: fix $(\mu^{(n)}_\bullet, w^{(n)})\in \A_\Theta$. Since the spaces $\{\mu\in \cp_2: \langle |v|^2, \mu\rangle \le \Theta(T)\}$ are compact for $W$, and using the Boltzmann equation (\ref{eq: BE}) and the second moment bound to check equicontinuity, we can pass to a subsequence converging to a limit $(\mu_\bullet, w)$. First, since continuous functions are closed for Skorokhod convergence, $\mu_\bullet$ must also be continuous, so one can upgrade to uniform convergence $\sup_t W(\mu^{(n)}_t, \mu_t)\to 0$ by Proposition \ref{prop: Sk continuous functions}c). Immediately, $\mu_0=\mu_0^\star$, and the lower semicontinuity of moments gives $\langle |v|^4, \mu_t\rangle\le A(t), \sup_t \langle |v|^2, \mu_t\rangle \le \Theta(T)$. Using the same argument as (\ref{eq: third tirm in bad LLN}) and the second moment bound, $\overline{m}_{\mu^{(n)}}=w^{(n)}\to \overline{m}_{\mu}$ so that $w=\overline{m}_{\mu}$, and the same argument as before allows us to take the limit of the continuity equation to conclude that $\mu_\bullet$ solves (\ref{eq: BE}). Finally, repeating the arguments of Lemma \ref{lemma: bad LLN 3}, $\langle |v|^2, \mu_t\rangle$ can be found as the limit of $\langle |v|^2, \mu^{(n)}_t\rangle = \Theta(t)$ away from $P$ to obtain $\langle |v|^2, \mu_t\rangle, t\not \in P$. Since $\mu_\bullet$ solve (\ref{eq: BE}), $\langle |v|^2, \mu_t\rangle$ is nondecreasing, and we may take left-limits to extend the equality to $t\in P$.\bigskip \\ For the rate function, we return to the definition (\ref{eq: leo rate function}): all $\mu_\bullet \in \A_\Theta$ start at $\mu_0=\mu^\star_0$, we have $H(\mu_0|\mu^\star_0)=0$, and the unique choice $K=1$ gives $\tau(K)=0$, so $\mathcal{J}(\mu_\bullet,w)=0$ and $\mathcal{I}(\mu_bullet, w)=0$ as desired.  \bigskip \\ We now prove (\ref{eq: first item of main},\ref{eq: second item of main}). For the first item, let $\U\supset \mathcal{A}_\Theta$ be any open set, and $S'\subset \mathbb{N}$ a subsequence such that \begin{equation} \label{eq: defin of S'} \lim_{N\to \infty, N\in S'} \frac{1}{N}\log \PP\left((\mu^N_\bullet,w^N) \in \U\right)=\liminf_{N\in \mathbb{N}} \frac{1}{N}\log \PP\left((\mu^N_\bullet,w^N) \in \U\right).\end{equation} For the changes of measure $\PPm^N$ constructed above, we recall Lemma \ref{lemma: estimate on change of measure} to see that Corollary \ref{cor: tightness} applies, so that the laws $\PPm^N\circ (\mu^N_\bullet, w^N)^{-1}$ are tight. We can therefore pass to a further subsequence $S\subset S'$ such that the laws $\PPm^N\circ (\mu^N_\bullet, w^N)^{-1}$ converge to the law of a new random variable $(\mu_\bullet, w)$ under a new probability measure $\PPm$. This is exactly the setting of Lemma \ref{lemma: bad LLN main}, from which $\PPm((\mu_\bullet, w)\in \A_\Theta)=1$, which certainly implies that $\A_\Theta$ is nonempty. We then have $$ \liminf_{N\in S} \PPm^N\left((\mu^N_\bullet,w^N)\in \U\right)\ge \PPm((\mu_\bullet, w)\in \U)\ge \PPm((\mu_\bullet, w)\in \A_\Theta) =1 $$ since $\U\supset \mathcal{A}_\Theta$. Fixing $\epsilon>0$ and recalling Lemma \ref{lemma: estimate on change of measure}, we see that, for $N\in S$ large enough, $$\PPm^N\left((\mu^N_\bullet, w^N) \in \U, \frac{d\PPm^N}{d\PP} \le e^{N(\Theta(T)z_2+\epsilon)}\right)>\frac{1}{2}.$$ It follows that, for $N\in S$ large enough, \begin{equation} \begin{split} \PP\left((\mu^N_\bullet, w^N) \in \U\right) & = \EE_{\PPm^N}\left[\left(\frac{d\PPm^N}{d\PP}\right)^{-1}1(\mu^N_\bullet, w^N)\in \U]\right]  \\ & \ge e^{-N(\Theta(T)z_2+\epsilon)}\PPm^N\left(\mu^N_\bullet \in \U, \frac{d\PPm^N}{d\PP} \le e^{N(z_2\Theta(T)+\epsilon)}\right) \\ &\ge \frac{1}{2}e^{-N(\Theta(T)z_2+\epsilon)}. \end{split} \end{equation} Taking the logarithm and the limit $N\to \infty$ through $S\subset S'$ and then the limit $\epsilon\downarrow 0$, we conclude that $$ \lim_{N\to \infty, N\in S'}\frac{1}{N}\log \PP\left((\mu^N_\bullet, w^N) \in \U\right) \ge -\Theta(T)z_2 $$ and by the choice (\ref{eq: defin of S'}), we have proven the same bound for the limit inferior over the full sequence $N\in \mathbb{N}$. The lower bound is independent of $\U\supset \A_\Theta$, and so we have proven the claim (\ref{eq: first item of main}). \bigskip \\ For the second item (\ref{eq: second item of main}), we observe that $\Theta$ is locally constant at $T$, so we can find an interval $I\ni T $, with $\inf(t-s: t\in I, s\in P)>0$ and such that $\Theta(t)=\Theta(T)>1$ for all $t\in I$. Thanks to the fourth moment bound in the construction of $\A_\Theta$, we can choose $R<\infty$ and a continuous, compactly supported function $0\le f_R(v)\le |v|^2$ such that, for all $\mu_\bullet \in \A_\Theta$, \begin{equation} \inf_{t\in I} \langle f_R, \mu_t\rangle >\frac{1+\Theta(T)}{2}.\end{equation} Now, writing $|I|$ for the Lebesgue measure of $I$, we choose $\V$ to be the set \begin{equation} \V=\left\{(\mu_\bullet, w)\in \D\times \Mm: \int_I \langle f_R, \mu_t\rangle dt> \frac{1+\Theta(T)}{2} |I|\right\}. \end{equation} $\V$ is open in $\D\times\Mm$ by Lemma \ref{lemma: really useful continuity result}, and $\A_\Theta\subset \V$ by construction. However, for all $N$, we have the bound $\langle f_R, \mu^N_t\rangle \le \langle |v|^2, \mu^N_0\rangle$ for all $t$, because the kinetic energy is constant in time, so \begin{equation} \mathbb{P}\left((\mu^N_\bullet, w^N) \in \V\right)\le \mathbb{P}\left(\langle |v|^2,\mu^N_0\rangle>\frac{1+\Theta(T)}{2}\right).\end{equation} We now apply Cranm\'er's theorem. Recalling the notation $\psi_0, \psi_0^\star$ defined in Subsection \ref{subsec: construction of com}, we recall that $\psi_0^\star(a)>0$ for all $a\neq \langle |v|^2, \mu_0^\star\rangle =1$, and \begin{equation}\begin{split}\label{eq: crame'r} \liminf_N \left[\frac{1}{N}\log \mathbb{P}\left((\mu^N_\bullet, w^N) \in \V\right)\right] &\le \liminf_N \left[\frac{1}{N}\log \mathbb{P}\left(\langle |v|^2,\mu^N_0 \rangle >\frac{1+\Theta(T)}{2}\right)\right] \\[1ex]& = -\psi_0^\star\left(\frac{1+\Theta(T)}{2}\right)<0.\end{split}\end{equation} \end{proof}

 \subsection{Maxwell Molecules Case} We now give the proof in the case of Maxwell molecules. In this case, since the kernel $B$ is bounded, the moment creation property no longer holds; we also change measure so that, under $\PPm^N$, the Kac process has a kernel $\widetilde{B}=\widetilde{B}_\delta$ with linear growth. The previous argument then applies, albeit with an additional (small) exponential cost. Since the argument is almost identical, we will discuss only the essential modifications relative to the regularised hard spheres case. As before, let us fix $(\Omega, \mathfrak{F}, (\mathfrak{F}_t)_{t\ge 0}, \mathbb{P})$ on which are defined Maxwell molecule Kac processes $\mu^N_\bullet$ and their empirical fluxes $w^N$. \begin{proof}[Proof of Theorem \ref{thm: main}b)] Fix $\Theta, P, \delta>0$ as in the statement. We construct the modification of the initial data via $\varphi_M$ exactly as for the case of hard spheres above. With the same notation on $M_i, \tri$ and the special set of `frozen' particles $S_t$, we now choose $K$ to be given by \begin{equation} K^{N,M,r}(t,v,v_\star,\sigma)=\begin{cases} 0 & \text{if either }v,v_\star\in S_t; \\ N(1+\delta|v-v_\star|)1(N_t\ge 1)/N_t & \text{else} \end{cases} \end{equation} where, again, we supress the argument $\mu^N_0$. In this way, the non-frozen particles interact as a Kac process with kernel $(1+\delta|v|)$ on $N_t$ particles on each time interval $[\trim, \tri)$. We choose $M_N, r_N \to \infty$ in exactly the same way as before, and write $\PPm^N$ for the resulting changes of measure via Proposition \ref{prop: com}:  \begin{equation} \begin{split} \label{eq: COM MM} \frac{d\PPm^N}{d\PP}&=\exp\bigg(N\langle \varphi_{M_N}, \mu^N_0\rangle + \langle \log K^N, w^N\rangle \\ &\hspace{2cm} -N\int_0^T \int_{\rrd\times\rrd\times \ssd}(K^N-1)\overline{m}_{\mu^N}(dt,dv,dv_\star,d\sigma)\bigg).\end{split} \end{equation} We will write $K=(1+\delta|v-v_\star|)$ for the limiting tilting function. The strategy is now similar to the previous case. The law of large numbers follows in the same way for the new definition of $\A_\Theta$ without essential modification, allowing $\alpha$ to depend on $\delta$ and arguing in the same was as leading to (\ref{eq: difference of K and 1}) to obtain  \begin{equation}\label{eq: difference of KN and K}\begin{split} & \int_E |K^N-K|(t,v,v_\star,\sigma)\overline{m}_{\mu^N}(dt,dv,dv_\star,d\sigma) \\& \hspace{4cm}\le CT\left(\left(\frac{N}{N_0(M_N,r_N)}-1\right)+\frac{1}{M_N}\right)\langle 1+|v|^2, \mu^N_0\rangle^2.\end{split} \end{equation} This is, in fact, the same estimate as before, up to the inclusion of $\delta$; the linear factor $(1+\delta|v-v_\star|)$ now included in $K, K^N$ replaces the equivalent one previously in the measure $\overline{m}_{\mu^N}=B(v-v_\star)\mu^N_t(dv)\mu^N_t(dv_\star)d\sigma dt$ so that the previous calculations are unchanged. With these modifications, the proof of the law of large numbers works exactly as before. \bigskip \\ We again estimate the change of measure $\frac{1}{N}\log \frac{d\PPm^N}{d\PP}$. In this case, we will only find an estimate which asymptotically holds with sufficiently large probability, rather than with probabilities converging to $1$ as we did before; this will not affect the final result. We recall that \begin{equation}\label{eq: reiteration of COM}\begin{split} \frac{1}{N}\log \frac{d\mathbb{Q}^N}{d\PP}& =\langle \varphi_{M_N}, \mu^N_0\rangle + \langle \log K^N, w^N\rangle \\ & \hs \hs - \int_E(K^N-1)(t,v,v_\star,\sigma)\overline{m}_{\mu^N}(dt,dv,dv_\star,d\sigma)\end{split}\end{equation} where again, $K^N$ is allowed to depend on $\mu^N_0$. Let us fix $\epsilon>0$. The term from the change of initial data is exactly as in the hard spheres case: \begin{equation} \label{eq: t1 in MM COM} \mathbb{Q}^N\left( \langle \varphi_{M_N}, \mu^N_0\rangle > z_2\Theta(T) +\frac{\epsilon}{3}\right) \to 0. \end{equation} In the second term, we now use the upper bound \begin{equation} \begin{split}\label{eq: upper bound KN MM} \log K^N &\le \log N/N_0(M_N, r_N) + \log (1+\delta|v-v_\star|) \\& \le \log N/N_0(M_N, r_N) + \delta(|v|+|v_\star|). \end{split} \end{equation} As in the hard spheres case, the first term contributes at most $(\log N/N_0)w^N(E) \le \epsilon/2$ with high $\mathbb{Q}^N$-probability, and in the second term, observe that  \begin{equation}\begin{split} &\delta \langle |v|+|v_\star|, w^N_t\rangle \\& \hs - \delta \int_{(0,t]\times\rrd\times\rrd\times\ssd}  \frac{N}{N_0} (1+\delta|v-v_\star|)(|v|+|v_\star|)ds\mu^N_s(dv)\mu^N_s(dv_\star)d\sigma\end{split} \end{equation} is a $\mathbb{Q}^N$-supermartingale, so there exists a constant $C$ such that \begin{equation} \begin{split} \label{eq: second moment on MM change of measure} &\mathbb{E}_{\mathbb{Q}^N}\left[\delta\langle |v|+|v_\star|, w^N\rangle 1\left(N_0\ge \frac{1}{2}N\right)\right]\\ & \hs \hs\le \delta \mathbb{E}_{\mathbb{Q}^N}\left[\int_E 2(|v|+|v_\star|+2\delta |v|^2+2\delta|v_\star|^2)ds \mu^N_s(dv)\mu^N_s(dv_\star)d\sigma\right] \\[1ex] &\hs \hs \le \delta C\mathbb{E}_{\mathbb{Q}^N}\left[\int_0^T\langle |v|^2, \mu^N_s\rangle ds\right] = \delta CT\Theta(T). \end{split} \end{equation} Therefore, up to a new choice of $C$, for all $N$, \begin{equation} \mathbb{Q}^N\left(\delta \langle |v|+|v_\star|, w^N\rangle > \delta C T\Theta(T), \h\h N_0\ge \frac{1}{2}N\right) \le \frac{1}{9}\end{equation} and recalling (\ref{eq: construction of QN2}), for all $N$ sufficiently large, $\mathbb{Q}^N(N_0\le N/2) \le \frac{1}{9}$, so, for all $N$ sufficiently large, \begin{equation}\label{eq: t3a in MM COM} \mathbb{Q}^N\left(\delta \langle |v|+|v_\star|, w^N\rangle > \delta C T\Theta(T)\right) \le \frac{2}{9}\end{equation} and including the term $\log N/N_0$, we conclude that \begin{equation}\label{eq: t2 in MM COM} \mathbb{Q}^N\left(\langle \log K^N, w^N\rangle > \delta C T\Theta(T) + \epsilon / 3 \right) < \frac{1}{3}. \end{equation} For the final term of (\ref{eq: reiteration of COM}), we observe that \begin{equation} |K^N-1| \le |K^N-K|+|K-1|\le |K^N-K|+\delta(|v|+|v_\star|)\end{equation}  and arguing from (\ref{eq: difference of KN and K}), for all $N$ sufficiently large, \begin{equation} \mathbb{Q}^N\left(\int_E |K^N-K|\overline{m}_{\mu^N}(dt,dv,dv_\star,d\sigma) >\epsilon/3\right) \to 0 \end{equation}  while in the second term, we have the pathwise inequality \begin{equation} \begin{split} \int_E \delta(|v|+|v_\star|)\overline{m}_{\mu^N}(dt,dv,dv_\star,d\sigma) \le 2 \delta \int_0^T \langle |v|^2, \mu^N_t\rangle dt = 2\delta T\langle |v|^2, \mu_0\rangle. \end{split} \end{equation}The right-hand side converges with $\PPm^N$-probability to $2\delta T\Theta(T)$ by (\ref{eq: construction of QN1}), and so \begin{equation} \label{eq: t3 in MM COM}\mathbb{Q}^N\left(\int_E |K^N-1| \overline{m}_{\mu^N}(dt,dv,dv_\star,d\sigma) > 2\delta T\Theta(T) +\epsilon/3 \right)\to 0. \end{equation} Gathering (\ref{eq: t1 in MM COM}, \ref{eq: t2 in MM COM}, \ref{eq: t3a in MM COM}, \ref{eq: t3 in MM COM}) and returning to (\ref{eq: reiteration of COM}), we conclude that, for some absolute constant $C$ and all $N$ large enough, \begin{equation} \label{eq: MM COM final} \mathbb{Q}^N\left( \frac{1}{N}\log \frac{d\PPm^N}{d\PP} > \Theta(T) (z_2+C\delta) +\epsilon \right) < \frac{1}{4}. \end{equation} Exactly the same argument also implies that \begin{equation}\lim_{a\to \infty} \limsup_N \PPm^N\left(\frac{d\PPm^N}{d\PP} > e^{Na}\right)=0 \end{equation} so that Corollary \ref{cor: tightness} applies. \bigskip \\
 The conclusions of the theorem now follow in the same pattern as the hard spheres case. For the dynamic cost of any $(\mu_\bullet, w)\in \A_{\Theta,\delta}$, one bounds $\tau(k)\le (k-1)^2$ to obtain \begin{equation}\begin{split} \mathcal{J}(\mu_\bullet, w)&= \int_E \tau(1+\delta|v-v_\star|)dt\mu_t(dv)\mu_t(dv_\star)d\sigma \\ &\le 2\delta^2 \int_E (|v|^2+|v_\star|^2)dt\mu_t(dv)\mu_t(dv_\star)d\sigma \\ & = 4\delta^2 \int_0^T \Theta(t)dt \le 4\delta^2 T\Theta(T)\end{split} \end{equation} and recalling that $\mu_0=\mu_0^\star$ for all such $\mu_\bullet$, we conclude that the same bound holds for $\mathcal{I}(\mu_\bullet, w)$. If we now fix an open set $\U\supset \A_{\Theta,\delta}$, we let $S'$ be an infinite subsequence along which $N^{-1}\log \PP((\mu^N_\bullet, w^N)\in \U)$ converges to its $\liminf$; using Corollary \ref{cor: tightness} to prove tightness, we can pass to a further subsequence $S$ such that the laws $\PPm^N\circ (\mu^N_\bullet, w^N)^{-1}$ converge weakly for the changes of measure above. By the law of large numbers, for $N\in S$ large enough, \begin{equation} \PPm^N\left((\mu^N_\bullet, w^N)\in \U\right)\ge \frac{1}{2} \end{equation} and, for $\epsilon>0$ fixed, combining with (\ref{eq: MM COM final}), for all sufficiently large $N$, \begin{equation} \PPm^N\left((\mu^N_\bullet, w^N)\in \U, \frac{1}{N}\log \frac{d\PPm^N}{d\PP} \le \Theta(T)(z_2+C\delta)+\epsilon \right) \ge \frac{1}{4}. \end{equation} For such $N\in S$, we invert in the usual way to find \begin{equation} \PP\left((\mu^N_\bullet, w^N)\in \U\right)\ge \frac{1}{4} \exp\left(-N(\Theta(T)(z_2+C\delta)+\epsilon)\right). \end{equation} Since $S\subset S'$ attains the $\liminf$, we take the logarithm and send $N\to \infty$ through $S$ to obtain \begin{equation}  \liminf_N \frac{1}{N}\log \PP\left((\mu^N_\bullet, w^N)\in \U\right) \ge -\Theta(T)(z_2+C\delta)-\epsilon \end{equation} and taking $\epsilon\to 0$ proves the claim. The final item, regarding open $\V_\delta$, follows in the same way as in the hard spheres case: we fix an open interval $I\ni T$ on which $\Theta$ is constant, and bounded away from $P$, and write $|I|$ for its Lebesgue measure. Recalling that the fourth moment condition on $\A_{\Theta, \delta}$ depends on $\delta$, we can choose $R=R_\delta$ and a continuous, compactly supported $0\le f_{\delta}\le |v|^2$, which coincides on $|v|^2$ when $|v|\le R_\delta$ such that, for all $(\mu_\bullet, w)\in \A_{\Theta,\delta}$, \begin{equation} \inf_{t\in I} \langle f_\delta, \mu_t\rangle > \frac{1+\Theta(T)}{2}\end{equation} and, following the previous case, take \begin{equation} \V_\delta:=\left\{(\mu_\bullet, w)\in \D\times\Mm: \int_I \langle f_\delta,\mu_t\rangle dt > \left(\frac{1+\Theta(T)}{2}\right)|I|\right\}.\end{equation} Using Lemma \ref{lemma: really useful continuity result} as before, these are open and contain $\A_{\Theta, \delta}$ by construction, and uniformly in $\delta$, \begin{equation}  \frac{1}{N}\log \PP\left((\mu^N_\bullet, w^N)\in \V_\delta\right) \le \frac{1}{N}\log \PP\left(\langle |v|^2, \mu^N_0\rangle > \frac{1+\Theta(T)}{2}\right) \end{equation} so by Cram\'er, \begin{equation} \begin{split} \liminf_N \frac{1}{N}\log \PP\left((\mu^N_\bullet, w^N)\in \V_\delta\right) &\le \frac{1}{N}\log \PP\left(\langle |v|^2, \mu^N_0\rangle > \frac{1+\Theta(T)}{2}\right)\\ & \le -\psi_0^\star\left(\frac{1+\Theta(T)}{2}\right) <0. \end{split} \end{equation} The final bound is uniform in $\delta>0$, and the theorem is complete.\end{proof}

\section{Proof of Corollaries} \label{sec: corrs} We now give the two corollaries \ref{cor: bad by evolution}, \ref{cor: no energy conserving LDP}. \begin{proof}[Proof of Corollary \ref{cor: bad by evolution}] Let $B$ be either the hard-spheres or Maxwell Molecules kernel, and let $\mu_0^\star=\gamma$, recalling that $\gamma(dv)=e^{-|v|^2/2d}/(2\pi d)^{d/2}dv$. We start with the well-known observation that, in either case, all Kac processes $\mu^N_\bullet$ are reversible in equilibrium when the initial data are sampled independently from $\gamma$, so that the law of $\mu^N_\bullet$ is the same as the time-reversal $$ \mathbb{T}\mu^N_\bullet=\left(\mu^N_{(T-t)-}\right)_{0\le t\le T}. $$ Let us now set $\Theta_\mathbb{T}$ to be the time reversed function $\Theta_\mathbb{T}(t):=\Theta(T-t)$. By hypothesis, $\Theta_\mathbb{T}$ satisfies the conditions required in Theorem \ref{thm: main}; in the case of Maxwell Molecules, choose $\delta>0$ arbitrarily, and in either case set $A_{\mathbb{T}}$ to be the fourth moment bound given by Theorem \ref{thm: main} and $\mathcal{A}_{\mathbb{T},\Theta}$ the resulting bad set constructed by Theorem \ref{thm: main}. We now set $A(t):=A_\mathbb{T}(T-t)$ and set $\widehat{\mathcal{A}}_{\mathbb{T}}$ to be the projection \begin{equation} \widehat{\A}_{\mathbb{T}}=\left\{\mu_\bullet: (\mu_\bullet, w)\in \A_{\mathbb{T},\Theta}\right\}.\end{equation} Since $\A_{\mathbb{T},\Theta}$ are compact and $\mathbb{T}$ preserves the Skorokhod topology of $\D$, it follows that $\widehat{\A}_{\mathbb{T}}$ are also compact, as are $\widehat{\A}:=\{\mathbb{T}\mu_\bullet: \mu_\bullet \in \widehat{\A}_\mathbb{T}\}$  and by construction, the set desired can be written as $\mathcal{B}=\widehat{\A}\times\Mm$.\bigskip \\  Let us now fix $\U\supset \mathcal{B}$ open and $M>0$ to be chosen later. Thanks to Lemma \ref{lemma: easy UT}, we can choose a compact $\K\subset\Mm$ such that $\PP(w^N\not\in \K)\le e^{-MN}$ for all $N$, and since $\widehat{\A}\times \K$ is compact, we can choose $\U_1, \U_2$, open in $\D, \Mm$ respectively, such that $\widehat{\A}\times \K\subset\U_1\times\U_2\subset\U$. Now, $\mathbb{T}\U_1$ is open and contains $\widehat{\A}_\mathbb{T}$ and using reversibility, \begin{equation}\label{eq: reversibility} \PP(\mu^N_\bullet\in \mathbb{T}\U_1)=\PP(\mathbb{T}\mu^N_\bullet\in \U_1)=\PP(\mu^N_\bullet\in \U_1). \end{equation} Using Theorem \ref{thm: main} in either of the two cases on the open set $\mathbb{T}\U_1\times\Mm\supset \A_{\mathbb{T},\Theta}$, for some finite $C$, independent of $M$, it holds that \begin{equation} \liminf_N \frac{1}{N}\log \PP\left(\mu^N_\bullet \in \mathbb{T}\U_1\right) =  \liminf_N \frac{1}{N}\log \PP\left((\mu^N_\bullet, w^N) \in \mathbb{T}\U_1\times\Mm\right) \ge -C \end{equation} and thanks to (\ref{eq: reversibility}), the same holds with $\U_1$ in place of $\mathbb{T}\U_1$. We now observe that \begin{equation}\begin{split} & -C\le \liminf_N \frac{1}{N}\log \PP\left((\mu^N_\bullet, w^N)\in \U_1\times\Mm\right) \\& \hs \hs \le  \max\left(\liminf_N \frac{1}{N}\log \PP\left((\mu^N_\bullet, w^N) \in \U_1\times\U_2\right), \liminf_N \frac{1}{N}\log \PP\left(w^N\not \in \U_2\right)\right) \\& \hs \hs \le  \max\left(\liminf_N \frac{1}{N}\log \PP\left((\mu^N_\bullet, w^N) \in \U\right), -M\right)\end{split} \end{equation} where, in the final line, we use the choice of $\K$ and recall that $\U_2\supset\K$. If we now choose $M>C$, we must have that \begin{equation} \liminf_N \frac{1}{N} \log \PP\left((\mu^N_\bullet, w^N)\in \U\right)\ge -C>-\infty \end{equation} as claimed.  \end{proof}

We next prove that there can be no energy-conserving large deviation principle.

\begin{proof}[Proof of Corollary \ref{cor: no energy conserving LDP}]  Throughout, fix $\Theta$ arbitrarily as in the statement of Theorem \ref{thm: main}, and, in the case of Maxwell molecules, pick $\delta>0$ arbitrarily, and let $\A=\A_{\Theta}, \A_{\Theta,\delta}$ be the resulting `bad' set from Theorem \ref{thm: main} in either case.  For a contradiction, let $(\mu^N_\bullet)_{N\in S}$ be a subsequence which satisfies a large deviation principle with a rate function $\widetilde{\mathcal{I}}$ such that $\widetilde{\mathcal{I}}(\mu_\bullet,w)=\infty$ if $\mu_\bullet$ does not conserve energy. Since no paths in $\A$ conserve energy, we know that $\mathcal{I}(\mu_\bullet,w)=\infty$ for all $(\mu_\bullet,w) \in \A$ by construction. Due to exponential tightness in Proposition \ref{prop: ET + UB}, the rate function must be good; that is, the sublevel sets $\{(\mu_\bullet,w) \in \mathcal{D}\times\Mm: \widetilde{\mathcal{I}}(\mu_\bullet,w)\le a\}$ are compact in $\D\times\Mm$ for any $a\in [0,\infty)$. Now, for any $a$, $\{ \widetilde{\mathcal{I}}\le a\}$ is disjoint from $\mathcal{A}$, and since $\D\times\Mm$ is a normal topological space, there exists an open set $\mathcal{U}_a \supset \mathcal{A}$ whose closure $\overline{\mathcal{U}}_a$  is disjoint from $\{\widetilde{I}\le a\}$. By hypothesis, \begin{equation} \begin{split}\label{eq: LDP contradiction 1} \limsup_{N\in S} \left[\frac{1}{N}\log \PP\left((\mu^N_\bullet,w^N) \in \overline{\mathcal{U}}_a\right)\right] &\le -\inf\left\{\widetilde{I}(\mu_\bullet,w): \mu_\bullet \in \overline{\mathcal{U}}_a\right\} \\ & \le -a. \end{split}\end{equation} This is inconsistent with (\ref{eq: first item of main},\ref{eq: first point of main MM}) for $a$ large enough, and we have the desired contradiction. \end{proof}

 Finally, we prove the result on entropy as a quasipotential. \begin{proof}[Proof of Corollary \ref{cor: quasipotential}] This follows from Theorem \ref{thrm: main positive} using a contraction principle argument. Since we do not have a true large deviation principle, and must further compensate for the failure of the rate function $\mathcal{I}$ to be good, the arguments do not follow from any statement of the contraction principle we have found in the literature, and we present the arguments in detail. \bigskip \\  Let us fix $\mu\in \cp_2$ and $\epsilon>0$ and consider \begin{equation} \U_\epsilon:=\left\{(\mu_\bullet, w)\in \D\times\Mm: W(\mu_T, \mu)<\epsilon\right\} \end{equation} so that the closure is \begin{equation} \overline{\U}_\epsilon:=\left\{(\mu_\bullet, w)\in \D\times\Mm: W(\mu_T, \mu)\le \epsilon\right\} \end{equation}Let us take $\mu_0^\star=\gamma$ and, for each $N$, sample initial velocities independently from $\gamma$. In this case, the $N$-particle system is in equilibrium, so that the distribution of $\mu^N_T$ is that of a $N$-particle independent sample from $\gamma$, and in particular, Sanov's theorem applies, so that $\mu^N_T$ satisfies a large deviation principle with rate function $H(\cdot|\gamma)$. We first prove the first item (\ref{eq: LB of quasipotential}): by Sanov's Theorem \begin{equation}\label{eq: sanov 1} \begin{split} \liminf_N \frac{1}{N} \log \PP\left((\mu^N_\bullet, w^N)\in {\U}_\epsilon\right)&=\liminf_N \frac{1}{N} \log \PP\left(W(\mu^N_T, \mu)< \epsilon \right) \\ & \ge  -\inf\{H(\nu|\gamma): W(\nu, \mu)<  \epsilon\} \\ & \ge -H(\mu|\gamma) \end{split} \end{equation} while, immediately  \begin{equation}\begin{split} \label{eq: apply UB for QP}\liminf_N \frac{1}{N} \log \PP\left((\mu^N_\bullet, w^N)\in {\U}_\epsilon\right)& \le \limsup_N \frac{1}{N} \log \PP\left( (\mu^N_\bullet, w^N)\in \overline{\U}_\epsilon\right). \end{split}\end{equation} If $H(\mu|\gamma)=\infty$ there is, of course, nothing to prove; otherwise, we choose $M>H(\mu|\gamma)$ and using Proposition \ref{prop: ET + UB}i), pick a compact set $\K\subset \D\times\Mm$ such that \begin{equation}\limsup_N N^{-1}\PP((\mu^N_\bullet, w^N)\not\in \K)\le -M.\end{equation} From (\ref{eq: apply UB for QP}) and applying Theorem \ref{thrm: main positive}i), \begin{equation}\begin{split} &\limsup_N \frac{1}{N}\log \PP\left((\mu^N_\bullet, w^N)\in \overline{\U}_\epsilon\right)\\ &\hs   \le \max\left(\limsup_N \frac{1}{N}\log \PP\left((\mu^N_\bullet, w^N)\in \overline{\U}_\epsilon \cap \K\right), \limsup_N \frac{1}{N}\log \PP\left((\mu^N_\bullet, w^N)\not\in \K\right)\right) \\ & \hs \le \max\left(-\inf\left\{\mathcal{I}(\nu_\bullet, w): (\nu_\bullet, w)\in \overline{\U}_\epsilon \cap \K)\right\}, -M\right). \end{split} \end{equation} Comparing against (\ref{eq: sanov 1}), we must have that \begin{equation} H(\mu|\gamma)\ge \min\left(\inf\{\mathcal{I}(\nu_\bullet, w): (\nu_\bullet, w)\in \overline{\U}_\epsilon \cap \K\}, M\right) \end{equation} and since $M>H(\mu|\gamma)$ by construction, \begin{equation}\label{eq: QP 1 pre-limit} H(\mu|\gamma)\ge\inf\{\mathcal{I}(\nu_\bullet, w): (\nu_\bullet, w)\in \overline{\U}_\epsilon \cap \K\} \end{equation}   We claim that the right-hand side converges as $\epsilon\downarrow 0$: \begin{equation}\label{eq: claimed convergence} \inf\left\{\mathcal{I}(\nu_\bullet, w): (\nu_\bullet, w)\in \overline{\U}_\epsilon \cap \K)\right\}\to \inf\{\mathcal{I}(\nu_\bullet, w): \nu_T=\mu, (\nu_\bullet, w)\in \K\}.\end{equation} It is immediate that the left-hand side is increasing as $\epsilon\downarrow  0$ and that the right-hand side is an upper bound; it is therefore sufficient to prove convergence on a subsequence.  For each $n$, pick $(\nu^{(n)}_\bullet, w^{(n)})\in \overline{\U}_{1/n}\cap\K$ with error at most $1/n$ from the infimum. Since $\K$ is compact, we can pass to a subsequence with $(\nu^{(n)}_\bullet, w^{(n)})\to (\nu_\bullet, w)$; the limit has $\nu_T=\mu$ and $(\nu_\bullet, w)\in \K$. By lower-semicontinuity from Proposition \ref{prop: ET + UB}ii), we have\begin{equation} \mathcal{I}(\nu_\bullet, w)\le \liminf_n \mathcal{I}(\nu^{(n)}_\bullet, w^{(n)}) \le \liminf_n \left(\inf\left\{\mathcal{I}(\nu_\bullet, w): (\nu_\bullet, w)\in \overline{\U}_{1/n} \cap \K)\right\} + \frac{1}{n}\right)\end{equation} so that \begin{equation} \inf\{\mathcal{I}(\nu_\bullet, w): \nu_T=\mu, (\nu_\bullet, w)\in \K\}\le \liminf_n \left(\inf\left\{\mathcal{I}(\nu_\bullet, w): (\nu_\bullet, w)\in \overline{\U}_{1/n} \cap \K)\right\} \right) \end{equation} which proves the claim (\ref{eq: claimed convergence}). Returning to (\ref{eq: QP 1 pre-limit}), we take $\epsilon\to 0$ to find \begin{equation} \begin{split} H(\mu|\gamma) & \ge \inf\left\{\mathcal{I}(\mu_\bullet, w): \nu_T=\mu, (\nu_\bullet, w)\in \K\right\} \\ & \ge \inf\left\{\mathcal{I}(\mu_\bullet, w): \nu_T=\mu\right\}  \end{split} \end{equation} and observe that the right-hand side is exactly the claimed bound in (\ref{eq: LB of quasipotential}). For the second item (\ref{eq: UB of quasipotential}), we apply the lower bound of Sanov: \begin{equation}\begin{split} \label{eq: sanov 2} \limsup_N \frac{1}{N} \log \PP\left((\mu^N_\bullet, w^N)\in \overline{\U}_\epsilon\right)&=\limsup_N \frac{1}{N} \log \PP\left(W(\mu^N_T, \mu)\le \epsilon \right) \\& \le -\inf\{H(\nu|\gamma): W(\nu, \mu)\le \epsilon\}. \end{split} \end{equation}    On the other hand,   \begin{equation}\begin{split} \limsup_N \frac{1}{N} \log \PP\left((\mu^N_\bullet, w^N)\in \overline{\U}_\epsilon\right)& \ge \liminf_N \frac{1}{N} \log \PP\left( (\mu^N_\bullet, w^N)\in \U_\epsilon\right) \\ & \ge -\inf\left\{\mathcal{I}(\nu_\bullet, w): (\nu_\bullet, w)\in \U_\epsilon\cap\mathcal{R}\right\} \\ & \ge -\inf\left\{\mathcal{I}(\nu_\bullet, w): \nu_T=\mu, (\nu_\bullet, w)\in \mathcal{R}\right\}. \end{split}\end{equation} We conclude that  \begin{equation} \inf\{H(\nu|\gamma): W(\nu,\mu)\le \epsilon\} \le \inf\left\{\mathcal{I}(\nu_\bullet, w): \nu_T=\mu, (\nu_\bullet,w)\in \mathcal{R}\right\}. \end{equation} As $\epsilon\to 0$, the left-hand side converges to $H(\mu|\gamma)$ by the lower semi-continuity of entropy (cf. Lemma \ref{lemma: semiconntinuity}), and the right-hand side is exactly the right-hand side of (\ref{eq: UB of quasipotential}), so we are done.  \end{proof}

\begin{appendix}
\section{Some Properties of Skorohod Paths} \label{sec: sk} We will now recall some facts about right-continuous, left-limited (c{\`a}dl{\`a}g) paths, and the resulting Skorohod topology. For a fixed metric space $(X,d)$ and $T$, we write $D([0,T], (X,d))$ for the set of all such functions $x_\bullet: [0,T]\to X$, which we equip with the metric  \begin{equation}\label{eq: sk met}\rho(x_\bullet, y_\bullet)=\inf\left\{\max\left(\sup_{t\le T} d(x(t), y(\iota(t))),\sup_{t\le T}|t-\iota(t)|\right): \iota \in \Lambda\right\} \end{equation} where the infimum runs over the set $\Lambda$ of increasing, continuous bijections $\iota:[0,T]\to[0,T]$. We say that $x$ has a jump of size at least $\epsilon>0$ at $t$ if $d(x(t), x(t-))\ge \epsilon$. \bigskip \\  Our first result is a replacement for uniform continuity in the context of such paths.
\begin{proposition}\label{prop: Sk continuity}  Let  $x_\bullet \in D([0,T],(X,d))$ and fix $\epsilon>0$. Then \begin{enumerate}[label=\alph*).] \item There exists at most finitely many $t\in [0,T]$ such that $d(x(t), x(t-))> \epsilon$. \item There exists $\delta>0$ such that, for all $t\in [0,T]$, \emph{either} there exists $s\in [t, t+\delta)\cap [0,T]$ with a jump discontinuity  of size at least $d(x(s-), x(s))\ge \epsilon$ \emph{or}, for all $s\in [t, t+\delta)\cap [0,T]$, we have $d(x(t), x(s))<\epsilon$. \end{enumerate} \end{proposition} \begin{proof} For the first item, suppose that we can find a countable sequence of distinct $t_n\in (0,T]$ such that $d(x(t_n-), x(t_n))> \epsilon$, and up to passing to an infinite subsequence, we can also arrange that $t_n$ converges monotonically, either increasingly or decreasingly, to a limit $t\in [0,T]$. We consider the two cases separately: \begin{enumerate} \item If $t_n\uparrow t$, we can pick $s_n\in [t_n-n^{-1}, t_n]$ such that $d(x(s_n), x(t_n))>\epsilon$, which contradicts the fact that both $x(s_n), x(t_n)\to x(t-)$ by the left-limitedness.\item If $t_n\downarrow t$, we can pick $s_n\in (t, t_n)$, still so that $d(x(s_n), x(t_n))>\epsilon$, and obtain the same contradiction by the convergence $x(t_n), x(s_n)\to x(t)$ by right-continuity.\end{enumerate} In either case, we have a contradiction, so the claim is proven. \bigskip \\ We now prove item b). Suppose, for a contradiction, that the conclusion is false, so that we can construct $t_n, s_n$, with $t_n<s_n<t_n+n^{-1}$, such that there is no jump of size $\ge \epsilon$ in $[t_n, s_n)$, but such that $d(x(t_n), x(s_n))\ge \epsilon$. As before, by passing to a infinite subsequence, we can arrange that either $t_n\downarrow  t$ or $t_n\uparrow t$. We again deal with the cases separately. \begin{enumerate} \item If $t_n\uparrow t$, we split further into cases, depending on whether $s_n>t$ infinitely often or not. \begin{enumerate} \item If $s_n>t$ infinitely often, we can pass to a further subsequence so that $t_n\uparrow t, s_n\downarrow t$, so that $d(x(t_n), x(s_n))\to d(x(t-), x(t))$. Since $d(x(t_n), x(s_n))\ge \epsilon$ for all $n$ by construction, we conclude that there is a jump discontinuity of size $\ge \epsilon$ at $t$, which contradicts the hypothesis that $[t_n, s_n)$ contains no such jumps. \item Otherwise, $s_n\le t$ eventually, so by passing to a subsequence, $t_n, s_n\uparrow t$ and $x(t_n), x(s_n)\to x(t-)$, which is a contradiction in the usual way.  \end{enumerate} \item If $t_n\downarrow t$, then $s_n\downarrow t$ and $x(t_n), x(s_n)\to x(t)$, contradicting that $d(x(t_n), x(s_n))\ge \epsilon$. \end{enumerate} Since all possible cases lead to a contradiction, the claim is proven. \end{proof} We next classify some continuity properties for the Skorokhod convergence. These results are standard and included for completeness. \begin{proposition}\label{prop: Sk continuous functions} \begin{enumerate}[label=\alph*).]\item The maps $x_\bullet \mapsto x(0)$, $x_\bullet\mapsto x(T)$ are continuous with respect to the metric $\rho$. \item If $x^n_\bullet\in D([0,T],(X,d))$ converge to $x_\bullet$ with respect to $\rho$, then for all but countably many $t\in [0,T]$, $d(x^n(t), x(t))\to 0$.  \item If, in b), the limit path $x_\bullet$ is continuous, then we additionally have the uniform convergence $\sup_t d(x^n(t), x(t))\to 0$. \end{enumerate} \begin{proof} For the first item, observe that $\iota(0)=0, \iota(T)=T$ for all $\iota\in \Lambda$, which implies that $d(x(0), y(0))\le \rho(x_\bullet, y_\bullet)$ for all $x_\bullet, y_\bullet$, and similarly at $T$. For the second item, from the previous proposition, $x_\bullet$ is continuous at all but countably many $t\in [0,T]$. For points of continuity $t$ of $x_\bullet$, fix $\epsilon>0$: there exists $\delta>0$ such that, for all $s$ with $|s-t|<\delta$, $|x(s)-x(t)|<\epsilon/2$. For all $n$ sufficiently large, we have $\rho(x^n_\bullet, x_\bullet)<\min(\delta, \epsilon/2)$ and so we can pick $\iota\in \Lambda$ such that $\sup |t-\iota(t)|<\delta$ and $\sup_t |x^n(\iota(t))-x(t)|<\epsilon/2$. We now conclude: we have $|t-\iota^{-1}(t)|<\delta$, and so \begin{equation} |x^n(t)-x(t)|\le |x^n(t)-x(\iota^{-1}(t))|+|x(\iota^{-1}(t))-x(t)| < \epsilon/2 + \epsilon/2 = \epsilon \end{equation} and we are done. The final item also follows, noting that as $x_\bullet$ is continuous, it is uniformly continuous, which implies that $\delta$, and hence $n$, can be chosen independently of $t\in [0,T]$. \end{proof} \end{proposition}
\section{A Singular Girsanov Theorem for Jump Processes}\label{sec: singular girsanov} We now justify the changes of measure in Proposition \ref{prop: com}. We start from a filtered probability space $(\Omega, \mathfrak{F},(\mathfrak{F}_t)_{t\ge 0},\PP)$, on which is defined a Kac process $\mu^N_\bullet$ its empirical flux $w^N_t$. We have a deterministic tilting $\varphi:\rrd\to \mathbb{R}$ of the initial data, such that $\int e^{\varphi(v)} \mu_0^\star(dv)=1$, and  $K:\cp_2^N\times E\to [0,\infty)$ be measurable, with a bound $K/(1+|v|+|v_\star|)\le C$, for some absolute constant $C$. The modification of the dynamics with therefore be random, depending on the initial value $\mu^N_0$: our new measures are given by \begin{equation}\begin{split}\label{eq: COM0'} \frac{d\mathbb{Q}}{d\PP}=&\exp\bigg(N\langle \varphi, \mu^N_0\rangle +N\langle \log K(\mu^N_0, \cdot), w^N_T\rangle \\&\hspace{2cm}- N \int_{E}(K-1)(\mu^N_0,t,v,v_\star,\sigma)\overline{m}_{\mu^N}(dt,dv,dv_\star,d\sigma)\bigg)\end{split}\end{equation} where, if $w^N$ has any point with $K=0$, then the integral $\log K(\mu^N_0, \cdot), w^N\rangle=-\infty$ and the density is understood to be $0$. \bigskip \\ We start from a disintegration of $\PP$: let $L_0$ be the law of $\mu^N_0$ on $\cp^N_2$, and for any $\nu\in \cp^N_2$, let $\PP_{\nu}$ be the law of the Kac process started from $\nu$, so that, for any $A\in \mathfrak{F}$, \begin{equation} \PP(A)=\int_{\cp^N_2} \PP_{\nu}(A)L_0(d\nu).\end{equation} We now write, again for any $A\in \mathfrak{F}$, \begin{equation} \label{eq: disintegration} \PPm(A)=\int_{\cp^N_2} \EE_{\nu}\left[Z^{\delta,\nu}_T1_A\right]\left(e^{N\langle \varphi, \nu\rangle}L_0\right)(d\nu) =\int_{\cp^N_2} \PPm_{\nu}(A)\widetilde{L}_0(d\nu) \end{equation} where we define the modified law $\widetilde{L}_0$ by \begin{equation} \widetilde{L}_0(d\nu)=e^{N\langle \varphi, \nu\rangle}L_0(d\nu) \end{equation} and modify the conditional law $\PP_\nu$  by \begin{equation} Z^{\nu}_{t}=\exp\left(N\langle \log K(\nu,\cdot), w^N_t\rangle - N\int_E 1_{s\le t}(K^\delta-1)(\nu, s,v,v_\star,\sigma)\overline{m}_{\mu^N}(ds,dv,dv_\star,d\sigma)\right) \end{equation} where we again set $Z^\nu_t=0$ if there are any point with $K=0$, and $\PPm_\nu=Z^{\nu}_{T}\PP_\nu$. \bigskip \\ For the initial law, we can formally describe $L_0$ as the pushforward of $(\mu_0^\star)^{\otimes N}$ by the map $\pi^N: (v_1,..,v_N)\to N^{-1}\sum \delta_{v_i}$, and that $\exp(N\langle \varphi, \nu\rangle)=\exp(\sum \varphi(v_i))$. It therefore follows that $\widetilde{L}_0$ is the pushforward of the measure $\exp(\sum \varphi(v_i))\prod_i \mu_0^\star(dv_i)=\prod_i e^{\varphi(v_i)}\mu_0^\star(dv_i)$ by $\pi^N$. Since $\varphi$ was chosen so that $\int e^{\varphi}d\mu_0^\star=1$, each factor is a probability measure, and so $\widetilde{L}_0$ is the probability measure on $\cp_2^N$ for the empirical measure of sampling $N$ particles independently from the probability measure $e^{\varphi}\mu_0^\star$, as claimed. \bigskip \\ We now consider the modification of each $\PP_\nu$. We observe that the conservation of energy guarantees that there exists $M=M_\nu$ such that, $\PP_\nu$-almost surely, $\mu^N_t$ is supported on $[-M,M]^d$ for all $t$, and $w^N$ is supported on $(t,v,v_\star,\sigma)\in E$ with $|v|, |v_\star|\le M$. In particular, thanks to the hypothesised bound, one finds the upper bound \begin{equation}\label{eq: upper bound for Z} \sup_{t\le T} Z^{\nu}_t \le \exp\left(CNM_\nu (1+w^N(E))\right)  \end{equation} and the right-hand side has all moments finite, since $Nw^N_t(E)$ can be dominated by a Poisson process of rate $3(1+\langle |v|^2, \nu\rangle)$, as in Section \ref{subsec: ET}. We now observe that, at collisions, $Z^{\nu}_t$ changes by \begin{equation}\label{eq: jumps of Z} Z^{\nu}_t-Z^{\delta,\nu}_{t-} = Z^{\nu}_{t-}\left(e^{\log K(\nu, t,v,v_\star,\sigma)}-1\right) = Z^{\nu}_{t-}\left( K(\nu, t,v,v_\star,\sigma)-1\right)\end{equation} which is valid whether or not $K(\nu,t,v,v_\star,\sigma)$, while in between collisions, $Z^{\nu}_t$ is differentiable, with \begin{equation}\label{eq: drift of Z} \frac{d}{dt}Z^{\nu}_t = -N\int_{\rrd\times\rrd\times\ssd}(K-1)(\nu,t,v,v_\star,\sigma)B(v-v_\star)\mu^N_t(dv)\mu^N_t(dv_\star). \end{equation} Together, we obtain \begin{equation}Z^{\nu}_t=Z^{\nu}_0+\int_E 1_{s\le t}NZ^{\nu}_{s-}\left(K(\nu,s,v,v_\star,\sigma)-1\right)(w^N-\overline{m}_{\mu^N})(ds,dv,dv_\star,d\sigma)\end{equation} which is a $\PP_\nu$-local martingale, and hence a true martingale using the upper bound (\ref{eq: upper bound for Z}), with constant mean $Z^{\nu}_0=1$. It follows that each $\PPm_\nu$ is a probability measure, and hence so is $\PPm$. \bigskip \\ Let us now describe the dynamics under each $\PPm_\nu$. Let us fix a bounded, measurable function $F:\cp^2_N\times\cp^2_N\times \Mm\to \rr$, and let $A_t$ be given by \begin{equation}\begin{split} \label{eq: time dependent generator'}A_t=&N\int_{\rrd\times\rrd\times\ssd}(F(\nu,\mu^{N,v,v_\star,\sigma},w^{N,t,v,v_\star,\sigma})-F(\nu,\mu^N, w^N)))\\ & \hs \hs \hs\hs \dots \times K(\nu,t,v,v_\star,\sigma)B(v-v_\star)\mu^N(dv)\mu^N(dv_\star)d\sigma.\end{split}\end{equation}  We now consider $Y^\nu_t:=Z^\nu_t (F(\nu,\mu^N_t,w^N_t)-\int_0^t A_s ds)$. The changes at jumps are given by \begin{equation} Y^\nu_t-Y^\nu_{t-}=Z^\nu_{t-}(K(\nu,t,v,v_\star,\sigma)F(\nu,\mu^{N,v,v_\star,\sigma}_{t-},w^{N,t,v,v_\star,\sigma}_{t-})-F(\nu, \mu^N_{t-},w^N_{t-}))\end{equation} while the drift between jumps is \begin{equation}\begin{split} \frac{d}{dt}Y^\nu_t=Z^\nu_{t-}\left(F(\nu, \mu^N_{t-},w^N_{t-})\int_{\rrd\times\rrd\times\ssd} K(\nu,t,v,v_\star,\sigma)B(v-v_\star)\mu^N_t(dv)\mu^N_t(dv_\star)d\sigma-A_t\right) \end{split} \end{equation} and together we conclude that \begin{equation}\begin{split} Y^\nu_t-Y^\nu_0 &= \int_E 1_{s\le t} (F(\nu,\mu^N_{s-}+\Delta(v,v_\star,\sigma),w^N_{s-}+N^{-1}\delta_{(s,v,v_\star,\sigma)})-F(\nu,\mu^N_{s-},w^N_{s-}))\\ & \hs \hs \hs \hs \hs \hs  ....K(\nu,t,v,v_\star,\sigma)(w^N-\overline{m}_{\mu^N})(ds,dv,dv_\star,d\sigma)\end{split}\end{equation} which is again a $\PP^\nu$-martingale, using almost sure bound on the supports of $\overline{m}_\mu, w^N$ under $\PP_\nu$ as commented above. It follows that $F(\nu,\mu^N_t,w^N_t)-\int_0^t A_s ds$ is a $\PPm_\nu$-martingale, and we conclude that $(\mu^N_0,\mu^N_t,w^N_t)$ is a $(\Omega,\mathfrak{F},(\mathfrak{F}_t)_{t\ge 0},\PPm_\nu)$-Markov process with time-dependent generator \begin{equation}\begin{split} \label{eq: time dependent generator'}\mathcal{G}_tF(\nu',\mu^N, w^N)=&N\int_{\rrd\times\rrd\times\ssd}(F(\nu',\mu^{N,v,v_\star,\sigma},w^{N,t,v,v_\star,\sigma})-F(\nu',\mu^N, w^N))\\ & \hs \hs \hs\hs \dots \times K(\nu',t,v,v_\star,\sigma)B(v-v_\star)\mu^N(dv)\mu^N(dv_\star)d\sigma. \end{split}\end{equation}  Using the same boundedness arguments as before, this generator characterises a unique semigroup $P^K_{s,t}$ of transition kernels on $\cp^N_2\times\cp^N_2\times\Mm$, so that for any $0=t_0< t_1.... <t_n$ and Borel sets $A_i\subset \cp^N_2\times\cp^N_2\times\Mm$, we have \begin{equation} \begin{split}&\PPm_\nu\left((\mu^N_0, \mu^N_{t_i},w^N_{t_i})\in A_i, i=0,...,n\right) \\& \hs\hs \hs = \int_{A_0\times...\times A_n} \delta_{(\nu,\nu,0)}(dx_0)P^K_{t_0,t_1}(x_0,dx_1)....P^K_{t_{n-1},t_n}(x_{n-1},dx_n).\end{split} \end{equation} Returning to (\ref{eq: disintegration}), we conclude \begin{equation}\begin{split} & \PPm\left((\mu^N_0, \mu^N_{t_i},w^N_{t_i})\in A_i, i=0,...,n\right) \\ & \hs= \int_{A_0\times...\times A_n} \delta_{(\nu,\nu,0)}(dx_0)P^K_{t_0,t_1}(x_0,dx_1)....P^K_{t_{n-1},t_n}(x_{n-1},dx_n)\widetilde{L}_0(d\nu) \end{split} \end{equation} which is exactly the statement that, under $\PPm$, $(\mu^N_0, \mu^N_t,w^N_t)$ is the Markov process with generator (\ref{eq: time dependent generator'}), and initial data $(\mu^N_0,\mu^N_0,0)$, with $\mu^N_0$ sampled from $\widetilde{L}_0$ as above. 
\end{appendix}
\section*{Acknowledgements}
I would like to thank Robert Patterson and Michel Renger, conversations with whom at various points sparked and renewed my interest in the topic, as well as Sergio Simonella for an interesting discussion of the problem. I would also like to 
thank my doctoral supervisor, Prof. James Norris, who pointed out ways in which the counterexample in Theorem \ref{thm: main} could be extended into its current form.


\begin{thebibliography}{9}

\bibitem{adams2011large} Adams, S., Dirr, N., Peletier, M.A. and Zimmer, J., 2011. From a large-deviations principle to the Wasserstein gradient flow: a new micro-macro passage. Communications in Mathematical Physics, 307(3), pp.791-815.

\bibitem{adams2013large} Adams, S., Dirr, N., Peletier, M. and Zimmer, J., 2013. Large deviations and gradient flows. Philosophical Transactions of the Royal Society A: Mathematical, Physical and Engineering Sciences, 371(2005), p.20120341.

\bibitem{banerjee2020new} Banerjee, S., Budhiraja, A. and Perlmutter, M., 2020. A new approach to large deviations for the Ginzburg-Landau model. Electronic Journal of Probability, 25.

\bibitem{basile2021large} Basile, G., Benedetto, D., Bertini, L. and Orrieri, C., 2021. Large deviations for Kac-like walks. Journal of Statistical Physics, 184(1), pp.1-27.

\bibitem{basile2021large'} Basile, G., Benedetto, D., Bertini, L. and Caglioti, E., 2021. Large deviations for a binary collision model: energy evaporation. arXiv preprint arXiv:2111.12439.

\bibitem{basile2022asymptotic} Basile, G., Benedetto, D., Bertini, L. and Caglioti, E., 2022. Asymptotic probability of energy increasing solutions to the homogeneous Boltzmann equation. arXiv preprint arXiv:2202.07311.



\bibitem{bobylev1999rate} Bobylev, A.V. and Cercignani, C., 1999. On the rate of entropy production for the Boltzmann equation. Journal of statistical physics, 94(3), pp.603-618.

\bibitem{bodineau2020fluctuation} Bodineau, T., Gallagher, I., Saint-Raymond, L. and Simonella, S., 2020. Fluctuation theory in the Boltzmann-Grad limit. Journal of Statistical Physics, 180(1), pp.873-895.


\bibitem{bonetto2014kac} Bonetto, F., Loss, M. and Vaidyanathan, R., 2014. The Kac model coupled to a thermostat. Journal of Statistical Physics, 156(4), pp.647-667.

\bibitem{bouchet2020boltzmann} Bouchet, F., 2020. Is the Boltzmann equation reversible? A large deviation perspective on the irreversibility paradox. Journal of Statistical Physics, 181(2), pp.515-550.

\bibitem{budhiraja2013large} Budhiraja, A., Chen, J. and Dupuis, P., 2013. Large deviations for stochastic partial differential equations driven by a Poisson random measure. Stochastic Processes and their Applications, 123(2), pp.523-560.

\bibitem{budhiraja2020large} Budhiraja, A., Chen, Y. and Xu, L., 2020. Large Deviations of the Entropy Production Rate for a Class of Gaussian Processes. arXiv e-prints, pp.arXiv-2004.

\bibitem{budhiraja2021empirical} Budhiraja, A. and Conroy, M., 2021. Empirical measure and small noise asymptotics under large deviation scaling for interacting diffusions. Journal of Theoretical Probability, pp.1-55.

\bibitem{cercignani1982h} Cercignani, C., 1982. H-theorem and trend to equilibrium in the kinetic theory of gases. Archiwum Mechaniki Stosowanej, 34(3), pp.231-241.

\bibitem{darling2008differential} Darling, R.W.R. and Norris, J.R., 2008. Differential equation approximations for Markov chains. Probability surveys, 5, pp.37-79.

\bibitem{delarue2018master} Delarue, F., Lacker, D. and Ramanan, K., 2020. From the master equation to mean field game limit theory: Large deviations and concentration of measure. The Annals of Probability, 48(1), pp.211-263.


\bibitem{desvillettes1993some} Desvillettes, L., 1993. Some applications of the method of moments for the homogeneous Boltzmann and Kac equations. Archive for rational mechanics and analysis, 123(4), pp.387-404.

\bibitem{desvillettes2010celebrating} Desvillettes, L., Mouhot, C. and Villani, C., 2010. Celebrating Cercignani's conjecture for the Boltzmann equation. arXiv preprint arXiv:1009.4006.

\bibitem{djehiche1998large} Djehiche, B. and Schied, A., 1998. Large deviations for hierarchical systems of interacting jump processes. Journal of Theoretical Probability, 11(1), pp.1-24.


\bibitem{duong2013wasserstein} Duong, M.H., Laschos, V. and Renger, M., 2013. Wasserstein gradient flows from large deviations of many-particle limits. ESAIM: Control, Optimisation and Calculus of Variations, 19(4), pp.1166-1188.

\bibitem{dupuis2011weak} Dupuis, P. and Ellis, R.S., 2011. A weak convergence approach to the theory of large deviations (Vol. 902). John Wiley \& Sons.

\bibitem{dupuis2016large} Dupuis, P., Ramanan, K. and Wu, W., 2016. Large deviation principle for finite-state mean field interacting particle systems. arXiv preprint arXiv:1601.06219.

\bibitem{erbar2015large} Erbar, M., Maas, J. and Renger, M., 2015. From large deviations to Wasserstein gradient flows in multiple dimensions. Electronic Communications in Probability, 20.

\bibitem{erbar2016gradient} Erbar, M., 2016. A gradient flow approach to the Boltzmann equation. arXiv preprint arXiv:1603.00540.

\bibitem{feng2006large} Feng, J. and Kurtz, T.G., 2006. Large deviations for stochastic processes (No. 131). American Mathematical Soc..


\bibitem{freidlin1998random} Freidlin, M.I. and Wentzell, A.D., 1998. Random perturbations. In Random perturbations of dynamical systems (pp. 15-43). Springer, New York, NY.

\bibitem{hauray2014kac} Hauray, M. and Mischler, S., 2014. On Kac's chaos and related problems. Journal of Functional Analysis, 266(10), pp.6055-6157.

\bibitem{heydecker2018pathwise} Heydecker, D., 2019. Pathwise convergence of the hard spheres Kac process. Annals of Applied Probability, 29(5), pp.3062-3127.


\bibitem{jordan1998variational} Jordan, R., Kinderlehrer, D. and Otto, F., 1998. The variational formulation of the Fokker--Planck equation. SIAM journal on mathematical analysis, 29(1), pp.1-17.

\bibitem{kac1956foundations} Kac, M., 1956, January. Foundations of kinetic theory. In Proceedings of The third Berkeley symposium on mathematical statistics and probability (Vol. 3, pp. 171-197).


\bibitem{kipnis1998scaling} Kipnis, C. and Landim, C., 1998. Scaling limits of interacting particle systems (Vol. 320). Springer Science \& Business Media.


\bibitem{kraaij2017flux} Kraaij, R.C., 2017. Flux large deviations of weakly interacting jump processes via well-posedness of an associated Hamilton-Jacobi equation. arXiv preprint arXiv:1711.00274.

\bibitem{leonard1995large} L\'eonard, C., 1995. On large deviations for particle systems associated with spatially homogeneous Boltzmann type equations. Probability theory and related fields, 101(1), pp.1-44.

\bibitem{lu1999conservation} Lu, X., 1999. Conservation of energy, entropy identity, and local stability for the spatially homogeneous Boltzmann equation. Journal of statistical physics, 96(3), pp.765-796.

\bibitem{lu1999solutions} Lu, X. and Wennberg, B., 1999. Solutions with increasing energy for the spatially homogeneous Boltzmann equation.

\bibitem{lu2012measure} Lu, X. and Mouhot, C., 2012. On measure solutions of the Boltzmann equation, part I: moment production and stability estimates. Journal of Differential Equations, 252(4), pp.3305-3363.

\bibitem{mielke2014relation} Mielke, A., Peletier, M.A. and Renger, D.M., 2014. On the relation between gradient flows and the large-deviation principle, with applications to Markov chains and diffusion. Potential Analysis, 41(4), pp.1293-1327.

\bibitem{mischler1999spatially} Mischler, S. and Wennberg, B., 1999, July. On the spatially homogeneous Boltzmann equation. In Annales de l'Institut Henri Poincare (C) Non Linear Analysis (Vol. 16, No. 4, pp. 467-501). Elsevier Masson.

\bibitem{mischler2013kac} Mischler, S. and Mouhot, C., 2013. Kac's program in kinetic theory. Inventiones mathematicae, 193(1), pp.1-147.


\bibitem{morters2010introduction} M\"orters, P., 2010. Introduction to large deviations. October 19th.

\bibitem{nanbu1983interrelations} Nanbu, K., 1983. Interrelations between various direct simulation methods for solving the Boltzmann equation. Journal of the Physical Society of Japan, 52(10), pp.3382-3388.

\bibitem{nguyen2021large} Nguyen, N.N. and Yin, G., 2021. Large Deviations Principles for Langevin Equations in Random Environment and Applications. arXiv preprint arXiv:2101.07133.

\bibitem{norris1999smoluchowski} Norris, J.R., 1999. Smoluchowski's coagulation equation: Uniqueness, nonuniqueness and a hydrodynamic limit for the stochastic coalescent. Annals of Applied Probability, pp.78-109.

\bibitem{norris2016consistency} Norris, J., 2016. A consistency estimate for Kac's model of elastic collisions in a dilute gas. Annals of Applied Probability, 26(2), pp.1029-1081.


\bibitem{patterson2016dynamical} Patterson, R.I. and Renger, D.R., 2016. Dynamical large deviations of countable reaction networks under a weak reversibility condition.


\bibitem{patterson2018large} Patterson, R. and Renger, M., 2018. Large deviations of reaction fluxes. arXiv preprint arXiv:1802.02512.

\bibitem{renger2018flux} Renger, D.M., 2018. Flux large deviations of independent and reacting particle systems, with implications for macroscopic fluctuation theory. Journal of Statistical Physics, 172(5), pp.1291-1326.

\bibitem{rezakhanlou1998large} Rezakhanlou, F., 1998. Large deviations from a kinetic limit. Annals of probability, pp.1259-1340.

\bibitem{sznitman1984equations} Sznitman, A.S., 1984. \'Equations de type de Boltzmann, spatialement homogenes. Zeitschrift f\"{u}r Wahrscheinlichkeitstheorie und verwandte Gebiete, 66(4), pp.559-592.

\bibitem{sznitman1991topic} Sznitman, A.S., 1991. Topics in propagation of chaos. In Ecole d'\'et\'e de probabilit\'es de Saint-Flour XIX-1989 (pp. 165-251). Springer, Berlin, Heidelberg.


\bibitem{toscani1999sharp} Toscani, G. and Villani, C., 1999. Sharp entropy dissipation bounds and explicit rate of trend to equilibrium for the spatially homogeneous Boltzmann equation. Communications in mathematical physics, 203(3), pp.667-706.

\bibitem{tossounian2015partially} Tossounian, H. and Vaidyanathan, R., 2015. Partially thermostated Kac model. Journal of Mathematical Physics, 56(8), p.083301.

\bibitem{villani1999trend} Villani, C., 1999. On the trend to equilibrium for solutions of the Boltzmann equation: quantitative versions of Boltzmann's H-theorem. Unpublished review paper.

\bibitem{villani2003cercignani} Villani, C., 2003. Cercignani's conjecture is sometimes true and always almost true. Communications in mathematical physics, 234(3), pp.455-490.

\bibitem{villani2008h} Villani, C., 2008. H-Theorem and beyond: Boltzmann's entropy in today's mathematics (pp. 129-145). EMS Publishing House: Z\"urich, Switzerland.

\bibitem{wennberg1997entropy} Wennberg, B., 1997. Entropy dissipation and moment production for the Boltzmann equation. Journal of Statistical Physics, 86(5), pp.1053-1066.

\bibitem{zeitouni1998large} Schmock, U., 2000. Large deviations techniques and applications. Journal of the American Statistical Association, 95(452), pp.1380-1380.





	
\end{thebibliography}
\end{document}